\documentclass[11pt,english]{amsart}

\usepackage{amsmath,amsfonts,amssymb,amsthm,amscd}
\usepackage{dsfont, mathrsfs,enumerate,eucal,fontenc}
\usepackage[varg]{txfonts}
\usepackage{xspace}
\usepackage[english]{babel}

\usepackage[all]{xy, xypic}
\usepackage{color} 
\usepackage{graphicx}
\usepackage{psfrag}  
\usepackage{pst-all}
\usepackage{youngtab}

\usepackage{xcolor}
\definecolor{rouge}{rgb}{0.85,0.1,.4}
\definecolor{bleu}{rgb}{0.1,0.2,0.9}
\definecolor{violet}{rgb}{0.7,0,0.8}
\usepackage[colorlinks=true,linkcolor=bleu,urlcolor=violet,citecolor=rouge]{hyperref}
\usepackage{breakurl}
\usepackage{xspace}
\DeclareMathAlphabet{\mathpzc}{OT1}{pzc}{m}{it}
\usepackage{fullpage}

\theoremstyle{plain}
\newtheorem{theorem}{Theorem}[section]
\newtheorem{lemma}[theorem]{Lemma}
\newtheorem{theo}[theorem]{Theorem}

\newtheorem{coro}[theorem]{Corollary}
\newtheorem{prop}[theorem]{Proposition}

\theoremstyle{definition}

\theoremstyle{remark}
\newtheorem{rema}[theorem]{Remark}
\newtheorem{claim}[theorem]{Claim}

\def\g{{\mathfrak{g}}}

\def\k{{\Bbbk}}
\def\x{{\mathrm {x}}}
\def\rg{\ell}               
\def\r{{\rm reg}}
\def\rs{{\rm reg,ss}}          

\def\mycom#1#2{\genfrac{}{}{0pt}{}{#1}{#2}}

\def\poie#1#2#3#4#5#6#7#8#9{\def\un{#5#6#7#8#9}\def\deux{#6#7#8#9}\def\trois{#2#4#8#9}
\def\quatre{#8#9}\def\cinq{#5#6#7}\def\six{#6#7}\def\sept{#2#4}
\ifx\un\empty {#1}_{#2}{#3 \hskip 0.15em}{#1}_{#4} \else \ifx\deux\empty 
{#5}(#1_{#2}){#3 \hskip 0.15em}{#5}(#1_{#4})
\else \ifx\trois\empty {#5}_{#6}(#1){#3 \hskip 0.15em}{#5}_{#7}(#1) 
\else \ifx\quatre\empty {#5}_{#6}(#1_#2){#3 \hskip 0.15em}{#5}_{#7}(#1_#4) 
\else \ifx\cinq\empty {#1}_{#2}^{#8}{#3 \hskip 0.15em}#1_#4^{#9} 
\else \ifx\six\empty {#5}(#1_{#2}^{#8}){#3 \hskip 0.15em}{#5}(#1_{#4}^{#9}) 
\else \ifx\sept\empty {#5}_{#6}^{#8}(#1){#3 \hskip 0.15em}{#5}_{#7}^{#9}(#1) \else
{#5}_{#6}(#1_{#2}^{#8})^{#9}{#3 \hskip 0.15em}{#5}_{#7}(#1_{#4}^{#8})^{#9} 
\fi \fi \fi \fi \fi \fi \fi}
\def\poi#1#2#3#4#5#6#7{\poie {#1}{#2}{#3}{#4}{#5}{#6}{#7}{}{}}
\def\rond{\raisebox{.3mm}{\scriptsize$\circ$}}

\def\tens{\raisebox{.3mm}{\scriptsize$\otimes$}}

\def\dv#1#2{\langle {#1},{#2}\rangle}

\def\tk#1#2{{#2}\otimes _{#1}}

\def\ec#1#2#3#4#5{\def\un{#3#4#5}\def\deux{#3#5}\def\trois{#3}
\def\four{#2#4#5}\def\five{#2#5}\def\six{#2}\def\seven{#3#4}
\def\eight{#2#4} \def\nine{#2#3#4}
\ifx\nine\empty {\rm #1}_{#5} \else
\ifx\un\empty {\rm #1}({\goth #2}) \else
\ifx\deux\empty {\rm #1}({\goth #2}_{#4}) \else
\ifx\trois\empty {\rm #1}_{#5}({\goth #2}_{#4}) \else
\ifx\four\empty {\rm #1}(#3) \else
\ifx\five\empty {\rm #1}(#3_{#4}) \else
\ifx\six\empty {\rm #1}_{#5}(#3_{#4}) \else
\ifx\seven\empty {\rm #1}_{#5} ({\goth#2})\else
\ifx\eight\empty {\rm #1}_{#5}({#3})
\fi \fi \fi \fi \fi \fi \fi \fi \fi}
\def\hec#1#2#3#4#5{\def\un{#3#4#5}\def\deux{#3#5}\def\trois{#3}
\def\four{#2#4#5}\def\five{#2#5}\def\six{#2}\def\seven{#3#4}
\def\eight{#2#4} \def\nine{#2#3#4}
\ifx\nine\empty \hat{{\rm #1}}_{#5} \else
\ifx\un\empty \hat{{\rm #1}}({\goth #2}) \else
\ifx\deux\empty \hat{{\rm #1}}({\goth #2}_{#4}) \else
\ifx\trois\empty \hat{{\rm #1}}_{#5}({\goth #2}_{#4}) \else
\ifx\four\empty \hat{{\rm #1}}(#3) \else
\ifx\five\empty \hat{{\rm #1}}(#3_{#4}) \else
\ifx\six\empty \hat{{\rm #1}}_{#5}(#3_{#4}) \else
\ifx\seven\empty \hat{{\rm #1}}_{#5} ({\goth#2})  \else
\ifx\eight\empty \hat{{\rm #1}}_{#5}({#3})
\fi \fi \fi \fi \fi \fi \fi \fi \fi}
\def\e#1#2{\ec {#1}#2{}{}{}}
\def\es#1#2{\ec {#1}{}{#2}{}{}}

\def\ai#1#2#3{\def\deux{#2#3} \def\trois{#3} \def\quatre{#2} 
\ifx\deux\empty \es S{{\goth #1}}^{{\goth #1}} \else
\ifx\trois\empty \es S{{\goth #1}^{#2}}^{{\goth #1}^{#2}} \else
\ifx\quatre\empty \es S{{\goth #1}_{#3}}^{{\goth #1}_{#3}} \else
\es S{{\goth #1}_{#3}^{#2}}^{{\goth #1}_{#3}^{#2}} \fi \fi \fi}

\def\Bbb{\mathbb}
\def\goth{\mathfrak}
\def\cal{\mathcal}

\def\gi#1#2#3#4{\def\trois{#3#4} \def\quatre{#4}\def\cinq{#3}\ifx\trois\empty 
{\rm i}_{#1,{\goth #2}}
\else \ifx\quatre\empty {\rm i}_{#1_{#3},{\goth #2}} \else\ifx\cinq\empty 
{\rm i}_{#1,{\goth #2}_{#4}} \else {\rm i}_{#1_{#3},{\goth #2}_{#4}} \fi \fi \fi}
\def\j#1#2{\def\deux{#2} \ifx\deux\empty {\rm rk}\hskip .125em{{\goth #1}} \else 
{\rm rk}\hskip .125em{{\goth #1}_{#2}} \fi}
\def\aj#1#2{\def\deux{#2} \ifx\deux\empty {\rm j}_{{\goth #1}} \else 
{\rm j}_{{\goth #1}_{#2}} \fi}
\def\an#1#2{\def\deux{#2} \ifx\deux\empty {\cal O}_{#1} \else {\cal O}_{#1,#2} \fi }
\def\han#1#2{\def\deux{#2} \ifx\deux\empty {\hat{{\cal O}}}_{#1} \else 
{\hat{{\cal O}}}_{#1,#2} \fi }

\def\gg#1#2{{\goth #1}_{#2}\times {\goth #1}_{#2}}

\def\sgg#1#2{\es S{\gg {#1}{#2}}}

\def\dim{{\rm dim}\hskip .125em}
\def\dd{{\rm d}}
\def\ad{{\rm ad}\hskip .1em}

\def\j#1#2{{{\mathrm {rk}}}_{{\goth #1}_{#2}}}

\def\pr#1{{\rm pr}_{#1}}
\def\n{{\rm n}}
\def\s{{\rm s}}
\def\u{{\rm u}}

\def\b#1#2{{\mathrm {b}}_{{\mathfrak{#1}}_{#2}}}
\def\sqx#1#2{{#1}\times _{B}{#2}}
\def\sqxx#1#2{G^{#1}\times _{B^{#1}}{#2}}

\title
[Commuting variety]
{On the Commuting variety of a reductive Lie algebra and other related varieties.}

\author
[J-Y Charbonnel]{Jean-Yves Charbonnel}

\address{Jean-Yves Charbonnel, Universit\'e Paris Diderot - CNRS \\
Institut de Math\'ematiques de Jussieu - Paris Rive Gauche\\
UMR 7586 \\ Groupes, repr\'esentations et g\'eom\'etrie \\
B\^atiment Sophie Germain \\ Case 7012 \\ 
75205 Paris Cedex 13, France}
\email{jean-yves.charbonne@imj-prg.fr}

\author
[M. Zaiter]{Mouchira Zaiter}

\address{Mouchira Zaiter, Universit\'e Libanaise Al- Hadath\\
Facult\'e des sciences, branche I\\
Beyrouth Liban}

\email{zaiter.mouchira@hotmail.fr}



\date\today

\begin{document}

\large

\begin{abstract}
The nilpotent cone of a reductive Lie algebra has a desingularization given by the 
cotangent bundle of the flag variety. Analogously, the nullcone of a cartesian 
power of the algebra has a desingularization given by a vector bundle over the 
flag variety. As for the nullcone, the subvariety of elements whose components 
are in a same Borel subalgebra, has a desingularization given by a vector bundle over
the flag variety. In this note, we study geometrical properties of these varieties. For 
the study of the commuting variety, the analogous variety to the flag variety is 
the closure in the Grassmannian of the set of Cartan subalgebras. So some 
properties of this variety are given. In particular, it is smooth in codimension $1$.
We introduce the generalized isospectral commuting varieties and give some properties. 
Furthermore, desingularizations of these varieties are given by fiber bundles over a 
desingularization of the closure in the grassmannian of the set of Cartan subalgebras 
contained in a given Borel subalgebra.  
\end{abstract}

\maketitle

\setcounter{tocdepth}{1}
\tableofcontents

\section{Introduction} \label{int}
In this note, the base field $\k$ is algebraically closed of characteristic $0$, 
${\goth g}$ is a reductive Lie algebra of finite dimension, $\rg$ is its rank,
$\dim {\goth g}=\rg + 2n$ and $G$ is its adjoint group. As usual, ${\goth b}$ denotes a 
Borel subalgebra of ${\goth g}$, ${\goth h}$ a Cartan subalgebra of ${\goth g}$, 
contained in ${\goth b}$, and $B$ the normalizer of ${\goth b}$ in $G$.

\subsection{Main results.} \label{int1}
Let ${\cal B}^{(k)}$ be the subset of elements 
$(\poi x1{,\ldots,}{k}{}{}{})$ of ${\goth g}^{k}$ such that $\poi x1{,\ldots,}{k}{}{}{}$ 
are in a same Borel subalgebra of ${\goth g}$. This subset of ${\goth g}^{k}$ is closed 
and contains two interesting subsets: the nullcone of ${\goth g}^{k}$ denoted by 
${\cal N}^{(k)}$ and the generalized commuting variety of ${\goth g}$ that is the  
closure in ${\goth g}^{k}$ of the subset of elements $(\poi x1{,\ldots,}{k}{}{}{})$ such 
that $\poi x1{,\ldots,}{k}{}{}{}$ are in a same Cartan subalgebra of ${\goth g}$. We 
denote it by ${\cal C}^{(k)}$. According to~\cite[Ch.2, \S 1, Theorem]{Mu0}, for 
$(\poi x1{,\ldots,}{k}{}{}{})$ in ${\cal B}^{(k)}$, $(\poi x1{,\ldots,}{k}{}{}{})$ is in 
${\cal N}^{(k)}$ if and only if $\poi x1{,\ldots,}{k}{}{}{}$ are nilpotent. According to 
a Richardson Theorem \cite{Ric}, ${\cal C}^{(2)}$ is the commuting variety of 
${\goth g}$. 

There is a natural projective morphism 
$\xymatrix{\sqx G{{\goth b}^{k}}\ar[r] & {\cal B}^{(k)}}$.
For $k=1$, this morphism is not birational but for $k\geq 2$, it is birational (see
Lemma~\ref{lbo1} and Lemma~\ref{lbo2}). 
Furthermore, denoting by ${\cal X}$ the subvariety of elements $(x,y)$ of 
${\goth g}\times {\goth h}$ such that $y$ is in the closure of the orbit of $x$ under 
$G$, the morphism 
$$ \xymatrix{G\times {\goth b} \ar[r] & {\cal X}}, \qquad 
(g,x) \longmapsto (g.x,\overline{x})$$
with $\overline{x}$ the projection of $x$ onto ${\goth h}$ defines through the quotient
a projective and birational morphism $\xymatrix{\sqx G{{\goth b}}\ar[r]& {\cal X}}$
and ${\goth g}$ is the categorical quotient of ${\cal X}$ under the action
of $W({\cal R})$ on the factor ${\goth h}$, with $W({\cal R})$ the Weyl group of 
${\goth g}$. For $k\geq 2$, the inverse image of ${\cal B}^{(k)}$ by the canonical 
projection from 
${\cal X}^{k}$ to ${\goth g}^{k}$ is not irreducible but the canonical action of 
$W({\cal R})^{k}$ on ${\cal X}^{k}$ induces a simply transitive action on the set of its 
irreducible components. Setting:
$${\cal B}_{\x}^{(k)} := 
\{((g(x_{1}),\overline{x_{1}}),\ldots,(g(x_{k}),\overline{x_{k}}))
\; \vert \; (g,\poi x1{,\ldots,}{k}{}{}{}) \in G\times {\goth b}^{k} \},$$ 
${\cal B}_{\x}^{(k)}$ is an irreducible component of the inverse image of ${\cal B}^{(k)}$
in ${\cal X}^{k}$ (see Corollary~\ref{cbo3}) and we have a commutative diagram
$$ \xymatrix{{\sqx G{{\goth b}^{k}}} \ar[rr] \ar[rd]_{\gamma} && 
{\cal B}_{\x}^{(k)} \ar[ld]^{\eta }\\  & {\cal B}^{(k)} &  }  $$
with $\eta $ the restriction to ${\cal B}_{\x}^{(k)}$ of the canonical projection
$\varpi $ from ${\cal X}^{k}$ to ${\goth g}^{k}$. The first main theorem of this note is
the following theorem:

\begin{theo}\label{tint}
{\rm (i)} The variety ${\cal N}^{(k)}$ is normal if and only if so is
${\cal B}_{\x}^{(k)}$.

{\rm (ii)} The variety ${\cal N}^{(k)}$ is Cohen-Macaulay if and only if so is
${\cal B}_{\x}^{(k)}$.

{\rm (iii)} The variety ${\cal N}^{(k)}$ has rational singularities if and only if 
it is Cohen-Macaulay.

{\rm (iv)} The variety ${\cal B}_{\x}^{(k)}$ has rational singularities if and only if 
it is Cohen-Macaulay.
 
{\rm (v)} The algebra $\k[{\cal B}_{\x}^{(k)}]$ is a free extension of 
$\k[{\cal B}_{\x}^{(k)}]^{G}$ which identifies with $\es S{{\goth h}^{k}}$.

{\rm (vi)} The algebra $\k[{\cal B}^{(k)}]^{G}$ identifies with
$\es S{{\goth h}^{k}}^{W({\cal R})}$ with respect to the diagonal action of $W({\cal R})$
in ${\goth h}^{k}$.

{\rm (vii)} The ideal $\k[{\cal B}^{(k)}]\k[{\cal B}^{(k)}]^{G}_{+}$ of 
$\k[{\cal B}^{(k)}]$ is strictly contained in the ideal of definition of ${\cal N}^{(k)}$
in $\k[{\cal B}^{(k)}]$.
\end{theo}

According to K. Vilonen and T. Xue~\cite{VX}, ${\cal N}^{(k)}$ and 
${\cal B}_{\x}^{(k)}$ are not normal in general. In the study of the generalized 
commuting variety, the closure in $\ec {Gr}g{}{}{\rg}$ of the orbit of 
${\goth h}$ under the action of $G$ plays an analogous role to the flag variety. Denoting
by $X$ the closure in $\ec {Gr}b{}{}{\rg}$ of the orbit of ${\goth h}$ under $B$, $G.X$ 
is the closure of the orbit of $G.{\goth h}$ and we have the following second main result:

\begin{theo}\label{t2int}
Let $X'$ be the set of centralizers of regular elements of ${\goth b}$ whose semisimple
components is regular or subregular. 

{\rm (i)} All element of $X$ is a commutative algebraic subalgebra of ${\goth g}$.

{\rm (ii)} For $x$ in ${\goth g}$, the set of elements of $G.X$ containing $x$ has 
dimension at most $\dim {\goth g}^{x}-\rg$.

{\rm (iii)} The sets $X\setminus B.{\goth h}$ and $G.X\setminus G.{\goth h}$ are 
equidimensional of dimension $n-1$ and $2n-1$ respectively.

{\rm (iv)} The sets $X'$ and $G.X'$ are smooth big open subsets of $X$ and $G.X$
respectively.
\end{theo}

This is a main result with respect to the generalized commuting varieties as it will be 
shown in the next two notes. We recall that an element of ${\goth g}$ is subregular if 
its centralizer in ${\goth g}$ has dimension $\rg +2$. Let ${\goth X}_{0,k}$ be the 
closure in ${\goth b}^{k}$ of $B.{\goth h}^{k}$ and let $\Gamma $ be a desingularization 
of $X$ in the category of $B$-varieties. Let ${\cal E}_{0}$ be the tautological bundle 
over $X$ and set:
$$ {\cal E}_{\s} := {\cal E}_{0}\times _{X}\Gamma , \qquad {\cal E}_{\s}^{(k)} := 
\underbrace{ {\cal E}_{\s} \times _{\Gamma } \cdots \times _{\Gamma } {\cal E}_{\s}}
_{k \hbox{ \footnotesize{factors}}} .$$
Then ${\cal E}_{\s}^{(k)}$ is a desingularization of ${\goth X}_{0,k}$. 
Set: ${\cal C}_{\x}^{(k)}:=\eta ^{-1}({\cal C}^{(k)})$. The following theorem is the 
third main result of this note:

\begin{theo}\label{t3int}
The variety ${\cal C}_{\x}^{(k)}$ is irreducible and $\sqx G{{\cal E}_{\s}^{(k)}}$ is
a desingularization of ${\cal C}_{\x}^{(k)}$.
\end{theo}

It will be proved in a next note that the normalizations of ${\goth X}_{0,k}$, 
${\cal C}^{(k)}$ and ${\cal C}_{\x}^{(k)}$ are Gorenstein with rational singularities.
As a matter of fact, as a consequence, ${\goth X}_{0,k}$ is normal.

{\bf Acknowledgments} We are grateful to K. Vilonen and T. Xue for pointing out 
a negative result about the nullcone in a previous version. We are also grateful to
the referee for his remarks and suggestions.

\subsection{Notations}\label{int2}
$\bullet$ An algebraic variety is a reduced scheme over $\k$ of finite type.

$\bullet$ For $V$ a vector space, its dual is denoted by $V^{*}$ and the augmentation 
ideal of its symmetric algebra $\es SV$ is denoted by $\ec S{}V{}+$. For $A$ a graded
algebra over ${\Bbb N}$, $A_{+}$ is the ideal generated by the homogeneous elements of
positive degree.

$\bullet$
All topological terms refer to the Zariski topology. If $Y$ is a subset of a topological
space $X$, denote by $\overline{Y}$ the closure of $Y$ in $X$. For $Y$ an open subset
of the algebraic variety $X$, $Y$ is called {\it a big open subset} if the codimension
of $X\setminus Y$ in $X$ is at least $2$. For $Y$ a closed subset of an algebraic 
variety $X$, its dimension is the biggest dimension of its irreducible components and its
codimension in $X$ is the smallest codimension in $X$ of its irreducible components. For 
$X$ an algebraic variety, $\an X{}$ is its structural sheaf, $\k[X]$ is the algebra of 
regular functions on $X$ and $\k(X)$ is the field of rational functions on $X$ when $X$ 
is irreducible. 

$\bullet$
For $X$ an algebraic variety and for ${\cal M}$ a sheaf on $X$, $\Gamma (V,{\cal M})$
is the space of local sections of ${\cal M}$ over the open subset $V$ of $X$. For 
$i$ a nonnegative integer, ${\mathrm {H}}^{i}(X,{\cal M})$ is the $i$-th group of 
cohomology of ${\cal M}$. For example, 
${\mathrm {H}}^{0}(X,{\cal M})=\Gamma (X,{\cal M})$.

\begin{lemma}\label{lint}~\cite[Corollaire 5.4.3]{Gro}
Let $X$ be an irreducible affine algebraic variety and let $Y$ be a desingularization of 
$X$. Then ${\mathrm {H}}^{0}(Y,\an Y{})$ is the integral closure of $\k[X]$ in its 
fraction field. 
\end{lemma}

$\bullet$ For $K$ a group and for $E$ a set with a group action of $K$, $E^{K}$ is the
set of invariant elements of $E$ under $K$. The following lemma is straightforward
and will be used in the proof of Corollary~\ref{c4bo5}.

\begin{lemma}\label{l2int}
Let $A$ be an algebra generated by the subalgebras $A_{1}$ and $A_{2}$. Let $K$ be a group
acting on $A_{2}$. Suppose that the following conditions are 
verified:
\begin{itemize}
\item [{\rm (1)}] $A_{1}\cap A_{2}$ is contained in $A_{2}^{K}$,  
\item [{\rm (2)}] $A$ is a free $A_{2}$-module having a basis contained in $A_{1}$,
\item [{\rm (3)}] $A_{1}$ is a free $A_{1}\cap A_{2}$-module having the same basis.
\end{itemize}
Then there exists a unique group action of $K$ on the algebra $A$ extending 
the action of $K$ on $A_{2}$ and fixing all the elements of $A_{1}$. Moreover, if
$A_{1}\cap A_{2}=A_{2}^{K}$ then $A^{K} = A_{1}$. 
\end{lemma}

$\bullet$
For $E$ a finite set, its cardinality is denoted by $\vert E \vert$. For $E$ a vector 
space and for $x=(\poi x1{,\ldots,}{k}{}{}{})$ in $E^{k}$, $E_{x}$ is the subspace of $E$
generated by $\poi x1{,\ldots,}{k}{}{}{}$. Moreover, there is a canonical action of 
${\mathrm {GL}}_{k}(\k)$ in $E^{k}$ given by:
$$ (a_{i,j},1\leq i,j\leq k).(\poi x1{,\ldots,}{k}{}{}{}) := 
(\sum_{j=1}^{k} a_{i,j}x_{j},i=1,\ldots,k)$$ 
In particular, the diagonal action of $G$ in ${\goth g}^{k}$ commutes with the 
action of ${\mathrm {GL}}_{k}(\k)$.

$\bullet$ For ${\goth a}$ reductive Lie algebra, its rank is denoted by $\j a{}$ and the
dimension of its Borel subalgebras is denoted by $\b a{}$. In particular, 
$\dim {\goth a}=2\b a{} - \j a{}$.

$\bullet$
If $E$ is a subset of a vector space $V$, denote by span($E$) the vector subspace of
$V$ generated by $E$. The grassmanian of all $d$-dimensional subspaces of $V$ is denoted
by Gr$_d(V)$. By definition, a {\it cone} of $V$ is a subset of $V$ invariant under the 
natural action of $\k^{*}:=\k\setminus \{0\}$ and a \emph{multicone} of $V^{k}$ is a
subset of $V^{k}$ invariant under the natural action of $(\k^{*})^{k}$ on $V^{k}$.

\begin{lemma}\label{l3int}
Let $X$ be an open cone of $V$ and let $S$ be a closed multicone of $X\times V^{k-1}$. 
Denoting by $S'$ the image of $S$ by the first projection, 
$S'\times \{0\}=S\cap (X\times \{0\})$. In particular, $S'$ is closed in $X$. 
\end{lemma}

\begin{proof}
For $x$ in $X$, $x$ is in $S'$ if and only if for some $(\poi v2{,\ldots,}{k}{}{}{})$ 
in $V^{k-1}$, $(x,\poi {tv}2{,\ldots,}{k}{}{}{})$ is in $S$ for all $t$ in $\k$ since $S$
is a closed multicone of $X\times V^{k-1}$, whence the lemma.
\end{proof}

$\bullet$
The dual ${\goth g}^{*}$ of ${\goth g}$ identifies with ${\goth g}$ by a given non 
degenerate, invariant, symmetric bilinear form $\dv ..$ on $\gg g{}$
extending the Killing form of $[{\goth g},{\goth g}]$. 

$\bullet$
Let ${\cal R}$ be the root system of ${\goth h}$ in ${\goth g}$ and ${\cal R}_{+}$ the 
positive root system of ${\cal R}$ defined by ${\goth b}$. The Weyl group of ${\cal R}$ 
is denoted by $W({\cal R})$ and the basis of ${\cal R}_{+}$ is denoted by $\Pi $. The 
neutral elements of $G$ and $W({\cal R})$ are denoted by $1_{{\goth g}}$ and 
$1_{{\goth h}}$ respectively. For $\alpha $ in ${\cal R}$, the corresponding root 
subspace is denoted by ${\goth g}^{\alpha }$ and a generator $x_{\alpha }$ of 
${\goth g}^{\alpha }$ is chosen so that $\dv {x_{\alpha }}{x_{-\alpha }} = 1$ for all 
$\alpha $ in ${\cal R}$. 

$\bullet$ The normalizers of ${\goth b}$ and ${\goth h}$ in $G$ are denoted by $B$ and
$N_{G}({\goth h})$ respectively. For $x$ in ${\goth b}$, $\overline{x}$ is the element
of ${\goth h}$ such that $x-\overline{x}$ is in the nilpotent radical ${\goth u}$ of 
${\goth b}$.

$\bullet$ 
For $X$ an algebraic $B$-variety, denote by $\sqx GX$ the quotient of 
$G\times X$ under the right action of $B$ given by $(g,x).b := (gb,b^{-1}.x)$. More 
generally, for $k$ positive integer and for $X$ an algebraic $B^{k}$-variety, denote 
by $\sqxx kX$ the quotient of $G^{k}\times X$ under the right action of $B^{k}$ given by 
$(g,x).b := (gb,b^{-1}.x)$ with $g$ and $b$ in $G^{k}$ and $B^{k}$ respectively. 

\begin{lemma}\label{l4int}
Let $P$ and $Q$ be parabolic subgroups of $G$ such that $P$ is contained in $Q$. Let 
$X$ be a $Q$-variety and let $Y$ be a closed subset of $X$, invariant under $P$. Then 
$Q.Y$ is a closed subset of $X$. Moreover, the canonical map from 
$Q\times _{P}Y$ to $Q.Y$ is a projective morphism.
\end{lemma}

\begin{proof}
Since $P$ and $Q$ are parabolic subgroups of $G$ and since $P$ is contained in $Q$, 
$Q/P$ is a projective variety. Denote by $Q\times _{P}X$ and $Q\times _{P}Y$ the 
quotients of $Q\times X$ and $Q\times Y$ under the right action of $P$ given by 
$(g,x).p := (gp,p^{-1}.x)$. Let $g\mapsto \overline{g}$ be the quotient map from
$Q$ to $Q/P$. Since $X$ is a $Q$-variety, the map 
$$\xymatrix{Q\times X \ar[rr] && Q/P \times X}, \qquad 
(g,x) \longmapsto (\overline{g},g.x) $$
defines through the quotient an isomorphism from $Q\times _{P}X$ to $Q/P\times X$. 
Since $Y$ is a $P$-invariant closed subset of $X$, $Q\times _{P}Y$ is a closed subset
of $Q\times _{P}X$ and its image by the above isomorphism is closed. Hence
$Q.Y$ is a closed subset of $X$ since $Q/P$ is a projective variety. From the commutative
diagram:
$$\xymatrix{ Q\times _{P}Y \ar[r] \ar[rd] & Q/P\times Q.Y \ar[d] \\ & Q.Y } $$
we deduce that the map $Q\times _{P}Y\rightarrow Q.Y$ is a projective morphism.
\end{proof}

$\bullet$
For $k\geq 1$ and for the diagonal action of $B$ in ${\goth b}^{k}$, ${\goth b}^{k}$ is a
$B$-variety. The image of $(g,\poi x1{,\ldots,}{k}{}{}{})$ in $G\times {\goth b}^{k}$ in 
$\sqx G{{\goth b}^{k}}$ is denoted by $\overline{(g,\poi x1{,\ldots,}{k}{}{}{})}$. The 
sets ${\cal B}^{(k)}$ and ${\cal N}^{(k)}$ are the images of $G\times {\goth b}^{k}$ and 
$G\times {\goth u}^{k}$ respectively by the map 
$(g,\poi x1{,\ldots,}{k}{}{}{})\mapsto (\poi x1{,\ldots,}{k}{g}{}{})$ so that 
${\cal B}^{(k)}$ and ${\cal N}^{(k)}$ are closed subsets of ${\goth g}^{k}$ 
by Lemma~\ref{l4int}. Let ${\cal B}_{\n}^{(k)}$ be the normalization of 
${\cal B}^{(k)}$ and let $\eta _{\n}$ be the normalization morphism. The map
$$\xymatrix{ G\times {\goth b}^{k} \ar[r] & {\cal B}^{(k)}}, \qquad
(g,\poi x1{,\ldots,}{k}{}{}{}) \longrightarrow (\poi x1{,\ldots,}{k}{g}{}{})$$ 
defines through the quotient a morphism 
$\gamma : \xymatrix{\sqx G{{\goth b}^{k}} \ar[r] & {\cal B}^{(k)}}$ and we have the 
commutative diagram:
$$ \xymatrix{{\sqx G{{\goth b}^{k}}} \ar[rr]^{\gamma _{\n}} \ar[rd]_{\gamma} && 
{\cal B}_{\n}^{(k)} \ar[ld]^{\eta _{\n}}\\  & {\cal B}^{(k)} &  }  $$
where $\gamma _{\n}$ is uniquely defined by this diagram.
Let ${\cal N}_{\n}^{(k)}$ be the normalization of ${\cal N}^{(k)}$ and let $\varkappa $ 
be the normalization morphism. We have the commutative diagram:
$$ \xymatrix{\sqx G{{\goth u}^{k}} \ar[rd]_{\upsilon } \ar[rr]^{\upsilon _{\n}} && 
{\cal N}_{\n}^{(k)} \ar[ld]^{\varkappa }\\  & {\cal N}^{(k)} & }  $$
with $\upsilon $ the restriction of $\gamma $ to $\sqx G{{\goth u}^{k}}$ and 
$\upsilon _{\n}$ is uniquely defined by this diagram.

$\bullet$ Let $i$ be the injection $(\poi x1{,\ldots,}{k}{}{}{})\mapsto 
\overline{(1_{{\goth g}},\poi x1{,\ldots,}{k}{}{}{})}$ from ${\goth b}^{k}$ to
$\sqx G{{\goth b}^{k}}$. Then $\iota := \gamma \rond i$ is the identity of 
${\goth b}^{k}$ and $\iota _{\n} := \gamma _{\n}\rond i$
is a closed embedding of ${\goth b}^{k}$ into ${\cal B}_{\n}^{(k)}$. In particular, 
${\cal B}^{(k)} = G.\iota ({\goth b}^{k})$ and 
${\cal B}_{\n}^{(k)} = G.\iota _{\n}({\goth b}^{k})$.

$\bullet$
Let $e$ be the sum of the $x_{\beta }$'s, $\beta $ in $\Pi $, and let $h$ be the 
element of ${\goth h}\cap [{\goth g},{\goth g}]$ such that $\beta (h)=2$ for all $\beta $
in $\Pi $. Then there exists a unique $f$ in $[{\goth g},{\goth g}]$ such that $(e,h,f)$ 
is a principal ${\goth {sl}}_2$-triple. The one-parameter subgroup of $G$ generated by 
$\ad h$ is denoted by $t\mapsto h(t)$. The Borel subalgebra containing $f$ is denoted by 
${\goth b}_{-}$ and its nilpotent radical is denoted by ${\goth u}_{-}$. Let $B_{-}$ be 
the normalizer of ${\goth b}_{-}$ in $G$ and let $U$ and $U_{-}$ be the unipotent 
radicals of $B$ and $B_{-}$ respectively.

\begin{lemma}\label{l5int}
Let $k\geq 2$ be an integer. Let $X$ be an affine variety and set 
$Y:= {\goth b}^{k}\times X$. Let $Z$ be a closed $B$-invariant subset of $Y$ under the 
group action given by 
$g.(\poi v1{,\ldots,}{k}{}{}{},x)=(\poi v1{,\ldots,}{k}{g}{}{},x)$ with 
$(g,\poi v1{,\ldots,}{k}{}{}{})$ in $B\times {\goth b}^{k}$ and $x$ in $X$. 
Then $Z\cap {\goth h}^{k}\times X$ is the image of $Z$ by the projection 
$(\poi v1{,\ldots,}{k}{}{}{},x)\mapsto (\overline{v_{1}},\ldots,\overline{v_{k}},x)$.  
\end{lemma}

\begin{proof}
For all $v$ in ${\goth b}$, 
$$ \overline{v} = \lim _{t\rightarrow 0} h(t)(v)$$
whence the lemma since $Z$ is closed and $B$-invariant.
\end{proof}

$\bullet$ 
For $x \in \g$, let $x_{\s}$ and $x_{\n}$ be the semisimple and nilpotent components of 
$x$ in ${\goth g}$. Denote by ${\goth g}^x$ and $G^{x}$ the centralizers of $x$ in 
${\goth g}$ and $G$ respectively. For ${\goth a}$ a subalgebra of ${\goth g}$ and for $A$
a subgroup of $G$, set:
$$\begin{array}{ccc}
{\goth a}^{x} := {\goth a}\cap {\goth g}^{x} && A^{x} := A \cap G^{x} \end{array}$$
The set of regular elements of $\g$ is 
$$\g_{\r} \ := \ \{ x\in \g \ \vert \ \dim \g^x=\rg \}$$
and denote by ${\goth g}_{\rs}$ the set of regular semisimple elements of
${\goth g}$. Both $\g_{\r}$ and $\g_{\rs}$ are $G$-invariant dense open subsets of
${\goth g}$. Setting ${\goth h}_{\r} := {\goth h}\cap {\goth g}_{\r}$, 
${\goth b}_{\r} := {\goth b}\cap {\goth g}_{\r}$, 
${\goth u}_{\r} := {\goth u}\cap {\goth g}_{\r}$, ${\goth g}_{\rs}=G({\goth h}_{\r})$,
${\goth g}_{\r}=G({\goth b}_{\r})$ and $G({\goth u}_{\r})$ is the set of regular 
elements of the nilpotent cone ${\goth N}_{{\goth g}}$ of ${\goth g}$. 

\begin{lemma}\label{l6int}
Let $k\geq 2$ be an integer and let $x$ be in ${\goth g}^{k}$. For $O$ open subset of 
${\goth g}_{\r}$, $E_{x}\cap O$ is not empty if and only if for some $g$ in 
${\mathrm {GL}}_{k}(\k)$, the first component of $g.x$ is in $O$.
\end{lemma}

\begin{proof}
Since the components of $g.x$ are in $E_{x}$ for all $g$ in ${\mathrm {GL}}_{k}(\k)$,
the condition is sufficient. Suppose that $E_{x}\cap O$ is not 
empty and denote by $\poi x1{,\ldots,}{k}{}{}{}$ the components of $x$. For 
some $(\poi a1{,\ldots,}{k}{}{}{})$ in $\k^{k}\setminus \{0\}$, 
$$ a_{1}x_{1}+\cdots +a_{k}x_{k} \in O$$
Let $i$ be such that $a_{i}\neq 0$ and let $\tau $ be the transposition such that 
$\tau (1)=i$. Denoting by $g$ the element of ${\mathrm {GL}}_{k}(\k)$ such that
$g_{1,j} = a_{\tau (j)}$ for $j=1,\ldots,k$, $g_{j,j}=1$ for $j=2,\ldots,k$ and 
$g_{j,l}=0$ for $j\geq 2$ and $j\neq l$, the first component of $g\tau .x$ is in $O$.
\end{proof}

$\bullet$
Denote by $\e Sg^{{\goth g}}$ the algebra of ${\goth g}$-invariant elements of 
$\e Sg$. Let $p_1,\ldots,p_{\rg}$ be homogeneous generators of $\e Sg^{\g}$ of degree
$\poi d1{,\ldots,}{\rg}{}{}{}$ respectively. Choose the polynomials
$\poi p1{,\ldots,}{\rg}{}{}{}$ so that $\poi d1{\leq\cdots \leq }{\rg}{}{}{}$. For 
$i=1,\ldots,\rg$ and $(x,y)\in\g \times \g$, consider a shift of $p_i$ in 
the direction $y$: $p_i(x+ty)$ with $t\in\k$. Expanding $p_i(x+ty)$ as a polynomial in 
$t$, we obtain
\begin{eqnarray}\label{eq:pi}
p_i(x+ty)=\sum\limits_{m=0}^{d_i} p_{i}^{(m)} (x,y) t^m;  && \forall
(t,x,y)\in\k\times\g\times\g
\end{eqnarray}
where $y \mapsto (m!)p_{i}^{(m)}(x,y)$ is the derivative at $x$ of $p_i$ at the order
$m$ in the direction $y$. The elements $p_{i}^{(m)}$ defined by~(\ref{eq:pi}) are
invariant elements of $\tk {\k}{\e Sg}\e Sg$ under the diagonal action of $G$ in
$\gg g{}$. Remark that $p_i^{(0)}(x,y)=p_i(x)$ while $p_i^{(d_i)}(x,y)=p_i(y)$  for 
all $(x,y)\in \g\times \g$.

\begin{rema}\label{rint3}
The family
$\mathcal{P}_x  :=
\{p_{i}^{(m)}(x,.); \ 1 \leq i \leq \rg, 1 \leq m \leq d_i  \}$ for $x\in\g$,
is a Poisson-commutative family of $\e Sg$ by Mishchenko-Fomenko~\cite{MF}.
We say that the family $\mathcal{P}_x$ is constructed by the \emph{argument shift method}.
\end{rema}

$\bullet$
Let $i \in\{1,\ldots,\rg\}$. For $x$ in $\g$, denote by $\varepsilon _i(x)$ the 
element of $\g$ given by
$$ \dv {\varepsilon _{i}(x)}y = \frac{\dd }{\dd t} p_{i}(x+ty) \left \vert _{t=0} \right.
$$
for all $y$ in ${\goth g}$. Thereby, $\varepsilon _{i}$ is an invariant element of 
$\tk {\k}{\e Sg}\g$ under the canonical action of $G$. According to 
\cite[Theorem 9]{Ko}, for $x$ in ${\goth g}$, $x$ is in ${\goth g}_{\r}$ if and only if 
$\poi x{}{,\ldots,}{}{\varepsilon }{1}{\rg}$ are linearly independent. In this case, 
$\poi x{}{,\ldots,}{}{\varepsilon }{1}{\rg}$ is a basis of ${\goth g}^{x}$.

Denote by ${\goth z}_{{\goth g}}$ the center of ${\goth g}$ and for $x$ in ${\goth g}$
by ${\goth z}_{x}$ the center of ${\goth g}^{x}$. As 
$\poi {\varepsilon }1{,\ldots,}{\rg}{}{}{}$ are invariant, for all $x$ in ${\goth g}$,
$\poi x{}{,\ldots,}{}{\varepsilon }{1}{\rg}$ are in ${\goth z}_{x}$.
 
$\bullet$
Denote by $\varepsilon _{i}^{(m)}$, for $0\leq m\leq d_i-1$, the elements of 
$\tk{\k}{\sgg g{}}{\goth g}$ defined by the equality:
\begin{eqnarray}\label{eq:phi}
\varepsilon _i (x+ty) =\sum\limits_{ m=0}^{d_i-1} \varepsilon _i^{(m)}(x,y) t^m , &&
\forall (t,x,y)\in\k\times\g\times\g
\end{eqnarray}
and set:
$$V_{x,y} := 
{\mathrm {span}}(\{\poie {x,y}{}{,\ldots,}{}{\varepsilon }{i}{i}{(0)}{(d_{i}-1)}, \ 
i =1,\ldots,\rg\}) $$
for $(x,y)$ in $\gg g{}$. According to~\cite[Corollary 2]{Bol}, $V_{x,y}$ has dimension 
$\b g{}$ if and only if $E_{x,y}$ has dimension $2$ and $E_{x,y}\setminus \{0\}$ is 
contained in ${\goth g}_{\r}$.

\section{On the varieties \texorpdfstring{${\cal B}^{(k)}$}{Lg}} \label{bo}
Let $k\geq 2$ be an integer. According to the above notations, we have the commutative
diagrams:
$$\begin{array}{ccc}
\xymatrix{{\sqx G{{\goth b}^{k}}} \ar[rr]^{\gamma _{\n}} \ar[rd]_{\gamma} && 
{\cal B}_{\n}^{(k)} \ar[ld]^{\eta _{\n}}\\  & {\cal B}^{(k)} &  }  &&
\xymatrix{\sqx G{{\goth u}^{k}} \ar[rd]_{\upsilon } \ar[rr]^{\upsilon _{\n}} && 
{\cal N}_{\n}^{(k)} \ar[ld]^{\varkappa }\\  & {\cal N}^{(k)} & } \end{array}$$
Since the Borel subalgebras of ${\goth g}$ are conjugate under $G$, ${\cal B}^{(k)}$
is the subset of elements of ${\goth g}^{k}$ whose components are in a same Borel
subalgebra and ${\cal N}^{(k)}$ are the elements of ${\cal B}^{(k)}$ whose all the
components are nilpotent. 

\begin{lemma}\label{lbo}
{\rm (i)} The morphism $\gamma $ from $\sqx G{{\goth b}^{k}}$ to ${\cal B}^{(k)}$ is 
projective and birational. In particular, $\sqx G{{\goth b}^{k}}$ is a desingularization
of ${\cal B}^{(k)}$ and ${\cal B}^{(k)}$ has dimension $k\b g{}+n$.

{\rm (ii)} The morphism $\upsilon $ from $\sqx G{{\goth u}^{k}}$ to ${\cal N}^{(k)}$ is 
projective and birational. In particular, $\sqx G{{\goth u}^{k}}$ is a desingularization
of ${\cal N}^{(k)}$ and ${\cal N}^{(k)}$ has dimension $(k+1)n$.
\end{lemma}

\begin{proof}
(i) Denote by $\Omega _{{\goth g}}^{(2)}$ the subset of elements $(x,y)$ of 
${\goth g}^{2}$ such that $E_{x,y}$ has dimension $2$ and such that 
$E_{x,y}\setminus \{0\}$ is contained in ${\goth g}_{\r}$. According to 
Lemma~\ref{l4int}, $\gamma $ is a projective morphism. For $1\leq i<j\leq k$, let 
$\Omega ^{(k)}_{i,j}$ be the inverse image of $\Omega _{{\goth g}}^{(2)}$ by the 
projection 
$$ (\poi x1{,\ldots,}{k}{}{}{}) \longmapsto (x_{i},x_{j}) $$ 
Then $\Omega ^{(k)}_{i,j}$ is an open subset of ${\goth g}^{k}$ whose intersection with
${\cal B}^{(k)}$ is not empty. Let $\Omega ^{(k)}_{{\goth g}}$ be the union of the 
$\Omega ^{(k)}_{i,j}$. According to~\cite[Corollary 2]{Bol} and~\cite[Theorem 9]{Ko}, for 
$(x,y)$ in $\Omega _{{\goth g}}^{(2)}\cap {\cal B}^{(2)}$, $V_{x,y}$ is the unique Borel 
subalgebra of ${\goth g}$ containing $x$ and $y$ so that the restriction of $\gamma $ to
$\gamma ^{-1}(\Omega ^{(k)}_{{\goth g}})$ is a bijection onto 
$\Omega _{{\goth g}}^{(k)}$. Hence $\gamma $ is birational. Moreover, 
$\sqx G{{\goth b}^{k}}$ is a smooth variety as a vector bundle over the smooth variety 
$G/B$, whence the assertion since $\sqx G{{\goth b}^{k}}$ has dimension $k\b g{}+n$. 

(ii) According to Lemma~\ref{l4int}, $\upsilon $ is a projective morphism. Let 
${\cal N}_{\r}^{(k)}$ be the subset of elements of ${\cal N}^{(k)}$ whose at 
least one component is a regular element of ${\goth g}$. Then ${\cal N}_{\r}^{(k)}$ is 
an open subset of ${\cal N}^{(k)}$. Since a regular nilpotent element is contained in 
one and only one Borel subalgebra of ${\goth g}$, the restriction of $\upsilon $ to 
$\upsilon ^{-1}({\cal N}_{\r}^{(k)})$ is a bijection onto ${\cal N}_{\r}^{(k)}$. Hence 
$\upsilon $ is birational. Moreover, $\sqx G{{\goth u}^{k}}$ is a smooth variety as a 
vector bundle over the smooth variety $G/B$, whence the assertion since 
$\sqx G{{\goth u}^{k}}$ has dimension $(k+1)n$. 
\end{proof}

\subsection{} \label{bo1}
Let $\kappa $ be the map
$$\begin{array}{ccc} U_{-}\times {\goth u}_{\r} \longrightarrow {\goth N}_{{\goth g}}
&& (g,x) \longmapsto g(x) \end{array}$$

\begin{lemma}\label{lbo1}
Let $V$ be the set of elements of ${\cal N}^{(k)}$ whose first component is in 
$U_{-}({\goth u}_{\r})$ and let $V_{k}$ be the set of elements $x$ of ${\cal N}^{(k)}$
such that $E_{x}\cap {\goth g}_{\r}$ is not empty.

{\rm (i)} The image of $\kappa $ is a smooth open subset of ${\goth N}_{{\goth g}}$ and 
$\kappa $ is an ismorphism onto $U_{-}({\goth u}_{\r})$.

{\rm (ii)} The subset $V$ of ${\cal N}^{(k)}$ is open.

{\rm (iii)} The open subset $V$ of ${\cal N}^{(k)}$ is smooth.

{\rm (iv)} The set $V_{k}$ is a smooth open subset of ${\cal N}^{(k)}$.
\end{lemma}

\begin{proof}
(i) Since ${\goth N}_{{\goth g}}$ is the nullvariety of $\poi p1{,\ldots,}{\rg}{}{}{}$ 
in ${\goth g}$, ${\goth N}_{{\goth g}}\cap {\goth g}_{\r}$ is a smooth open susbet 
of ${\goth N}_{{\goth g}}$ by~\cite[Theorem 9]{Ko}. For $(g,x)$ in 
$U_{-}\times {\goth u}_{\r}$ such that $g(x)$ is in ${\goth u}$, $b^{-1}g$ is in $G^{x}$
for some $b$ in $B$ since $B(x)={\goth u}_{\r}$. Hence $g=1_{{\goth g}}$ since $G^{x}$ is
contained in $B$ and since $U_{-}\cap B=\{1_{{\goth g}}\}$. As a result, $\kappa $ is an 
injective morphism from the smooth variety $U_{-}\times {\goth u}_{\r}$ to the smooth 
variety ${\goth N}_{{\goth g}}\cap {\goth g}_{\r}$. Hence $\kappa $ is an open 
immersion by Zariski's Main Theorem \cite[\S 9]{Mu}.

(ii) By definition, $V$ is the intersection of ${\cal N}^{(k)}$ and 
$U_{-}({\goth u}_{\r})\times {\goth N}_{{\goth g}}^{k-1}$. So, by (i), it is an 
open subset of ${\cal N}^{(k)}$.
 
(iii) Let $(\poi x1{,\ldots,}{k}{}{}{})$ be in ${\goth u}^{k}$ and let $g$ be in $G$ 
such that $(\poi x1{,\ldots,}{k}{g}{}{})$ is in $V$. Then $x_{1}$ is in ${\goth u}_{\r}$
and for some $(g',b)$ in $U_{-}\times B$, $g'b(x_{1})=g(x_{1})$. Hence $g^{-1}g'b$ is in 
$G^{x_{1}}$ and $g$ is in $U_{-}B$ since $G^{x_{1}}$ is contained in $B$. As a result, 
the map
$$\begin{array}{ccc}
U_{-}\times {\goth u}_{\r}\times {\goth u}^{k-1} \longrightarrow V &&
(g,\poi x1{,\ldots,}{k}{}{}{}) \longmapsto (\poi x1{,\ldots,}{k}{g}{}{})
\end{array}$$
is an isomorphism whose inverse is given by
$$\begin{array}{ccc}
V\longrightarrow U_{-}\times {\goth u}_{\r}\times {\goth u}^{k-1}  &&
(\poi x1{,\ldots,}{k}{}{}{}) \longmapsto (\kappa ^{-1}(x_{1})_{1},
\poi x1{,\ldots,}{k}{\kappa ^{-1}(x_{1})_{1}}{}{})\end{array}$$
with $\kappa ^{-1}$ the inverse of $\kappa $ and $\kappa ^{-1}(x_{1})_{1}$ the component
of $\kappa ^{-1}(x_{1})$ on $U_{-}$, whence the assertion since 
$U_{-}\times {\goth u}_{\r}\times {\goth u}^{k-1}$ is smooth.

(iv) According to Lemma~\ref{l6int}, $V_{k}={\mathrm {GL}}_{k}(\k).V$, whence the 
assertion by (iii).
\end{proof}

\begin{coro}\label{cbo1}
{\rm (i)} The subvariety ${\cal N}^{(k)}\setminus V_{k}$ of ${\cal N}^{(k)}$ has 
codimension $k+1$. 

{\rm (ii)} The restriction of $\upsilon $ to $\upsilon ^{-1}(V_{k})$ is an isomorphism 
onto $V_{k}$.

{\rm (iii)} The subset $\upsilon ^{-1}(V_{k})$ is a big open subset of 
$\sqx G{{\goth u}^{k}}$.
\end{coro}

\begin{proof}
(i) By definition, ${\cal N}^{(k)}\setminus V_{k}$ is the subset of elements $x$ of 
${\cal N}^{(k)}$ such that $E_{x}$ is contained in ${\goth g}\setminus {\goth g}_{\r}$.
Hence ${\cal N}^{(k)}\setminus V_{k}$ is contained in the image of
$\sqx G{({\goth u}\setminus {\goth u}_{\r})^{k}}$ by $\upsilon $. Let 
$(\poi x1{,\ldots,}{k}{}{}{})$ be in ${\goth u}^{k}\cap ({\cal N}^{(k)}\setminus V_{k})$.
Then, for all $(\poi a1{,\ldots,}{k}{}{}{})$ in $\k^{k}$, 
$$ \dv {x_{-\beta }}{a_{1}x_{1}+\cdots + a_{k}x_{k}} = 0$$
for some $\beta $ in $\Pi $. Since $\Pi $ is finite, $E_{x}$ is orthogonal to 
$x_{-\beta }$ for some $\beta $ in $\Pi $. As a result, the subvariety of Borel 
subalgebras of ${\goth g}$ containing $\poi x1{,\ldots,}{k}{}{}{}$ has positive 
dimension. Hence
$$ \dim ({\cal N}^{(k)}\setminus V_{k}) < 
\dim \sqx G{({\goth u}\setminus {\goth u}_{\r})^{k}} = n + k(n-1)$$
Moreover, for $\beta $ in $\Pi $, denoting by ${\goth u}_{\beta }$ the orthogonal 
complement of ${\goth g}^{-\beta }$ in ${\goth u}$, 
$\upsilon (\sqx G{({\goth u}_{\beta })^{k}})$ is contained in 
${\cal N}^{(k)}\setminus V_{k}$ and its dimension equal $(k+1)(n-1)$ since the 
variety of Borel subalgebras containing ${\goth u}_{\beta }$ has dimension $1$,
whence the assertion.

(ii) For $x$ in ${\cal N}^{(k)}$, $E_{x}$ is contained in all Borel subalgebra of 
${\goth g}$, containing the components of $x$. Then the restriction of $\upsilon $ to 
$\upsilon ^{-1}(V_{k})$ is injective since all regular nilpotent element of ${\goth g}$ 
is contained in a single Borel subalgebra of ${\goth g}$, whence the assertion 
by Zariski's Main Theorem \cite[\S 9]{Mu} since $V_{k}$ is a smooth open subset of 
${\cal N}^{(k)}$ by Lemma~\ref{lbo1},(iv).

(iii) Identify $U_{-}$ with the open subset $U_{-}B/B$ of $G/B$ and denote by 
$\pi _{0}$ the bundle projection from $\xymatrix{\sqx G{{\goth u}^{k}} \ar[r] & G/B}$. 
Since $\upsilon ^{-1}(V_{k})$ is $G$-invariant, it suffices to prove that 
$\upsilon ^{-1}(V_{k})\cap \pi _{0}^{-1}(U_{-})$ is a big open subset of 
$\pi _{0}^{-1}(U_{-})$. Let $V_{0}$ be the subset of elements $x$ of ${\goth u}^{k}$ such
that $E_{x}\cap {\goth g}_{\r}$ is not empty. Then ${\goth u}^{k}\setminus V_{0}$ is 
contained in $({\goth u}\setminus {\goth u}_{\r})^{k}$ and has codimension at least $2$
in ${\goth u}^{k}$ since $k\geq 2$. As a result, $U_{-}\times V_{0}$ is a big open subset
of $U_{-}\times {\goth u}^{k}$. The open subset $\pi _{0}^{-1}(U_{-})$ of 
$\sqx G{{\goth u}^{k}}$ identifies with $U_{-}\times {\goth u}^{k}$ and 
$\upsilon ^{-1}(V_{k})\cap \pi _{0}^{-1}(U_{-})$ identifies with $U_{-}\times V_{0}$, 
whence the assertion.
\end{proof}

\subsection{} \label{bo2}
Denote by $\pi _{{\goth g}} : \xymatrix{{\goth g}\ar[r] & {\goth g}//G}$ and 
$\pi _{{\goth h}} : \xymatrix{{\goth h} \ar[r] & {\goth h}/W({\cal R)}}$ the quotient 
maps, i.e. the morphisms defined by the invariants. Recall 
${\goth g}//G={\goth h}/W({\cal R})$, and let ${\cal X}$ be the following fiber product:
$$ \xymatrix{ {\cal X} \ar[rr]^{\overline{\chi }} \ar[d]_{\overline{\rho }} && 
{\goth g} \ar[d]^{\pi _{{\goth g}}} \\ {\goth h} \ar[rr]_{\pi _{{\goth h}}} &&
{\goth h}/W({\cal R)}} $$
where $\overline{\chi }$ and $\overline{\rho }$ are the restriction maps. The 
actions of $G$ and $W({\cal R})$ on ${\goth g}$ and ${\goth h}$ respectively induce
an action of $G\times W({\cal R})$ on ${\cal X}$: $(g,w).(x,y) := (g(x),w(y))$.

\begin{lemma}\label{lbo2}
{\rm (i)} There exists a well defined $G$-equivariant morphism $\chi _{\n}$ from 
$\sqx G{{\goth b}}$ to ${\cal X}$ such that $\gamma $ is the composition of 
$\chi _{\n}$ and $\overline{\chi }$.

{\rm (ii)} The morphism $\chi _{\n}$ is projective and birational. Moreover, 
${\cal X}$ is irreducible.

{\rm (iii)} The subscheme ${\cal X}$ is normal. Moreover, every element of 
${\goth g}_{\r}\times {\goth h}\cap {\cal X}$ is a smooth point of ${\cal X}$.

{\rm (iv)} The algebra $\k[{\cal X}]$ is the space of global sections of 
$\an {\sqx G{{\goth b}}}{}$ and $\k[{\cal X}]^{G}=\e Sh$.
\end{lemma}

\begin{proof}
(i) Since the map $(g,x)\mapsto (g(x),\overline{x})$ is constant on the $B$-orbits, there
exists a uniquely defined morphism $\chi _{\n}$ from $G\times _{B}{\goth b}$ to 
${\goth g}\times {\goth h}$ such that $(g(x),\overline{x})$
is the image by $\chi _{\n}$ of the image of $(g,x)$ in $G\times _{B}{\goth b}$. 
The image of $\chi _{\n}$ is contained in ${\cal X}$ since for all $p$ in 
$\e Sg^{G}$, $p(\overline{x})=p(x)=p(g(x))$. Furthermore, $\chi _{\n}$ verifies the
condition of the assertion.

(ii) According to Lemma~\ref{l4int}, $\chi _{\n}$ is a projective morphism. 
Let $(x,y)$ be in ${\goth g}\times {\goth h}$ such that $p(x)=p(y)$ for all $p$ in 
$\e Sg^{G}$. For some $g$ in $G$, $g(x)$ is in ${\goth b}$ and its semisimple 
component is $y$ so that $(x,y)$ is in the image of $\chi _{\n}$. As a result, 
${\cal X}$ is irreducible as the image of the irreducible variety 
$G\times _{B}{\goth b}$. Since for all $(x,y)$ in ${\cal X}\cap \gg h{\r}$, there exists 
a unique $w$ in $W({\cal R})$ such that $y=w(x)$, the fiber of $\chi _{\n}$ at any 
element ${\cal X}\cap G.(\gg h{\r})$ has one element. Hence $\chi _{\n}$ is birational,
whence the assertion. 

(iii) The morphism $\pi _{{\goth h}}$ is finite, and so is $\overline{\chi }$.
Moreover $\pi _{{\goth h}}$ is smooth over ${\goth h}_{\r}$, $\overline{\chi }$
is smooth over ${\goth g}_{\r}$. Finally, $\pi _{{\goth g}}$
is flat and all fibers are normal and Cohen-Macaulay. Thus the same holds for the 
morphism $\overline{\rho }$. Since ${\goth h}$ is smooth this implies that ${\cal X}$ is 
normal and Cohen-Macaulay by \cite[Ch. 8, \S 23]{Mat}.

(iv) According to (ii), (iii) and Lemma~\ref{lint}, 
$\k[{\cal X}]={\mathrm {H}}^{0}(\sqx G{{\goth b}},\an {\sqx G{{\goth b}}}{})$. Under the
action of $G$ in ${\goth g}\times {\goth h}$, 
$\k[{\goth g}\times {\goth h}]^{G}=\tk {\k}{\e Sg^{G}}\e Sh$ and its image in 
$\k[{\cal X}]$ by the quotient morphism is equal to $\e Sh$. Moreover, since $G$ is 
reductive, $\k[{\cal X}]^{G}$ is the image of $\k[{\goth g}\times {\goth h}]^{G}$ by the 
quotient morphism, whence the assertion. 
\end{proof}

\begin{prop}\label{pbo2}\cite[Theorem B and Corollary]{He}
The variety ${\cal X}$ has rational singularities.
\end{prop}

\begin{coro}\label{cbo2}
{\rm (i)} Let $x$ and $x'$ be in ${\goth b}_{\r}$ such that $x'$ is in $G.x$ and 
$\overline{x'}=\overline{x}$. Then $x'$ is in $B(x)$.

{\rm (ii)} For all $w$ in $W({\cal R})$, the map
$$U_{-}\times {\goth b}_{\r} \longrightarrow {\cal X} , \qquad
(g,x) \longmapsto (g(x),w(\overline{x})) $$
is an isomorphism onto a smooth open subset of ${\cal X}$. 
\end{coro}

\begin{proof}
(i) The semisimple components of $x$ and $x'$ are conjugate under 
$B$ since they are conjugate to $\overline{x}$ under $B$. Let $b$ and $b'$ be in $B$ such
that $\overline{x}$ is the semisimple component of
$b(x)$ and $b'(x')$. Then the nilpotent components of $b(x)$ and $b'(x')$ are regular 
nilpotent elements of ${\goth g}^{\overline{x}}$, belonging to the Borel subalgebra
${\goth b}\cap {\goth g}^{\overline{x}}$ of ${\goth g}^{\overline{x}}$. Hence 
$x'$ is in $B(x)$ since regular nilpotent elements of a Borel subalgebra of a reductive
Lie algebra are conjugate under the correponding Borel subgroup.

(ii) Since the action of $G$ and $W({\cal R})$ on ${\cal X}$ commute, it suffices to 
prove the assertion for $w=1_{{\goth h}}$. Denote by $\theta $ the map
$$ \xymatrix{U_{-}\times {\goth b}_{\r} \ar[rr] && {\cal X}} , \qquad
(g,x) \longmapsto (g(x),\overline{x}) .$$
Let $(g,x)$ and $(g',x')$ be in $U_{-}\times {\goth b}_{\r}$ such that 
$\theta (g,x)=\theta (g',x')$. By (i), $x'=b(x)$ for some $b$ in $B$. Hence
$g^{-1}g'b$ is in $G^{x}$. Since $x$ is in ${\goth b}_{\r}$, $G^{x}$ is contained in 
$B$ and $g^{-1}g'$ is in $U_{-}\cap B$, whence $(g,x)=(g',x')$ since 
$U_{-}\cap B=\{1_{{\goth g}}\}$. As a result, $\theta $ is a dominant injective map from 
$U_{-}\times {\goth b}_{\r}$ to the normal variety ${\cal X}$. Hence $\theta $ is an 
isomorphism onto a smooth open subset of ${\cal X}$, by Zariski's Main Theorem 
\cite[\S 9]{Mu}.
\end{proof}

\subsection{} \label{bo3}
According to Lemma~\ref{lbo},(i), $\sqx G{{\goth b}^{k}}$ is a desingularization of 
${\cal B}^{(k)}$ and we have the commutative diagram:
$$\xymatrix{\sqx G{{\goth b}^{k}} \ar[rr]^{\gamma _{\n}} \ar[rd]_{\gamma } &&
{\cal B}_{\n}^{(k)} \ar[ld]^{\eta _{\n}} \\ & {\cal B}^{(k)} & } $$

\begin{lemma}\label{lbo3}
Let $\varpi $ be the canonical projection from ${\cal X}^{k}$ to ${\goth g}^{k}$. Denote 
by $\iota _{k}$ the map
$$ {\goth b}^{k} \longrightarrow {\cal X}^{k} , \qquad 
(\poi x1{,\ldots,}{k}{}{}{}) \longmapsto 
(\poi x1{,\ldots,}{k}{}{}{},\overline{x_{1}},\ldots,\overline{x_{k}}) .$$

{\rm (i)} The map $\iota _{k}$ is a closed embedding of ${\goth b}^{k}$ into 
${\cal X}^{k}$.

{\rm (ii)} The subvariety $\iota _{k}({\goth b}^{k})$ of ${\cal X}^{k}$ is an irreducible 
component of $\varpi ^{-1}({\goth b}^{k})$.

{\rm (iii)} The subvariety $\varpi ^{-1}({\goth b}^{k})$ of ${\cal X}^{k}$ is 
invariant under the canonical action of $W({\cal R})^{k}$ in ${\cal X}^{k}$ and this 
action induces a simply transitive action of $W({\cal R})^{k}$ on the set of irreducible 
components of $\varpi ^{-1}({\goth b}^{k})$.
\end{lemma}

\begin{proof}
(i) The map 
$$\xymatrix{{\goth b}^{k} \ar[rr] && G^{k}\times {\goth b}^{k}} , \qquad
(\poi x1{,\ldots,}{k}{}{}{}) \longmapsto 
(1_{{\goth g}},\ldots,1_{{\goth g}},\poi x1{,\ldots,}{k}{}{}{}) $$
defines through the quotient a closed embedding of ${\goth b}^{k}$ in 
$\sqxx k{{\goth b}^{k}}$. Denote it by $\iota '$. Let $\chi _{\n}^{(k)}$ be the map
$$ \xymatrix{\sqxx k{{\goth b}^{k}} \ar[rr] && {\cal X}^{k}} , \qquad 
(\poi x1{,\ldots,}{k}{}{}{}) \longmapsto (\poi x1{,\ldots,}{k}{\chi }{\n}{\n}) .$$
Then $\iota _{k}= \chi _{\n}^{(k)}\rond \iota '$. Since $\chi _{\n}$ is a projective 
morphism, $\iota _{k}$ is a closed morphism. Moreover, it is an embedding since 
$\varpi \rond \iota _{k}$ is the identity of ${\goth b}^{k}$. 

(ii) Since $\e Sh$ is a finite extension of $\e Sh^{W({\cal R})}$, $\varpi $ is a finite 
morphism. So $\varpi ^{-1}({\goth b}^{k})$ and ${\goth b}^{k}$ have the same dimension. 
According to (i), $\iota _{k}({\goth b}^{k})$ is an irreducible subvariety of 
$\omega ^{-1}({\goth b}^{k})$ of the same dimension, whence the assertion.  
 
(iii) Since all the fibers of $\varpi $ are invariant under the action of 
$W({\cal R})^{k}$ on ${\cal X}^{k}$, $\varpi ^{-1}({\goth b}^{k})$ is 
invariant under this action and $W({\cal R})^{k}$ permutes the irreducible 
components of $\varpi ^{-1}({\goth b}^{k})$. For $w$ in $W({\cal R})^{k}$, set 
$Z_{w}:= w.\iota _{k}({\goth b}^{k})$. If $Z_{w}=\iota _{k}({\goth b}^{k})$, then for all 
$(\poi x1{,\ldots,}{k}{}{}{})$ in ${\goth h}_{\r}^{k}$, 
$(\poi x1{,\ldots,}{k}{}{}{},w.(\poi x1{,\ldots,}{k}{}{}{}))$ is in 
$\iota _{k}({\goth b}^{k})$ so that $(\poi x1{,\ldots,}{k}{}{}{})$ is invariant under 
$w$ and $w$ is the identity. 

Let $Z$ be an irreducible component of $\varpi ^{-1}({\goth b}^{k})$ and let $Z_{0}$ be 
its image by the map 
$$(\poi x1{,\ldots,}{k}{}{}{},\poi y1{,\ldots,}{k}{}{}{})\longmapsto 
(\overline{x_{1}},\ldots,\overline{x_{k}},\poi y1{,\ldots,}{k}{}{}{}) .$$ 
Since $\varpi $ is $G^{k}$-equivariant and ${\goth b}^{k}$ is invariant under 
$B^{k}$, $\varpi ^{-1}({\goth b}^{k})$ and $Z$ are invariant under $B^{k}$. Hence by 
Lemma~\ref{l5int}, $Z_{0}$ is closed. Moreover, since the image of the map
$$\xymatrix{Z_{0}\times {\goth u}^{k} \ar[rr] && {\cal X}^{k}} , \quad
((\poi x1{,\ldots,}{k}{}{}{},\poi y1{,\ldots,}{k}{}{}{}),(\poi u1{,\ldots,}{k}{}{}{}))
\longmapsto (x_{1}+u_{1},\ldots,x_{k}+u_{k},\poi y1{,\ldots,}{k}{}{}{})$$
is an irreducible subset of $\varpi ^{-1}({\goth b}^{k})$ containing $Z$, $Z$ is the 
image of this map. Since $Z_{0}$ is contained in ${\cal X}^{k}$, $Z_{0}$ is contained 
in the image of the map
$$ \xymatrix{{\goth h}^{k}\times W({\cal R})^{k} \ar[rr] &&  
{\goth h}^{k}\times {\goth h}^{k}} , \quad 
(\poi x1{,\ldots,}{k}{}{}{},\poi w1{,\ldots,}{k}{}{}{}) \longmapsto 
(\poi x1{,\ldots,}{k}{}{}{},\poi x1{,\ldots,}{k}{w}{1}{k}) .$$ 
As $W({\cal R})$ is finite and $Z_{0}$ is irreducible, for some 
$w$ in $W({\cal R})^{k}$, $Z_{0}$ is the image of ${\goth h}^{k}$ by the map 
$$(\poi x1{,\ldots,}{k}{}{}{})\longmapsto 
(\poi x1{,\ldots,}{k}{}{}{},w.(\poi x1{,\ldots,}{k}{}{}{})) $$ 
and $Z=Z_{w}$, whence the assertion.
\end{proof}

Set ${\goth Y} := \sqxx k{{\goth b}^{k}}$. The map
$$ \xymatrix{G\times {\goth b}^{k} \ar[rr] &&  G^{k}\times {\goth b}^{k}} , \qquad
(g,\poi v1{,\ldots,}{k}{}{}{}) \longmapsto (g,\ldots,g,\poi v1{,\ldots,}{k}{}{}{}) $$
defines through the quotient a closed immersion from $\sqx G{{\goth b}^{k}}$ to 
${\goth Y}$. Denote it by $\nu $. Consider the diagonal action of $G$ on 
${\goth g}^{k}$ and ${\cal X}^{k}$: 
$g.(\poi x1{,\ldots,}{k}{}{}{},\poi y1{,\ldots,}{k}{}{}{}) := 
(\poi x1{,\ldots,}{k}{g}{}{},\poi y1{,\ldots,}{k}{}{}{})$, and identify 
$\sqx G{{\goth b}^{k}}$ with $\nu (\sqx G{{\goth b}^{k}})$ by the closed 
immersion $\nu $.

\begin{coro}\label{cbo3}
Set ${\cal B}_{\x}^{(k)} := G.\iota _{k}({\goth b}^{k})$.

{\rm (i)} The subset ${\cal B}_{\x}^{(k)}$ is the image of $\sqx G{{\goth b}^{k}}$ 
by $\chi _{\n}^{(k)}$. Moreover, the restriction of $\chi _{\n}^{(k)}$ to 
$\sqx G{{\goth b}^{k}}$ is a projective birational morphism from $\sqx G{{\goth b}^{k}}$ 
onto ${\cal B}_{\x}^{(k)}$.

{\rm (ii)} The subset ${\cal B}_{\x}^{(k)}$ of ${\cal X}^{k}$ is an irreducible 
component of $\varpi ^{-1}({\cal B}^{(k)})$. 

{\rm (iii)} The subvariety $\varpi ^{-1}({\cal B}^{(k)})$ of ${\cal X}^{k}$ is 
invariant under $W({\cal R})^{k}$ and this action induces a simply transitive action of 
$W({\cal R})^{k}$ on the set of irreducible components of $\varpi ^{-1}({\cal B}^{(k)})$.

{\rm (iv)} The subalgebra $\k[{\cal B}^{(k)}]$ of $\k[\varpi ^{-1}({\cal B}^{(k)}]$
equals $\k[\varpi ^{-1}({\cal B}^{(k)})]^{W({\cal R})^{k}}$ with respect to the action 
of $W({\cal R})^{k}$ on $\varpi ^{-1}({\cal B}^{(k)})$.
\end{coro}

\begin{proof}
(i) Let $\gamma _{\x}$ be the restriction of $\chi _{\n}^{(k)}$ to 
$\sqx G{{\goth b}^{k}}$. Since $\iota _{k}= \chi _{\n}^{(k)}\rond \iota '$, 
$\sqx G{{\goth b}^{k}}=G.\iota '({\goth b}^{(k)})$ and $\chi _{\n}^{(k)}$ is 
$G$-equivariant, ${\cal B}_{\x}^{(k)}=\gamma _{\x}(\sqx G{{\goth b}^{k}})$. 
Hence ${\cal B}_{\x}^{(k)}$ is closed in ${\cal X}^{k}$ and $\gamma _{\x}$ is
a projective morphism from $\sqx G{{\goth b}^{k}}$ to ${\cal B}_{\x}^{(k)}$ since 
$\chi _{\n}^{(k)}$ is a projective morphism. According to Lemma~\ref{lbo},(i), 
$\varpi \rond \gamma _{\x}$ is a birational morphism onto ${\cal B}^{(k)}$. Then 
$\gamma _{\x}$ is birational since 
$\varpi ({\cal B}_{\x}^{(k)})={\cal B}^{(k)}$, whence the assertion.

(ii) Since $\varpi $ is a finite morphism, $\varpi ^{-1}({\cal B}^{(k)})$,
${\cal B}^{(k)}$ and ${\cal B}_{\x}^{(k)}$ have the same dimension, whence the 
assertion since ${\cal B}_{\x}^{(k)}$ is irreducible as an image of an irreducible
variety.

(iii) Since the fibers of $\varpi $ are invariant under $W({\cal R})^{k}$,
$\varpi ^{-1}({\cal B}^{(k)})$ is invariant under this action and 
$W({\cal R})^{k}$ permutes the irreducible components of 
$\varpi ^{-1}({\cal B}^{(k)})$. Let $Z$ be an irreducible component of 
$\varpi ^{-1}({\cal B}^{(k)})$. As $\varpi $ is $G^{k}$-equivariant, 
$\varpi ^{-1}({\cal B}^{(k)})$ and $Z$ are invariant under the diagonal action of $G$. 
Moreover, $Z=G.(Z\cap \varpi ^{-1}({\goth b}^{k}))$ since 
${\cal B}^{(k)}=G.{\goth b}^{k}$. Hence for some irreducible
component $Z_{0}$ of $Z\cap \varpi ^{-1}({\goth b}^{k})$, $Z=G.Z_{0}$. According to 
Lemma~\ref{lbo3},(iii), $Z_{0}$ is contained in $w.\iota _{k}({\goth b}^{k})$ for some 
$w$ in $W({\cal R})^{k}$. Hence $Z=w.{\cal B}_{\x}^{(k)}$ since the 
actions of $G^{k}$ and $W({\cal R})^{k}$ on ${\cal X}^{k}$ commute and $Z$ is an 
irreducible component of $\varpi ^{-1}({\cal B}^{(k)})$. 

Let $w=(\poi w1{,\ldots,}{k}{}{}{})$ be in $W({\cal R})^{k}$ such that 
$w.{\cal B}_{\x}^{(k)}={\cal B}_{\x}^{(k)}$. Let $x$ be in ${\goth h}_{\r}$ 
and let $i=1,\ldots,k$. Set: 
$$ z := (\poi x1{,\ldots,}{k}{}{}{},\overline{x_{1}},\ldots,\overline{x_{k}}) 
\mbox{ with } x_{j} := \left \{ \begin{array}{ccc} x & \mbox{ if } & j=i \\
x_{j} = e & \mbox{ otherwise } & \end{array} \right. . $$
Then there exists $(\poi y1{,\ldots,}{k}{}{}{})$ in ${\goth b}^{k}$ and $g$ in $G$ such 
that
$$ w.z = (\poi y1{,\ldots,}{k}{g}{}{},\overline{y_{1}},\ldots,\overline{y_{k}}) .$$ 
For some $b$ in $B$, $b(y_{i})=\overline{y_{i}}$ since $y_{i}$ is a regular 
semisimple element, belonging to ${\goth b}$. As a result, $gb^{-1}(\overline{y_{i}})=x$
and $w_{i}(x)=\overline{y_{i}}$. Hence $gb^{-1}$ is an element of $N_{G}({\goth h})$
representing $w_{i}^{-1}$. Furthermore, since $gb^{-1}(b(y_{j}))=e$ for $j\neq i$, 
$b(y_{j})$ is a regular nilpotent element belonging to ${\goth b}$. Then, since 
there is one and only one Borel subalgebra containing a regular nilpotent element, 
$gb^{-1}({\goth b})={\goth b}$ and $w_{i}=1_{{\goth h}}$. As a result, $w$ is the
identity of $W({\cal R})^{k}$, whence the assertion.   

(iv) Since the fibers of $\varpi $ are invariant under $W({\cal R})^{k}$,
$\k[{\cal B}^{(k)}]$ is contained in $\k[\varpi ^{-1}({\cal B}^{(k)})]^{W({\cal R})^{k}}$.
Let $p$ be in $\k[\varpi ^{-1}({\cal B}^{(k)})]^{W({\cal R})^{k}}$. Since $W({\cal R})$
is a finite group, $p$ is the restriction to $\varpi ^{-1}({\cal B}^{(k)})$ of an element
$q$ of $\k[{\cal X}]^{\tens k}$, invariant under $W({\cal R})^{k}$. Since 
$\k[{\cal X}]^{W({\cal R})}=\e Sg$, $q$ is in $\e Sg^{\tens k}$ by Lemma~\ref{lbo},(iv),
and $p$ is in $\k[{\cal B}^{(k)}]$, whence the assertion.
\end{proof}

\subsection{} \label{bo4}
For $\alpha $ a positive root, denote by ${\goth h}_{\alpha }$ the kernel of $\alpha $
and by $S_{\alpha }$ the closure in ${\goth b}$ of the image of the map
$$ \xymatrix{U\times {\goth h}_{\alpha } \ar[rr] && {\goth b}}, \qquad 
(g,x) \longmapsto g(x) .$$ 
For $\beta $ in $\Pi $, set:
$$ {\goth u}_{\beta } := 
\bigoplus _{\alpha \in {\cal R}_{+}\setminus \{\beta \}} {\goth g}^{\beta } , \qquad
{\goth b}_{\beta } := {\goth h}_{\beta } \oplus {\goth u}_{\beta } .$$

\begin{lemma}\label{lbo4}
For $\alpha $ in ${\cal R}_{+}$, let ${\goth h}'_{\alpha }$ be the set of subregular 
elements belonging to ${\goth h}_{\alpha }$.

{\rm (i)} For $\alpha $ in ${\cal R}_{+}$, $S_{\alpha }$ is a subvariety of codimension
$2$ of ${\goth b}$. Moreover, it is contained in ${\goth b}\setminus {\goth b}_{\r}$.

{\rm (ii)} For $\beta $ in $\Pi $, $S_{\beta }={\goth b}_{\beta }$.

{\rm (iii)} The $S_{\alpha }$'s, $\alpha \in {\cal R}_{+}$, are the irreducible
components of ${\goth b}\setminus {\goth b}_{\r}$.
\end{lemma}

\begin{proof}
(i) For $x$ in ${\goth h}'_{\alpha }$, ${\goth b}^{x}={\goth h}+\k x_{\alpha }$. Hence 
$U({\goth h}'_{\alpha })$ has dimension $n-1+\rg-1$, whence the assertion since 
$U({\goth h}'_{\alpha })$ is dense in $S_{\alpha }$ and ${\goth h}'_{\alpha }$
is contained in ${\goth b}\setminus {\goth b}_{\r}$.

(ii) For $\beta $ in $\Pi $, $U({\goth h}'_{\beta })$ is contained in ${\goth b}_{\beta }$
since ${\goth b}_{\beta }$ is an ideal of ${\goth b}$, whence the assertion by (i).

(iii) According to (i), it suffices to prove that ${\goth b}\setminus {\goth b}_{\r}$
is the union of the $S_{\alpha }$'s. Let $x$ be in ${\goth b}\setminus {\goth b}_{\r}$. 
According to~\cite{Ve}, for some $g$ in $G$ and for some $\beta $ in $\Pi $, $x$ is in
$g({\goth b}_{\beta })$. Since ${\goth b}_{\beta }$ is an ideal of ${\goth b}$, by 
Bruhat's decomposition of $G$, for some $b$ in $B$ and for some $w$ in $W({\cal R})$, 
$b^{-1}(x)$ is in $w({\goth b}_{\beta })\cap {\goth b}$. By definition,
$$ w({\goth b}_{\beta }) = w({\goth h}_{\beta }) \oplus w({\goth u}_{\beta })
= {\goth h}_{w(\beta )} \oplus 
\bigoplus _{\alpha \in {\cal R}_{+}\setminus \{\beta \}} 
{\goth g}^{w(\alpha )} .$$
So, 
$$w({\goth b}_{\beta })\cap {\goth b} = {\goth h}_{w(\beta )} \oplus 
{\goth u}_{0} \mbox{ with } {\goth u}_{0} :=
\bigoplus _{\mycom{\alpha \in {\cal R}_{+}\setminus \{\beta \}}{w(\alpha )\in 
{\cal R}_{+}}} {\goth g}^{w(\alpha )} .$$
The subspace ${\goth u}_{0}$ of ${\goth u}$ is a subalgebra, not containing 
${\goth g}^{w(\beta )}$. Then, denoting by $U_{0}$ the closed subgroup of $U$ whose
Lie algebra is $\ad {\goth u}_{0}$, 
$$ \overline{U_{0}({\goth h}_{w(\beta )})} = w({\goth b}_{\beta })\cap {\goth b}$$
since the left hand side is contained in the right hand side and has the same dimension.
As a result, $x$ is in $S_{w(\beta )}$ since $S_{w(\beta )}$ is $B$-invariant, 
whence the assertion.
\end{proof}

Recall that $\theta $ is the map
$$\xymatrix{U_{-}\times {\goth b}_{\r} \ar[rr] && {\cal X}} , \qquad  (g,x) \longmapsto 
(g(x),\overline{x}) $$
and denote by $W'_{k}$ the inverse image of $\theta (U_{-}\times {\goth b}_{\r})$
by the projection 
$$\xymatrix{{\cal B}_{\x}^{(k)} \ar[rr] && {\cal X}} , \qquad
(\poi x1{,\ldots,}{k}{}{}{},\poi y1{,\ldots,}{k}{}{}{}) \longmapsto (x_{1},y_{1}) .$$

\begin{lemma}\label{l2bo4}
Let $W_{k}$ be the subset of elements $(x,y)$ of ${\cal B}_{\x}^{(k)}$ 
($x\in {\goth g}^{k},y\in {\goth h}^{k}$) such that $E_{x}\cap {\goth g}_{\r}$ is not 
empty.

{\rm (i)} The subset $W'_{k}$ of ${\cal B}_{\x}^{(k)}$ is a smooth open subset.
Moreover, the map
$$\xymatrix{U_{-}\times {\goth b}_{\r}\times {\goth b}^{k-1} \ar[rr] &&  W'_{k}} , \qquad
(g,\poi x1{,\ldots,}{k}{}{}{}) \longmapsto (\poi x1{,\ldots,}{k}{g}{}{},
\overline{x_{1}},\ldots,\overline{x_{k}}) .$$
is an isomorphism of varieties.

{\rm (ii)} The subset ${\cal B}_{\x}^{(k)}$ of ${\goth g}^{k}\times {\goth h}^{k}$ 
is invariant under the canonical action of ${\mathrm {GL}}_{k}(\k)$.

{\rm (iii)} The subset $W_{k}$ of ${\cal B}_{\x}^{(k)}$ is a smooth open subset.
Moreover, $W_{k}$ is the $G\times {\mathrm {GL}}_{k}(\k)$-invariant set generated by
$W'_{k}$.

{\rm (iv)} The subvariety ${\cal B}_{\x}^{(k)}\setminus W_{k}$ of 
${\cal B}_{\x}^{(k)}$ has codimension at least $2k$. 
\end{lemma}

\begin{proof}
(i) According to Corollary~\ref{cbo2},(ii), $\theta $ is an isomorphism onto
a smooth open subset of ${\cal X}$. As a result, $W'_{k}$ is an open subset of 
${\cal B}_{\x}^{(k)}$ and the map
$$ \xymatrix{U_{-}\times {\goth b}_{\r}\times {\goth b}^{k-1} \ar[rr] && W'_{k}} , \qquad
(g,\poi x1{,\ldots,}{k}{}{}{}) \longmapsto 
(\poi x1{,\ldots,}{k}{g}{}{},\overline{x_{1}},\ldots,\overline{x_{k}}) $$
is an isomorphism whose inverse is given by
$$ \xymatrix{W'_{k} \ar[rr] && U_{-}\times {\goth b}_{\r}\times {\goth b}^{k-1}},$$ 
$$ (\poi x1{,\ldots,}{k}{}{}{}) \longmapsto  (\theta ^{-1}(x_{1},\overline{x_{1}})_{1},
\poi x1{,\ldots,}{k}{\theta ^{-1}(x_{1},\overline{x_{1}})_{1}}{}{}) $$
with $\theta ^{-1}$ the inverse of $\theta $ and 
$\theta ^{-1}(x_{1},\overline{x_{1}})_{1}$ the component of 
$\theta ^{-1}(x_{1},\overline{x_{1}})$ on $U_{-}$, whence the assertion since 
$U_{-}\times {\goth b}_{\r}\times {\goth b}^{k-1}$ is smooth.

(ii) For $(\poi x1{,\ldots,}{k}{}{}{})$ in ${\goth b}^{k}$ and for  
$(a_{i,j},1\leq i,j\leq k)$ in ${\mathrm {GL}}_{k}(\k)$,
$$ \overline{\sum_{j=1}^{k} a_{i,j}x_{j}} = \sum_{j=1}^{k}a_{i,j}\overline{x_{j}} $$
so that $\iota _{k}({\goth b}^{k})$ is invariant under the action of 
${\mathrm {GL}}_{k}(\k)$ in ${\goth g}^{k}\times {\goth h}^{k}$ defined by
$$ (a_{i,j},1\leq i,j\leq k) . (\poi x1{,\ldots,}{k}{}{}{},\poi y1{,\ldots,}{k}{}{}{})
:= (\sum_{j=1}^{k} a_{i,j}x_{j},j=1,\ldots,k,\sum_{j=1}^{k} a_{i,j}y_{j},j=1,\ldots,k) ,$$
whence the assertion since ${\cal B}_{\x}^{(k)}=G.\iota _{k}({\goth b}^{k})$ and 
the actions of $G$ and ${\mathrm {GL}}_{k}(\k)$ in 
${\goth g}^{k}\times {\goth h}^{k}$ commute.

(iii) According to (i), $G.W'_{k}$ is a smooth open subset of 
${\cal B}_{\x}^{(k)}$. Moreover, $G.W'_{k}$ is the subset of elements $(x,y)$
such that the first component of $x$ is regular. So, by (ii) and Lemma~\ref{l6int}, 
$W_{k}= {\mathrm {GL}}_{k}(\k).(G.W'_{k})$, whence the assertion. 

(iv) According to Corollary~\ref{cbo3},(i), ${\cal B}_{\x}^{(k)}$ is the image of 
$\sqx G{{\goth b}^{k}}$ by the restriction $\gamma _{\x}$ of $\chi _{\n}^{(k)}$ to
$\sqx G{{\goth b}^{k}}$. Then ${\cal B}_{\x}^{(k)}\setminus W_{k}$ is contained in
the image of $\sqx G{({\goth b}\setminus {\goth b}_{\r})^{k}}$ by $\gamma _{\x}$. 
As a result, by Lemma~\ref{lbo4},
$$ \dim {\cal B}_{\x}^{k}\setminus W_{k} \leq n + k(\b g{}-2) ,$$
whence the assertion.
\end{proof}

\begin{coro}\label{cbo4}
The restriction of $\gamma _{\x}$ to $\gamma _{\x}^{-1}(W_{k})$
is an isomorphism onto $W_{k}$. Moreover, $\gamma _{\x}^{-1}(W_{k})$ is a big open subset 
of $\sqx G{{\goth b}^{k}}$.
\end{coro}

\begin{proof}
Since the subset of Borel subalgebras containing a regular element is finite,
the fibers of $\gamma _{\x}$ over the elements of $W_{k}$ are finite. In particular, 
the restriction of $\gamma _{\x}$ to $\gamma _{\x}^{-1}(W_{k})$ is a quasi finite
surjective morphism onto $W_{k}$. So, by Zariski's Main Theorem \cite[\S 9]{Mu}, it is an 
isomorphism since $W_{k}$ is smooth by Lemma~\ref{l2bo4},(iii).

Recall that $\sqx G{{\goth b}^{k}}$ identifies with a closed subset of 
$G/B\times {\goth g}^{k}$. For $u$ in $G/B$ and $x$ in ${\goth g}^{k}$ such that 
$(u,x)$ is in $\sqx G{{\goth b}^{k}}$, $(u,x)$ is in $\gamma _{\x}^{-1}(W_{k})$ if 
and only if $E_{x}\cap {\goth g}_{\r}$ is not empty. Denote by $\pi $ the bundle 
projection of the vector bundle $\sqx G{{\goth b}^{k}}$ over $G/B$. Let $\Sigma $ be an 
irreducible component of $\sqx G{{\goth b}^{k}}\setminus \gamma _{\x}^{-1}(W_{k})$. For
$u$ in $\pi (\Sigma )$, set:
$$ \Sigma _{u} := \{x \in {\goth g}^{k} \ \vert \ (u,x) \in \Sigma \} .$$
Since $W_{k}$ is a cone, for all $u$ in $\pi (\Sigma )$, $\Sigma _{u}$ is a closed cone
of $u^{k}$, whence $\pi (\Sigma )\times \{0\} = \Sigma \cap G/B\times \{0\}$
so that $\pi (\Sigma )$ is a closed subset of $G/B$. Suppose that $\Sigma $ has 
codimension $1$ in $\sqx G{{\goth b}^{k}}$. A contradiction is expected. Then 
$\pi (\Sigma )$ has codimension at most $1$ in $G/B$. If $\pi (\Sigma )$ has 
codimension $1$ in $G/B$, then for all $u$ in $\pi (\Sigma )$, $\Sigma _{u} = u^{k}$. It 
is impossible since $u\cap {\goth g}_{\r}$ is not empty. As a result, for all $u$ in 
a dense open subset of $G/B$, $\Sigma _{u}$ is closed of codimension $1$ in $u^{k}$. 
According to Lemma~\ref{lbo4}, $u\setminus {\goth g}_{\r}$ has codimension $2$ in $u$
and $\Sigma _{u}$ is contained in $(u\setminus {\goth g}_{\r})^{k}$, whence the 
contradiction.
\end{proof}

\subsection{} \label{bo5}
For $E$ a $B$-module, denote by ${\cal L}_{0}(E)$ the sheaf of local sections of the 
vector bundle $\sqx GE$ over $G/B$. Let $\Delta $ be the diagonal of $(G/B)^{k}$ and let 
${\cal J}_{\Delta }$ be its ideal of definition in $\an {(G/B)^{k}}{}$. The variety $G/B$
identifies with $\Delta $ so that $\an {(G/B)^{k}}{}/{\cal J}_{\Delta }$ is isomorphic to
$\an {G/B}{}$. For $E$ a $B^{k}$-module, denote by ${\cal L}(E)$ the sheaf of local 
sections of the vector bundle $\sqxx kE$ over $(G/B)^{k}$.

\begin{lemma}\label{lbo5}
Let $E$ be a $B^{k}$-module. Denote by $\overline{E}$ the $B$-module defined by the 
diagonal action of $B$ on $E$. The short sequence of $\an {(G/B)^{k}}{}$-modules
$$ 0 \longrightarrow \tk {\an {\sqxx k{{\goth b}^{k}}}{}}{{\cal J}_{\Delta }}{\cal L}(E)
\longrightarrow {\cal L}(E) \longrightarrow {\cal L}_{0}(\overline{E}) \longrightarrow 0$$
is exact.
\end{lemma}

\begin{proof}
Since ${\cal L}(E)$ is a locally free $\an {(G/B)^{k}}{}$-module, the short 
sequence of $\an {(G/B)^{k}}{}$-modules
$$ 0 \longrightarrow \tk {\an {(G/B)^{k}}{}}{{\cal J}_{\Delta }}{\cal L}(E) 
\longrightarrow {\cal L}(E) \longrightarrow \tk {\an {(G/B)^{k}}{}}{\an {\Delta }{}}
{\cal L}(E) \longrightarrow 0$$
is exact, whence the assertion since $\tk {\an {(G/B)^{k}}{}}{\an {\Delta }{}}
{\cal L}(E)$ is isomorphic to ${\cal L}_{0}(\overline{E})$.
\end{proof}

From Lemma~\ref{lbo5} results a canonical morphism 
$$ \xymatrix{{\mathrm {H}}^{0}((G/B)^{k},{\cal L}(E)) \ar[rr] && 
{\mathrm {H}}^{0}(G/B,{\cal L}_{0}(\overline{E}))}$$
for all $B^{k}$-module $E$. According to the identification of ${\goth g}$ and 
${\goth g}^{*}$ by $\dv ..$, the duals of ${\goth b}$ and ${\goth u}$ identify with 
${\goth b}_{-}$ and ${\goth u}_{-}$ respectively so that ${\goth b}_{-}$ and 
${\goth u}_{-}$ are $B$-modules. 

\begin{lemma}\label{l2bo5}
{\rm (i)} The algebra $\k[{\cal B}_{\n}^{(k)}]$ is equal to 
${\mathrm {H}}^{0}(G/B,{\cal L}_{0}(\overline{\es S{{\goth b}_{-}^{k}}}))$.

{\rm (ii)}  The algebra $\k[{\cal N}_{\n}^{(k)}]$ is equal to 
${\mathrm {H}}^{0}(G/B,{\cal L}_{0}(\overline{\es S{{\goth u}_{-}^{k}}}))$.

{\rm (iii)} The algebra $\k[{\cal B}_{\x}^{(k)}]$ is the image of the morphism
$$ \xymatrix{{\mathrm {H}}^{0}((G/B)^{k},{\cal L}(\es S{{\goth b}_{-}^{k}})) \ar[rr] && 
{\mathrm {H}}^{0}(G/B,{\cal L}_{0}(\overline{\es S{{\goth b}_{-}^{k}}}))} .$$

{\rm (iv)} The algebra $\k[{\cal N}^{(k)}]$ is the image of the morphism 
$$ \xymatrix{{\mathrm {H}}^{0}((G/B)^{k},{\cal L}(\es S{{\goth u}_{-}^{k}})) \ar[rr] && 
{\mathrm {H}}^{0}(G/B,{\cal L}_{0}(\overline{\es S{{\goth u}_{-}^{k}}})) }. $$
\end{lemma}

\begin{proof}
(i) Since $\sqx G{{\goth b}^{k}}$ is a desingularization of the normal variety 
${\cal B}_{\n}^{(k)}$, $\k[{\cal B}_{\n}^{(k)}]$ is the space of global sections
of $\an {\sqx G{{\goth b}^{k}}}{}$ by Lemma~\ref{lint}. Let $\pi $ be the bundle 
projection of the fiber bundle $\sqx G{{\goth b}^{k}}$. Since $\es S{{\goth b}_{-}^{k}}$ 
is the space of polynomial functions on ${\goth b}^{k}$, 
$$ \pi _{*}(\an {\sqx G{{\goth b}^{k}}}{}) = 
{\cal L}_{0}(\overline{\es S{{\goth b}_{-}^{k}}}),$$
whence the assertion.

(ii) By Lemma~\ref{lbo},(ii), $\sqx G{{\goth u}^{k}}$ is a desingularization of 
${\cal N}_{\n}^{(k)}$ so that $\k[{\cal N}_{\n}^{(k)}]$ is the space of global 
sections of $\an {\sqx G{{\goth u}^{k}}}{}$ by Lemma~\ref{lint}. Denoting by $\pi _{0}$
the bundle projection of $\sqx G{{\goth u}^{k}}$, 
$$ {\pi _{0}}_{*}(\an {\sqx G{{\goth u}^{k}}}{}) = 
{\cal L}_{0}(\overline{\es S{{\goth u}_{-}^{k}}}),$$
whence the assertion.

(iii) Since $\sqxx k{{\goth b}^{k}}$ is isomorphic to $(\sqx G{{\goth b}})^{k}$, 
$$ {\mathrm {H}}^{0}((G/B)^{k},\an {\sqxx k{{\goth b}^{k}}}{}) = 
{\mathrm {H}}^{0}(G/B,\an {\sqx G{{\goth b}}}{})^{\tens k} .$$
By (i), 
$$H^{0}(G/B,\an {\sqx G{{\goth b}}}{})= 
{\mathrm {H}}^{0}(G/B,{\cal L}(\es S{{\goth b}_{-}}) = \k[{\cal X}]$$ 
since $\sqx G{{\goth b}}$ is a desingularization of ${\cal X}$ by Lemma~\ref{lbo},(i) 
and (ii), whence
$$ {\mathrm {H}}^{0}((G/B)^{k},{\cal L}(\es S{{\goth b}_{-}^{k}}) = \k[{\cal X}^{k}] .$$

By definition, ${\cal B}_{\x}^{(k)}$ is a closed subvariety of ${\cal X}^{k}$.
According to Corollary~\ref{cbo3}, $\k[{\cal B}^{(k)}]$ is a subalgebra of 
$\k[{\cal B}_{\x}^{(k)}]$ having the same fraction field and 
$\k[{\cal B}_{\x}^{(k)}]$ is a finite extension of $\k[{\cal B}^{(k)}]$. Hence 
$\k[{\cal B}_{\x}^{(k)}]$ is a subalgebra of $\k[{\cal B}_{\n}^{(k)}]$. 
For $\varphi $ in $\k[{\cal B}_{\x}^{(k)}]$, $\varphi $ is the restriction to
${\cal B}_{\x}^{(k)}$ of an element $\psi $ of $\k[{\cal X}^{(k)}]$. As mentioned 
above, $\psi $ is a global section of ${\cal L}({\goth b}_{-}^{k})$. Denoting by 
$\overline{\psi }$ its restriction to the diagonal of $(G/B)^{k}$, $\overline{\psi }$
is a global section of ${\cal L}_{0}(\overline{\es S{{\goth b}_{-}^{k}}}))$ so that 
$\overline{\psi }$ is in $\k[{\cal B}_{\n}^{(k)}]$. Moreover, for all smooth point $x$ of 
${\cal B}_{\x}^{(k)}$, $\overline{\psi }(x)=\varphi (x)$. Hence $\varphi $ is in the
image of the morphism 
$$ \xymatrix{{\mathrm {H}}^{0}((G/B)^{k},{\cal L}(\es S{{\goth b}_{-}^{k}})) \ar[rr] && 
{\mathrm {H}}^{0}(G/B,{\cal L}_{0}(\overline{\es S{{\goth b}_{-}^{k}}}))} .$$
Conversely, for $\varphi $ image of $\psi $ in 
${\mathrm {H}}^{0}((G/B)^{k},{\cal L}(\es S{{\goth b}_{-}^{k}})$ by this morphism, 
$\psi $ is in $\k[{\cal X}^{k}]$ and $\varphi (x)=\psi (x)$ for all smooth point $x$
of ${\cal B}_{\x}^{(k)}$ so that $\varphi $ is the restriction of $\psi $ to 
${\cal B}_{\x}^{(k)}$. 

(iv) Let $\varphi $ be in $\k[{\cal N}^{(k)}]$. Since ${\cal N}^{(k)}$ is a closed 
subvariety of ${\goth N}_{{\goth g}}^{k}$, $\varphi $ is the restriction to 
${\cal N}^{(k)}$ of an element $\psi $ of $\k[{\goth N}_{{\goth g}}^{k}]$. As mentioned 
above, $\psi $ is a global section of ${\cal L}(\es S{{\goth u}_{-}^{k}})$. Denoting by 
$\overline{\psi }$ the restriction of $\psi $ to the diagonal of $(G/B)^{k}$, 
$\overline{\psi }$ is a global section of 
${\cal L}_{0}(\overline{\es S{{\goth u}_{-}^{k}}}))$ so that 
$\overline{\psi }$ is in $\k[{\cal N}_{\n}^{(k)}]$. Moreover, for all smooth point $x$ of 
${\cal N}^{(k)}$, $\overline{\psi }(x)=\varphi (x)$. Hence $\varphi $ is in the
image of the morphism 
$$ \xymatrix{{\mathrm {H}}^{0}((G/B)^{k},{\cal L}(\es S{{\goth u}_{-}^{k}})) \ar[rr] && 
{\mathrm {H}}^{0}(G/B,{\cal L}_{0}(\overline{\es S{{\goth u}_{-}^{k}}}))} .$$
Conversely, for $\varphi $ image of $\psi $ in 
${\mathrm {H}}^{0}((G/B)^{k},{\cal L}(\es S{{\goth u}_{-}^{k}})$ by this morphism, 
$\psi $ is in $\k[{\goth N}_{{\goth g}}^{k}]$ and $\varphi (x)=\psi (x)$ for all smooth 
point $x$ of ${\cal N}^{(k)}$ so that $\varphi $ is the restriction of $\psi $ to 
${\cal N}^{(k)}$. 
\end{proof}

Let $\poi x1{,\ldots,}{k\rg}{}{}{}$ be a basis of ${\goth h}^{k}$ verifying the following 
conditions for $j=1,\ldots,k$:
\begin{itemize}
\item [{\rm (1)}] $x_{j,l}=0$ for $l<(j-1)\rg$ and $l>j\rg$
\item [{\rm (2)}] $x_{j,j\rg}=h$,
\item [{\rm (3)}] $\poi x{j,(j-1)\rg +1}{,\ldots,}{j,j\rg-1}{}{}{}$ is a basis of the 
orthogonal complement to $h$ in ${\goth h}$,
\end{itemize}
with $x_{j,l}$ the component of $x_{l}$ on the $j$-th factor ${\goth h}$. Set: 
$$ E_{0} := \{0\}, \qquad F_{0} := {\goth b}_{-}^{k},\qquad 
E_{i} := {\mathrm {span}}(\{\poi x1{,\ldots,}{i}{}{}{}\}), \qquad 
F_{i} := {\goth b}_{-}^{k}/E_{i} $$ 
for $i=1,\ldots,k\rg$. In the $B$-module ${\goth b}_{-}$, ${\goth h}$ is the subspace of 
invariant elements and ${\goth u}_{-}$ is the quotient of ${\goth b}_{-}$ by ${\goth h}$.
So for $i=0,\ldots,k\rg$, $F_{i}$ is a $B^{k}$-module. As a matter of fact, because of 
the choice of the basis $\poi x1{,\ldots,}{k\rg}{}{}{}$, 
$F_{i}=\poi F{i,1}{ \times \cdots \times }{i,k}{}{}{}$ where 
$\poi F{i,1}{,\ldots,}{i,k}{}{}{}$ are $B$-modules quotient of ${\goth b}_{-}$ and the 
action of $B^{k}$ on $F_{i}$ is the product action. For $i=0,\ldots,k\rg$, set:
$$ A_{i} := {\mathrm {H}}^{0}(G/B,{\cal L}_{0}(\overline{\es S{F_{i}}}))
\quad  \text{and} \quad
C_{i} := {\mathrm {H}}^{0}((G/B)^{k},{\cal L}(\es S{F_{i}})) .$$
Denote by $B_{i}$ the image of $C_{i}$ by the restriction map to the diagonal of 
$(G/B)^{k}$. Then $A_{i}$, $B_{i}$, $C_{i}$ are integral graded algebras. Moreover, by 
Lemma~\ref{lint}, $A_{i}$ and $C_{i}$ are normal as spaces of global sections of 
structural sheaves of vector bundles over $G/B$ and $(G/B)^{k}$. For $i<k\rg$, the 
$B^{k}$-module $F_{i+1}$ is a quotient of the $B^{k}$-module $F_{i}$ so that 
$\es S{F_{i+1}}$ is a quotient of $\es S{F_{i}}$ as a $B^{k}$-algebra and 
${\cal L}(\es S{F_{i+1}})$ and ${\cal L}_{0}(\overline{\es S{F_{i+1}}})$ are quotients of 
${\cal L}(\es S{F_{i}})$ and ${\cal L}_{0}(\overline{\es S{F_{i}}})$ respectively, whence
morphisms of algebras
$$ \xymatrix{C_{i} \ar[rr]^{\nu _{i}} && C_{i+1}} \quad  \text{and} \quad 
\xymatrix{A_{i} \ar[rr]^{\nu _{i,0}} && A_{i+1} } .$$ 

For $i=0,\ldots,k\rg-1$ and $m$ positive integer, denote again by $x_{i+1}$ the image 
of $x_{i+1}$ in $F_{i}$ by the quotient morphism and by $J_{m,i}$ the ideal of 
$\es S{F_{i}}$ generated by $x_{i+1}^{m}$. As $x_{i+1}$ is a fixed point of the 
$B^{k}$-module $F_{i}$, $J_{m,i}$ is a $B^{k}$-submodule of $\es S{F_{i}}$.

\begin{lemma}\label{l3bo5}
Let $i=0,\ldots,k\rg-1$ and $m$ a positive integer. 

{\rm (i)} The algebra $\k[x_{i+1}]$ is canonically embedded in $A_{i}$ and $C_{i}$.

{\rm (ii)} The space ${\mathrm {H}}^{0}(G/B,{\cal L}_{0}(\overline{J_{m,i}}))$ is the 
ideal of $A_{i}$ generated by $x_{i+1}^{m}$ and the image of the canonical morphism 
$$ \xymatrix{{\mathrm {H}}^{0}(G/B,{\cal L}_{0}(\overline{J_{m,i}}))  \ar[rr] && 
{\mathrm {H}}^{0}(G/B,{\cal L}_{0}(
\overline{\tk {\k}{\es S{F_{i+1}}}{\k x_{i+1}^{m}}})) }$$ 
is equal to $\tk {\k}{\nu _{i,0}(A_{i})}{\k x_{i+1}^{m}}$.

{\rm (iii)} The space ${\mathrm {H}}^{0}((G/B)^{k},{\cal L}(J_{m,i}))$ is the 
ideal of $C_{i}$  generated by $x_{i+1}^{m}$ and the image of the canonical morphism 
$$ \xymatrix{{\mathrm {H}}((G/B)^{k},{\cal L}(J_{m,i}))  \ar[rr] && 
{\mathrm {H}}((G/B)^{k},{\cal L}(\tk {\k}{\es S{F_{i+1}}}{\k x_{i+1}^{m}})) }$$ 
is equal to $\tk {\k}{C_{i+1}}{\k x_{i+1}^{m}}$.

{\rm (iv)} Let $\poi v1{,\ldots,}{l}{}{}{}$ be in $A_{i}$ such that 
$\poi v1{,\ldots,}{l}{\nu }{i,0}{i,0}$ are linearly free over $\k$. Then 
$\poi v1{,\ldots,}{l}{}{}{}$ are linearly free over $\k[x_{i+1}]$.

{\rm (v)} Let $\poi w1{,\ldots,}{l}{}{}{}$ be in $C_{i}$ such that 
$\poi w1{,\ldots,}{l}{\nu }{i}{i}$ are linearly free over $\k$. Then 
$\poi w1{,\ldots,}{l}{}{}{}$ are linearly free over $\k[x_{i+1}]$.
\end{lemma}

\begin{proof}
(i) Since $x_{i+1}$ is a fixed point of the $B^{k}$-module $\es S{F_{i}}$,
${\cal L}(\k[x_{i+1}])$ and ${\cal L}_{0}(\overline{\k[x_{i+1}]})$ are submodules of 
${\cal L}(\es S{F_{i}})$ and ${\cal L}_{0}(\overline{\es S{F_{i}}})$ respectively. 
Moreover they are isomorphic to $\tk {\k}{\an {(G/B)^{k}}{}}{\k[x_{i+1}]}$
$\tk {\k}{\an {G/B}{}}{\k[x_{i+1}]}$ respectively, whence the assertion. 

(ii) Since $F_{i+1}$ is the quotient of $F_{i}$ by $\k x_{i+1}$, we have the exact 
sequence of $B^{k}$-modules,
$$ 0 \longrightarrow J_{m+1,i} \longrightarrow J_{m,i} \longrightarrow 
\tk {\k}{\es S{F_{i+1}}}{\k x_{i+1}^{m}} \longrightarrow 0,$$ 
whence the exact sequence of $\an {G/B}{}$-modules,
$$ 0 \longrightarrow {\cal L}_{0}(\overline{J_{m+1,i}}) \longrightarrow 
{\cal L}_{0}(\overline{J_{m,i}}) \longrightarrow 
{\cal L}_{0}(\overline{\tk {\k}{\es S{F_{i+1}}}{\k x_{i+1}^{m}}}) \longrightarrow 0,$$ 
and whence the canonical morphism 
$$ \xymatrix{{\mathrm {H}}^{0}(G/B,{\cal L}_{0}(\overline{J_{m,i}}))  \ar[rr] && 
{\mathrm {H}}^{0}(G/B,{\cal L}_{0}(
\overline{\tk {\k}{\es S{F_{i+1}}}{\k x_{i+1}^{m}}})) }.$$ 
In particular, $\tk {\k}{\nu _{i,0}(A_{i})}{\k x_{i+1}^{m}}$ is contained in its image
since the image of $ax_{i+1}^{m}$ is equal to $\nu _{i,0}(a)\tens x_{i+1}^{m}$ for all
$a$ in $A_{i}$.

Let $a$ be in ${\mathrm {H}}^{0}(G/B,{\cal L}_{0}(\overline{J_{m,i}}))$. Let 
$\poi O1{,\ldots,}{l}{}{}{}$ be a cover of $G/B$ by affine trivialization 
open subsets of the vector bundle $\sqx G{\es S{F_{i}}}$. For $j=1,\ldots,l$, 
denoting by $\Phi _{j}$ a trivialization over $O_{j}$, we have a commutative diagram:
$$ \xymatrix{ \pi _{i}^{-1}(O_{j}) \ar[rr]^{\Phi _{j}} \ar[rrd]_{\pi _{i}} && 
O_{j}\times \es S{F_{i}} \ar[d]^{\pr1} \\ && O_{j} }$$
with $\pi _{i}$ the bundle projection. Since $x_{i+1}$ is invariant under $B$, for 
$\varphi $ local section of ${\cal L}_{0}(\overline{\es S{F_{i}}})$ above $O_{j}$, 
$\Phi _{j}(x_{i+1}^{m}\varphi ) = x_{i+1}^{m}\Phi _{j}(\varphi )$, whence 
$${\Phi _{j}}_{*}(\Gamma (O_{j},{\cal L}_{0}(\overline{J_{m,i}})) = 
\tk {\k}{\k[O_{j}]}\es S{F_{i}}x_{i+1}^{m} .$$
As a result, for some local section $a_{j}$ above $O_{j}$ of 
${\cal L}_{0}(\overline{\es S{F_{i}}})$, 
$a = x_{i+1}^{m}a_{j}$. Moreover, $a_{j}$ is uniquely defined by this equality. Then 
for all $j,j'$, $a_{j}$ and $a_{j'}$ have the same restriction to $O_{j}\cap O_{j'}$. 
Denoting by $a'$ the global section of ${\cal L}_{0}(\overline{\es S{F_{i}}})$ 
extending $\poi a1{,\ldots,}{l}{}{}{}$, $a=a'x_{i+1}^{m}$, whence the assertion.

(iii) According to~\cite[Theorem B and Corollary]{He}, for a $B$-module quotient $V$ of 
${\goth b}_{-}$ having ${\goth u}_{-}$ as quotient, 
${\mathrm {H}}^{1}(G/B,{\cal L}_{0}(V))=0$. Hence $C_{i+1}$ is the image of $\nu _{i}$. 
From the exact sequence of $(G/B)^{k}$-modules
$$ 0 \longrightarrow J_{m+1,i} \longrightarrow J_{m,i} \longrightarrow 
\tk {\k}{\es S{F_{i+1}}}{\k x_{i+1}^{m}} \longrightarrow 0,$$ 
we deduce the exact sequence of $\an {(G/B)^{k}}{}$-modules,
$$ 0 \longrightarrow {\cal L}(J_{m+1,i}) \longrightarrow {\cal L}(J_{m,i}) 
\longrightarrow {\cal L}(\tk {\k}{\es S{F_{i+1}}}{\k x_{i+1}^{m}}) \longrightarrow 0,$$ 
and whence the canonical morphism 
$$ \xymatrix{{\mathrm {H}}^{0}((G/B)^{k},{\cal L}(J_{m,i}))  \ar[rr] && 
{\mathrm {H}}^{0}((G/B)^{k},{\cal L}(\tk {\k}{\es S{F_{i+1}}}{\k x_{i+1}^{m}})) }.$$ 
In particular, $\tk {\k}{C_{i+1}}{\k x_{i+1}^{m}}$ is its image
since the image of $ax_{i+1}^{m}$ is equal to $\nu _{i}(a)\tens x_{i+1}^{m}$ for all
$a$ in $C_{i}$ and $C_{i+1}$ is the image of $\nu _{i}$.

Let $a$ be in ${\mathrm {H}}^{0}(G/B,{\cal L}(J_{m,i}))$. Prove by induction on 
$l$ that for some $a_{l}$ in $C_{i}$, $a-a_{l}x_{i+1}^{m}$ is in 
${\mathrm {H}}^{0}(G/B,{\cal L}(J_{m+l,i}))$. It is true for $l=0$. Suppose that it is
true for $l$. By the above result for $m+l$, for some $a'_{l}$ in $C_{i}$, 
$a-a_{l}x_{i+1}^{m}-a'_{l}x_{i+1}^{m+l}$ is in 
${\mathrm {H}}^{0}(G/B,{\cal L}(J_{m+l+1,i}))$, whence the statement. As $C_{i}$ is a 
graded algebra and ${\mathrm {H}}^{0}(G/B,{\cal L}(J_{m+l,i}))$ is a graded subspace whose
elements have degree at least $m+l$, for $l$ sufficiently big, $a=a_{l}x_{i+1}^{m}$, 
whence the assertion.

(iv) Suppose that $\poi v1{,\ldots,}{l}{}{}{}$ is not linearly free over 
$\k[x_{i+1}]$. A contradiction is expected. Let $\poi a1{,\ldots,}{l}{}{}{}$ be in
$\k[x_{i+1}]$ such that $a_{i}\neq 0$ for some $i$ and
$$ a_{1}v_{1} + \cdots + a_{l} v_{l} = 0 .$$
Suppose that $m$ is the biggest integer such that $x_{i+1}^{m}$ divides 
$\poi a1{,\ldots,}{l}{}{}{}$ in $\k[x_{i+1}]$. For $j=1,\ldots,l$, denote by 
$c_{j}$ the element of $\k$ such that $x_{i+1}^{m+1}$ divides $a_{j}-c_{j}x_{i+1}^{m}$.
Then by (ii), 
$$ c_{1}\nu _{i,0}(v_{1}) + \cdots + c_{l}\nu _{i,0}(v_{l}) = 0,$$
whence a contradiction by the maximality of $m$.  

(v) Suppose that $\poi w1{,\ldots,}{l}{}{}{}$ is not linearly free over 
$\k[x_{i+1}]$. A contradiction is expected. Let $\poi a1{,\ldots,}{l}{}{}{}$ be in
$\k[x_{i+1}]$ such that $a_{i}\neq 0$ for some $i$ and
$$ a_{1}w_{1} + \cdots + a_{l} w_{l} = 0 .$$
Suppose that $m$ is the biggest integer such that $x_{i+1}^{m}$ divides 
$\poi a1{,\ldots,}{l}{}{}{}$ in $\k[x_{i+1}]$. For $j=1,\ldots,l$, denote by 
$c_{j}$ the element of $\k$ such that $x_{i+1}^{m+1}$ divides $a_{j}-c_{j}x_{i+1}^{m}$.
Then by (iii), 
$$ c_{1}\nu _{i}(w_{1}) + \cdots + c_{l}\nu _{i}(w_{l}) = 0,$$
whence a contradiction by the maximality of $m$.  
\end{proof}

\begin{coro}\label{cbo5}
Let $i=0,\ldots,k\rg -1$.

{\rm (i)} The algebra $A_{i}$ is a free extension of $\k[x_{i+1}]$ and 
$\nu _{i,0}(A_{i})$ is the quotient of $A_{i}$ by the ideal generated by $x_{i+1}$.

{\rm (ii)} The algebra $C_{i}$ is a free extension of $\k[x_{i+1}]$ and $C_{i+1}$ is the
quotient of $C_{i}$ by the ideal generated by $x_{i+1}$.

{\rm (iii)} The algebra $B_{i}$ is a free extension of $\k[x_{i+1}]$ and $B_{i+1}$ is 
the quotient of $B_{i}$ by the ideal generated by $x_{i+1}$.
\end{coro}

\begin{proof}
(i) Let $K_{0}$ be a $\k$-subspace of $A_{i}$ such that the restriction of $\nu _{i,0}$ to
$K_{0}$ is an isomorphism onto the $\k$-space $\nu _{i,0}(A_{i})$. Prove by induction on 
$m$ the equality 
$$A_{i} = K_{0}\k[x_{i+1}] + {\mathrm {H}}^{0}(G/B,{\cal L}_{0}(\overline{J_{m,i}})) .$$
The equality is true for $m=0$. Suppose that it is true for $m$. Let $a$ be in 
${\mathrm {H}}^{0}(G/B,{\cal L}_{0}(\overline{J_{m,i}}))$. By Lemma~\ref{l3bo5},(ii), for
some $b$ in $A_{i}$, $a-bx_{i+1}^{m}$ is in 
${\mathrm {H}}^{0}(G/B,{\cal L}_{0}(\overline{J_{m+1,i}}))$, whence the equality. Since 
$A_{i}$ is graded with $x_{i+1}$ having degree $1$, $A_{i}=K_{0}\k[x_{i+1}]$. So $A_{i}$
is a free $\k[x_{i+1}]$-module by Lemma~\ref{l3bo5},(iv). Again by 
Lemma~\ref{l3bo5},(ii), $\nu _{i,0}(A_{i})$ is the quotient of $A_{i}$ by the 
ideal generated by $x_{i+1}$.

(ii) Let $K$ be a $\k$-subspace of $C_{i}$ such that the restriction of $\nu _{i}$ to
$K$ is an isomorphism onto the $\k$-space $C_{i+1}$. Prove by induction on $m$ 
the equality 
$$C_{i} = K\k[x_{i+1}] + {\mathrm {H}}^{0}((G/B)^{k},{\cal L}(J_{m,i})) .$$
The equality is true for $m=0$. Suppose that it is true for $m$. Let $a$ be in 
${\mathrm {H}}^{0}((G/B)^{k},{\cal L}(J_{m,i})$. By Lemma~\ref{l3bo5},(iii), for some 
$b$ in $C_{i}$, $a-bx_{i+1}^{m}$ is in 
${\mathrm {H}}^{0}((G/B)^{k},{\cal L}(J_{m+1,i}))$, 
whence the equality. Since $C_{i}$ is graded with $x_{i+1}$ having degree $1$, 
$C_{i}=K\k[x_{i+1}]$. So $C_{i}$ is a free $\k[x_{i+1}]$-module by 
Lemma~\ref{l3bo5},(v). Again by Lemma~\ref{l3bo5},(iii), $C_{i+1}$ is the quotient of 
$C_{i}$ by the ideal generated by $x_{i+1}$. 
 
(iii)  We have the commutative diagram
$$ \xymatrix{ C_{i} \ar[rr]^{\nu _{i}} \ar[d] && C_{i+1} \ar[d] \\
B_{i} \ar[rr]^{\nu '_{i,0}} && B_{i+1}}$$
where the vertical arrows are the restriction morphisms to the diagonal of $(G/B)^{k}$
and $\nu '_{i,0}$ is the restriction of $\nu _{i,0}$ to $B_{i}$. In particular 
$\nu '_{i,0}$ is surjective since so is $\nu _{i}$. Let $K'_{0}$ be a 
$\k$-subspace of $B_{i}$ such that the restriction of $\nu '_{i,0}$ to
$K'_{0}$ is an isomorphism onto the $\k$-space $B_{i+1}$. For $m$ positive integer, 
denote by ${\goth J}_{m,i}$ the image of the restriction morphism to the diagonal of 
$(G/B)^{k}$,
$$ \xymatrix{ {\mathrm {H}}^{0}((G/B)^{k},{\cal L}(J_{m,i})) \ar[rr] &&  B_{i}} .$$ 
Prove by induction on $m$ the equality
$$B_{i} = K'_{0}\k[x_{i+1}] + {\goth J}_{m,i} .$$
The equality is true for $m=0$. Suppose that it is true for $m$. Let $a$ be in 
${\goth J}_{m,i}$. Then $a$ is the image of an element $a'$ of 
${\mathrm {H}}^{0}((G/B)^{k},{\cal L}(J_{m,i}))$. By Lemma~\ref{l3bo5},(iii), for some 
$b'$ in $C_{i}$, $a'-b'x_{i+1}^{m}$ is in 
${\mathrm {H}}^{0}((G/B)^{k},{\cal L}(J_{m+1,i}))$. Denoting by $b$ the image of $b'$ in 
$B_{i}$, $a-bx_{i+1}^{m}$ is in ${\goth J}_{m+1,i}$, whence the equality. Since 
$B_{i}$ is graded with $x_{i+1}$ having degree $1$, $B_{i}=K'_{0}\k[x_{i+1}]$. So $B_{i}$
is a free $\k[x_{i+1}]$-module by Lemma~\ref{l3bo5},(iv).  

Since $B_{i+1}$ is contained in $\nu _{i,0}(A)$, $K'_{0}$ can be chosen contained in 
$K_{0}$. Let $v_{l},l\in L$ be a basis of $K_{0}$ such that $v_{l},l\in L'$ is a basis of
$K'_{0}$ for some subset $L'$ of $L$. Let $a$ be in the kernel of $\nu '_{i,0}$. Then $a$
is in the kernel of $\nu _{i,0}$ so that $a=bx_{i+1}$ for some $b$ in $A_{i}$ by 
Lemma~\ref{l3bo5},(ii). By (i) and the freeness of the extension $B_{i}$ of 
$\k[x_{i+1}]$, 
$$b = \sum_{l\in L} v_{l} p_{l} \quad  \text{and} \quad 
a = \sum_{l \in L'} v_{l} q_{l}$$ 
with $p_{l},l\in L$ and $q_{l},l\in L'$ in $\k[x_{i+1}]$ with finite supports. Then 
$q_{l}=p_{l}x_{i+1}$ for all $l$ in $L'$ and $p_{l}=0$ for $l$ in $L\setminus L'$ so 
that $a$ is in $B_{i}x_{i+1}$, whence the assertion.
\end{proof}

For $j=0,\ldots,k\rg$, set $F^{*}_{j} := {\mathrm {Specm}}(\es S{F_{j}})$. By definition,
$F_{j}^{*}$ is the subspace of elements $(\poi y1{,\ldots,}{k}{}{}{})$
of ${\goth b}^{k}$ such that for $m=1,\ldots,k$ and $l=1,\ldots,j$, 
$\dv {y_{m}}{x_{m,l}}=0$. 

\begin{lemma}\label{l4bo5}
Let $i=0,\ldots,k\rg -1$. Denote by $T$ the annihilator of $x_{i+1}$ in the 
$A_{i}$-module ${\mathrm {H}}^{1}(G/B,{\cal L}_{0}(\overline{\es S{F_{i}}}))$.

{\rm (i)} The algebra $A_{i}$ is the integral closure of $B_{i}$ in its fraction field.

{\rm (ii)} There is a well defined morphism $\mu : \xymatrix{ T \ar[r] & A_{i+1}}$
of $A_{i}$-modules.

{\rm (iii)} The $A_{i}$-module $A_{i+1}$ is the direct sum of $\mu (T)$ and 
$\nu _{i,0}(A_{i})$.
\end{lemma}

\begin{proof}
Since $A_{i}$ is the space of global sections of ${\cal L}_{0}(\overline{\es S{F_{i}}})$, 
for all integer $m$, ${\mathrm {H}}^{m}(G/B,{\cal L}_{0}(\overline{\es S{F_{i}}}))$ is 
a $A_{i}$-module.

(i) According to the proof of Lemma~\ref{lbo}, $\Omega _{{\goth g}}^{(k)}$ is an open
subset of ${\goth g}^{k}$ such that for all $x$ in 
$\Omega _{{\goth g}}^{(k)}\cap {\cal B}^{(k)}$, there exists only one Borel subalgebra of
${\goth g}$ containing $E_{x}$. By definition, $F_{i}^{*}$ is the subspace of elements 
$(\poi y1{,\ldots,}{k}{}{}{})$ of ${\goth b}^{k}$ such that for $j=1,\ldots,k$ and 
$l=1,\ldots,i$, $\dv {y_{j}}{x_{l,j}}=0$. By Lemma~\ref{l4int}, $G.F_{i}^{*}$ is 
a closed subset of ${\goth g}^{k}$ and the morphism 
$\xymatrix{ \sqx G{F_{i}^{*}} \ar[r] & G.F_{i}^{*}}$ is projective. By Conditions 
(1), (2), (3), for some $\poi y2{,\ldots,}{k-1}{}{}{}$ in ${\goth b}$, 
$(e,\poi y2{,\ldots,}{k-1}{}{}{},h)$ is in $F_{i}^{*}$. Hence 
$\Omega _{{\goth g}}^{(k)}\cap G.F_{i}^{*}$ is a dense open subset 
of $G.F_{i}^{*}$ and the above morphism is birational. So, by Lemma~\ref{lint}, 
$A_{i}$ is the integral closure of $\k[G.F_{i}^{*}]$ in its fraction field. 

By Lemma~\ref{l4int}, $G^{k}.F_{i}^{*}$ is a closed subset of ${\goth g}^{k}$ containing 
$G.F_{i}^{*}$. Then for all $\varphi $ in $\k[G.F_{i}^{*}]$, $\varphi $ is the 
restriction to $G.F_{i}^{*}$ of an element $\psi $ of $\k[G^{k}.F_{i}^{*}]$ so that 
$\psi $ is a global section of ${\cal L}(\es S{F_{i}^{*}})$. Denoting by 
$\overline{\psi }$ the restriction of $\psi $ to the diagonal of $(G/B)^{k}$, 
$\varphi =\overline{\psi }$. Hence $\k[G.F_{i}^{*}]$ is contained in $B_{i}$ so that 
$A_{i}$ is the integral closure of $B_{i}$ in its fraction field. 

(ii) Let $O := \poi O1{,\ldots,}{m}{}{}{}$ be a cover of $G/B$ by open subsets isomorphic
to the affine space of dimension $n$ so that $O_{i}$ is a trivialization open subset of 
the vector bundles $\sqx G{\overline{\es S{F_{i}}}}$ and 
$\sqx G{\overline{\es S{F_{i+1}}}}$. Denote by ${\cal Z}^{1}$ the space of cocycles of 
degree $1$ of the complex $C^{\bullet}$ of $\check{{\mathrm {C}}}$ech cohomology of $O$ 
with values in ${\cal L}_{0}(\overline{\es S{F_{i}}})$. 

Let $\overline{a}$ be in $T$ and $a$ a representative of $\overline{a}$ in 
${\cal Z}^{1}$. Since $\overline{a}$ is in $T$, $x_{i+1}a$ is the boundary of an element 
$b$ of $C^{0}$. For $l=1,\ldots,m$, denote by $b_{l}$ the component of $b$ in 
$\Gamma (O_{l},{\cal L}_{0}(\overline{\es S{F_{i}}}))$ and by 
$\widetilde{b_{l}}$ its image in $\Gamma (O_{l},{\cal L}_{0}(\overline{\es S{F_{i+1}}}))$
by the quotient morphism. Then for $1\leq l,l'\leq m$, $\widetilde{b_{l}}$ and 
$\widetilde{b_{l'}}$ have the same restriction to $O_{l}\cap O_{l'}$ so that 
$\widetilde{b_{l}}$ is the restriction to $O_{l}$ of an element $\widetilde{b}$ of 
$A_{i+1}$. If $a'$ is a representative of $\overline{a}$ in ${\cal Z}^{1}$, $a'-a$ is the 
boundary of an element $b'$ in $C^{0}$ and $x_{i+1}a$ is the boundary of $b+x_{i+1}b'$. 
Hence $\widetilde{b}$ only depends on $\overline{a}$, whence a well defined map 
$\xymatrix{T \ar[r] & A_{i+1}}$. It is clearly a morphism of $A_{i}$-modules.

(iii) From the exact sequence of $\an {G/B}{}$-modules
$$ 0 \longrightarrow {\cal L}_{0}(\overline{x_{i+1}\es S{F_{i}}}) \longrightarrow 
{\cal L}_{0}(\overline{\es S{F_{i}}}) \longrightarrow 
{\cal L}_{0}(\overline{\es S{F_{i+1}}}) \longrightarrow 0 $$
we deduce the long exact sequence of cohomology
$$ \cdots \longrightarrow A_{i} \longrightarrow A_{i+1} 
\longrightarrow x_{i+1}{\mathrm {H}}^{1}(G/B,{\cal L}_{0}(\overline{\es S{F_{i}}})) 
\longrightarrow {\mathrm {H}}^{1}(G/B,{\cal L}_{0}(\overline{\es S{F_{i}}})) 
\longrightarrow \cdots $$
since $x_{i+1}$ is a global section of ${\cal L}_{0}(\overline{\es S{F_{i}}})$ by 
Lemma~\ref{l3bo5},(i). Since the $\an {G/B}{}$-modules 
${\cal L}_{0}(\overline{x_{i+1}\es S{F_{i}}})$ and 
${\cal L}_{0}(\overline{\es S{F_{i}}})$ are isomorphic, we have an isomorphism 
$$ \xymatrix{ {\mathrm {H}}^{1}(G/B,{\cal L}_{0}(\overline{x_{i+1}\es S{F_{i}}}))
\ar[rr] && {\mathrm {H}}^{1}(G/B,{\cal L}_{0}(\overline{\es S{F_{i}}}))}$$
and the image of the kernel of the arrow
$$ x_{i+1}{\mathrm {H}}^{1}(G/B,{\cal L}_{0}(\overline{\es S{F_{i}}})) 
\longrightarrow {\mathrm {H}}^{1}(G/B,{\cal L}_{0}(\overline{\es S{F_{i}}})) $$
by this isomorphism is equal to $T$, whence an exact sequence
$$ A_{i} \longrightarrow A_{i+1} \longrightarrow T \longrightarrow 0 .$$
By (ii), from the definition of the arrow
$$ {\mathrm {H}}^{0}(G/B,{\cal L}_{0}(\overline{\es S{F_{i+1}}})) 
\longrightarrow x_{i+1}{\mathrm {H}}^{1}(G/B,{\cal L}_{0}(\overline{\es S{F_{i}}})) $$
we deduce that for $a$ in $T$, the image of $\mu (a)$ is equal to $a$. Hence 
$A_{i+1}$ is the direct sum of $\mu (T)$ and $\nu _{i,0}(A_{i})$.
\end{proof}

\begin{coro}\label{c2bo5}
For $i=0,\ldots,k\rg -1$, $\nu _{i,0}(A_{i})$ is equal to $A_{i+1}$.
\end{coro}

\begin{proof}
According to Lemma~\ref{l4bo5}, $A_{i+1}$ is the direct sum of $\mu (T)$ and 
$\nu _{i,0}(A_{i})$. Let $a$ be in $\mu (T)$. By Lemma~\ref{l4bo5},(i) $A_{i+1}$ and 
$\nu _{i,0}(A_{i})$ have the same fraction field since $B_{i+1}$ is contained in 
$\nu _{i,0}(A_{i})$ by the proof of Corollary~\ref{cbo5},(iii). So for some $b$ in 
$\nu _{i,0}(A_{i})$, $ba$ is in $\nu _{i,0}(A_{i})$, whence $ba=0$ by 
Lemma~\ref{l4bo5},(ii) and (iii). As a result, $\mu (T)=\{0\}$ and 
$\nu _{i,0}(A_{i})=A_{i+1}$.
\end{proof}

\begin{prop}\label{pbo5}
{\rm (i)} The algebra $\k[{\cal B}_{\n}^{(k)}]$ is a free extension of 
$\es S{{\goth h}^{k}}$ and $\k[{\cal N}_{\n}^{(k)}]$ is the quotient of 
$\k[{\cal B}_{\n}^{(k)}]$ by the ideal generated by $\ec S{}{{\goth h}^{k}}{}+$.

{\rm (ii)} The algebra $\k[{\cal B}_{\x}^{(k)}]$ is a free extension of 
$\es S{{\goth h}^{k}}$ and $\k[{\cal N}^{(k)}]$ is the quotient of 
$\k[{\cal B}_{\x}^{(k)}]$ by the ideal generated by $\ec S{}{{\goth h}^{k}}{}+$.
\end{prop}

\begin{proof}
(i) According to Lemma~\ref{l2bo5},(i) and (ii), $A_{0}=\k[{\cal B}_{\n}^{(k)}]$ and 
$A_{k\rg}=\k[{\cal N}_{\n}^{(k)}]$. Moreover, $\es S{{\goth h}^{k}} = 
\k[\poi x1{,\ldots,}{k\rg}{}{}{}]$ by definition. So the assertion will result from the 
following claim:

\begin{claim}
For $i=1,\ldots,k\rg$, $A_{0}$ is a free extension of $\k[\poi x1{,\ldots,}{i}{}{}{}]$ 
and $A_{i}$ is the quotient of $A_{0}$ by the ideal generated by 
$\poi x1{,\ldots,}{i}{}{}{}$.
\end{claim}

Prove the claim by induction on $i$. According to Corollary~\ref{cbo5},(i) and 
Corollary~\ref{c2bo5}, the claim is true for $i=1$. Suppose that it is true for $i$. By 
Corollary~\ref{cbo5},(i) and Corollary~\ref{c2bo5}, $A_{i+1}$ is the quotient of $A_{i}$
by the ideal generated by $x_{i+1}$. So by induction hypothesis, $A_{i+1}$ is the 
quotient of $A_{0}$ by the ideal generated by $\poi x1{,\ldots,}{i+1}{}{}{}$. For 
$j=1,\ldots,i+1$, denote by $\mu _{j}$ the quotient morphism 
$\xymatrix{A_{0} \ar[r] & A_{j}}$. Let $K_{i+1}$ be a $\k$-subspace
of $A_{0}$ such that the restriction of $\mu _{i+1}$ to $K_{i+1}$ is a $\k$-linear 
isomorphism onto $A_{i+1}$. Then $A_{0}=K_{i+1}+A_{0}x_{1} + \cdots + A_{0}x_{i+1}$. So 
by induction on $m$,
$$ A_{0} = K_{i+1}\k[\poi x1{,\ldots,}{i+1}{}{}{}] + A_{0}{\goth I}_{m}$$ 
with ${\goth I}_{m}$ the ideal of $\k[\poi x1{,\ldots,}{i+1}{}{}{}]$ generated by its
monomials of degree $m$. As a result, $A_{0}=K_{i+1}\k[\poi x1{,\ldots,}{i+1}{}{}{}]$ 
since $A_{0}$ is a graded algebra.

Let $\poi v1{,\ldots,}{l}{}{}{}$ be linearly free over $\k$ in $K_{i+1}$ and let
$\poi a1{,\ldots,}{l}{}{}{}$ be in $\k[\poi x1{,\ldots,}{i+1}{}{}{}]$ such that
$$ a_{1}v_{1} + \cdots a_{l}v_{l} = 0 .$$  
For $j=1,\ldots,l$, $a_{j}$ has an expansion
$$ a_{j} = \sum_{m\in {\Bbb N}} a_{j,m} x_{i+1}^{m} $$
with $a_{j,m},m\in {\Bbb N}$ in $\k[\poi x1{,\ldots,}{i}{}{}{}]$ with finite support.
According to Corollary~\ref{cbo5},(i), the sequence 
$x_{i+1}^{m}\mu _{i}(v_{j}),(j,m)\in \{1,\ldots,l\}\times {\Bbb N}$ is linearly free over
$\k$. So, by induction hypothesis, $a_{j,m}=0$ for all $(j,m)$. As a result, 
$A_{0}$ is a free extension of $\k[\poi x1{,\ldots,}{i+1}{}{}{}]$, whence the claim.

(ii) According to Lemma~\ref{l2bo5},(iii) and (iv), $B_{0}=\k[{\cal B}_{\x}^{(k)}]$ and 
$B_{k\rg}=\k[{\cal N}^{(k)}]$. So the assertion will result from the 
following claim:

\begin{claim}
For $i=1,\ldots,k\rg$, $B_{0}$ is a free extension of $\k[\poi x1{,\ldots,}{i}{}{}{}]$ 
and $B_{i}$ is the quotient of $B_{0}$ by the ideal generated by 
$\poi x1{,\ldots,}{i}{}{}{}$.
\end{claim}

Prove the claim by induction on $i$. According to Corollary~\ref{cbo5},(iii), 
the claim is true for $i=1$. Suppose that it is true for $i$. By 
Corollary~\ref{cbo5},(iii), $B_{i+1}$ is the quotient of $B_{i}$ by the ideal generated 
by $x_{i+1}$. So by induction hypothesis, $B_{i+1}$ is the quotient of $B_{0}$ by the 
ideal generated by $\poi x1{,\ldots,}{i+1}{}{}{}$. For $j=0,\ldots,k\rg$, $B_{j}$ is
contained in $A_{j}$ and the quotient morphism $\xymatrix{B_{0} \ar[r] & B_{j}}$ is the
restriction of $\mu _{j}$ to $B_{0}$. Let $K'_{i+1}$ be a $\k$-subspace of $B_{0}$ 
such that the restriction of $\mu _{i+1}$ to $K'_{i+1}$ is a $\k$-linear isomorphism onto 
$B_{i+1}$. Then $B_{0}=K'_{i+1}+B_{0}x_{1} + \cdots + B_{0}x_{i+1}$. So by induction on 
$m$,
$$ B_{0} = K'_{i+1}\k[\poi x1{,\ldots,}{i+1}{}{}{}] + B_{0}{\goth I}_{m} .$$ 
As a result, $B_{0}=K'_{i+1}\k[\poi x1{,\ldots,}{i+1}{}{}{}]$ since $B_{0}$ is a graded 
algebra. Moreover by (i), a basis of $K'_{i+1}$ is linearly free over 
$\k[\poi x1{,\ldots,}{i+1}{}{}{}]$ so that $B_{0}$ is a free extension of 
$\k[\poi x1{,\ldots,}{i+1}{}{}{}]$, whence the claim.   
\end{proof}

\begin{rema}\label{rbo5}
According to Proposition~\ref{pbo5}, $\es S{{\goth h}^{k}}$ is embedded in 
$\k[{\cal B}_{\x}^{(k)}]$ and by Lemma~\ref{l2bo5},(iii), the embedding is given by
the map
$$ \xymatrix{\es S{{\goth h}^{k}} \ar[rr] && \k[{\cal B}_{\x}^{(k)}]}, 
\qquad  p \longmapsto ((\poi x1{,\ldots,}{k}{}{}{},\poi y1{,\ldots,}{k}{}{}{})
\mapsto p(\poi y1{,\ldots,}{k}{}{}{}) .$$
Denote by $\Phi $ this map.
\end{rema}

\begin{coro}\label{c3bo5}
{\rm (i)} The image of $\Phi $ is equal to $\k[{\cal B}_{\x}^{(k)}]^{G}$. Moreover,
$\k[{\cal B}_{\x}^{(k)}]$ is generated by $\k[{\cal B}^{(k)}]$ and 
$\k[{\cal B}_{\x}^{(k)}]^{G}$.

{\rm (ii)} The image of $\Phi $ is equal to $\k[{\cal B}_{\n}^{(k)}]^{G}$.

{\rm (iii)} The subalgebras $\k[{\cal B}^{(k)}]^{G}$ and 
$\Phi ((\e Sh^{\tens k})^{W({\cal R})})$ of $\k[{\cal B}_{\x}^{(k)}]^{G}$ are equal.
\end{coro}

\begin{proof}
(i) Since ${\cal B}_{\x}^{(k)}$ is a closed subvariety of ${\cal X}^{k}$ and
$\k[{\cal X}]$ is generated by $\e Sg$ and $\e Sh$, $\k[{\cal B}_{\x}^{(k)}]$ is 
generated by $\es S{{\goth h}^{k}}$ and the image of $\es S{{\goth g}^{k}}$ in 
$\k[{\cal B}_{\x}^{(k)}]$ which is equal to $\k[{\cal B}^{(k)}]$. For $p$ in 
$\k[{\cal B}_{\x}^{(k)}]$, denote by $\overline{p}$ the element of 
$\e Sh^{\tens k}$ such that
$$ \overline{p}(\poi x1{,\ldots,}{k}{}{}{}) := 
p(\poi x1{,\ldots,}{k}{}{}{},\poi x1{,\ldots,}{k}{}{}{}) .$$
Then the restriction of $p-\Phi (\overline{p})$ to $\iota _{k}({\goth h}^{k})$ is equal 
to $0$. Moreover, if $p$ is in $\k[{\cal B}_{\x}^{(k)}]^{G}$, $p-\Phi (\overline{p})$ 
is $G$-invariant so that $p=\Phi (\overline{p})$ since $G.\iota _{k}({\goth h}^{k})$
is dense in ${\cal B}_{\x}^{(k)}$, whence the assertion.

(ii) Since $\k[{\cal B}_{\x}^{(k)}]$ is contained in $\k[{\cal B}_{\n}^{(k)}]$, 
so is $\k[{\cal B}_{\x}^{(k)}]^{G}$ by (i). Since $G$ is reductive, there exists a 
projection $\xymatrix{ \k[{\cal B}_{\n}^{(k)}] \ar[r] & \k[{\cal B}_{\n}^{(k)}]^{G}}$ 
which is $\k[{\cal B}_{\n}^{(k)}]^{G}$-linear. As a result, $\k[{\cal B}_{\n}^{(k)}]^{G}$
is the integral extension of $\k[{\cal B}_{\x}^{(k)}]^{G}$ in $\k[{\cal B}_{\n}^{(k)}]$
since $\k[{\cal B}_{\n}^{(k)}]$ is an integral extension of $\k[{\cal B}_{\x}^{(k)}]$. 
Let $J$ be the ideal of augmentation of $\k[{\cal B}_{\x}^{(k)}]^{G}$ and set 
$J':= \k[{\cal B}_{\n}^{(k)}]J$. By (i) and Proposition~\ref{pbo5},(i), $J'$ is a prime 
ideal. Suppose that $\k[{\cal B}_{\x}^{(k)}]^{G}$ is strictly contained in 
$\k[{\cal B}_{\n}^{(k)}]^{G}$. A contradiction is expected. The algebras 
$\k[{\cal B}_{\n}^{(k)}]^{G}$ and $\k[{\cal B}_{\x}^{(k)}]^{G}$ are graded subalgebras of
$\k[{\cal B}_{\n}^{(k)}]$. Let $a$ be a homogeneous element in 
$\k[{\cal B}_{\n}^{(k)}]^{G}\setminus \k[{\cal B}_{\x}^{(k)}]^{G}$ of minimal degree. 
Then $a$ has positive degree. As a result, it is in $J'$ since $J'$ is radical and 
$a$ satifies a dependence integral equation over $\k[{\cal B}_{\x}^{(k)}]^{G}$. Since 
$\k[{\cal B}_{\n}^{(k)}]^{G}J$ is the image of $J'$ by the projection 
$\xymatrix{ \k[{\cal B}_{\n}^{(k)}] \ar[r] & \k[{\cal B}_{\n}^{(k)}]^{G}}$, $a$ is in
$\k[{\cal B}_{\n}^{(k)}]^{G}J$. By the minimality of the degree of $a$, $a$ is in 
$\k[{\cal B}_{\x}^{(k)}]^{G}$, whence the contradiction.

(iii) For $(\poi x1{,\ldots,}{k}{}{}{})$ in ${\goth h}^{k}$, for $w$ in 
$W({\cal R})$ and for $g_{w}$ a representative of $w$ in $N_{G}({\goth h})$, we have
$$ (\poi x1{,\ldots,}{k}{w}{}{},\poi x1{,\ldots,}{k}{w}{}{}) =
g_{w}.(\poi x1{,\ldots,}{k}{}{}{},\poi x1{,\ldots,}{k}{w}{}{})$$
so that the subalgebra $\k[{\cal B}^{(k)}]^{G}$ of $\k[{\cal B}_{\x}^{(k)}]^{G}$
is contained in $\Phi ((\es S{{\goth h}^{k}})^{W({\cal R})})$ by (i). Moreover, 
since $G$ is reductive, $\k[{\cal B}^{(k)}]^{G}$ is the image of $(\e Sg^{\tens k})^{G}$ 
by the restriction morphism. According to~\cite[Theorem 2.9 and some remark]{J}, the 
restriction morphism $(\e Sg^{\tens k})^{G}\rightarrow (\e Sh^{\tens k})^{W({\cal R})}$ 
is surjective, whence the equality 
$\k[{\cal B}^{(k)}]^{G}=\Phi ((\es S{{\goth h}^{k}})^{W({\cal R})})$. 
\end{proof}

According to Proposition~\ref{pbo5},(ii) and Corollary~\ref{c3bo5},(i), 
$\k[{\cal B}_{\x}^{(k)}]$ is a free extension of \sloppy \hbox{
$\k[{\cal B}_{\x}^{(k)}]^{G}=\es S{{\goth h}^{k}}$}.

\begin{coro}\label{c4bo5}
Let $M$ be a graded complement to $\k[{\cal B}^{(k)}]_{+}^{G}\k[{\cal B}^{(k)}]$ in 
$\k[{\cal B}^{(k)}]$.

{\rm (i)} The space $M$ contains a basis of $\k[{\cal B}_{\x}^{(k)}]$ over 
$\e Sh^{\tens k}$.

{\rm (ii)} The intersection of $M$ and 
$\ec S{}{{\goth h}^{k}}{}+\k[{\cal B}_{\x}^{(k)}]$ is different from $\{0\}$. 
\end{coro}

\begin{proof}
(i) Since $M$ is a graded complement to 
$\k[{\cal B}^{(k)}]_{+}^{G}\k[{\cal B}^{(k)}]$ in $\k [{\cal B}^{(k)}]$, by induction on 
$l$,
$$ \k[{\cal B}^{(k)}] = M\k[{\cal B}^{(k)}]^{G} + (\k[{\cal B}^{(k)}]^{G}_{+})^{l}
\k[{\cal B}^{(k)}] .$$
Hence $\k[{\cal B}^{(k)}]=M\k[{\cal B}^{(k)}]^{G}$ since $\k[{\cal B}^{(k)}]$ is graded. 
Then, by Corollary~\ref{c3bo5},(i) and (iii), 
$$\k[{\cal B}_{\x}^{(k)}]=M\e Sh^{\tens k} \quad  \text{and} \quad 
\k[{\cal B}_{\x}^{(k)}] = M + \ec S{}{{\goth h}^{k}}{}+
\k[{\cal B}_{\x}^{(k)}] .$$
Then $M$ contains a graded complement $M'$ to 
$\ec S{}{{\goth h}^{k}}{}+\k[{\cal B}_{\x}^{(k)}]$ in 
$\k[{\cal B}_{\x}^{(k)}]$, whence the assertion.

(ii) Suppose that $M'=M$. We expect a contradiction. According to (i), the 
canonical maps
$$ \xymatrix{\tk {\k}M{\e Sh^{\tens k}} \ar[rr] && \k[{\cal B}_{\x}^{(k)}]} , \qquad
\xymatrix{\tk {\k}M{\k[{\cal B}^{(k)}]^{G}} \ar[rr] && \k[{\cal B}^{(k)}]} $$
are isomorphisms. Then, according to Lemma~\ref{l2int}, there exists a group 
action of $W({\cal R})$ on $\k[{\cal B}_{\x}^{(k)}]$ extending the diagonal action of 
$W({\cal R})$ in $\e Sh^{\tens k}$ and such that 
$\k[{\cal B}_{\x}^{(k)}]^{W({\cal R})}=\k[{\cal B}^{(k)}]$ since 
$\k[{\cal B}^{(k)}]\cap \e Sh^{\tens k}=(\e Sh^{\tens k})^{W({\cal R})}$ by 
Corollary~\ref{c3bo5},(iii). Moreover, since $W({\cal R})$ is finite, the subfield of 
invariant elements of the fraction field of $\k[{\cal B}_{\x}^{(k)}]$ is the 
fraction field of $\k[{\cal B}_{\x}^{(k)}]^{W({\cal R})}$. Hence the action of 
$W({\cal R})$ in $\k[{\cal B}_{\x}^{(k)}]$ is trivial since 
$\k[{\cal B}_{\x}^{(k)}]$ and $\k[{\cal B}^{(k)}]$ have the same fraction field, 
whence the contradiction since $(\e Sh^{\tens k})^{W({\cal R})}$ is strictly contained in
$\e Sh^{\tens k}$.  
\end{proof}

\section{On the nullcone}\label{nc}
Let $k\geq 2$ be an integer. Let $I$ be the ideal of 
$\k[{\cal B}_{\n}^{(k)}]$ generated by $1 \tens \ec S{}{{\goth h}^{k}}{}+$. 

\begin{lemma}\label{lnc}
Let $N$ be the subscheme of ${\cal B}_{\n}^{(k)}$ defined by $I$. 

{\rm (i)} The ideal $I$ is prime and $N$ is isomorphic to ${\cal N}_{\n}^{(k)}$.

{\rm (ii)} The variety $N$ is the inverse image of ${\cal N}^{(k)}$ by $\eta _{\n}$.
\end{lemma}

\begin{proof}
(i) By Proposition~\ref{pbo5},(i), $\k[N]=\k[{\cal N}_{\n}^{(k)}]$, whence the assertion.

(ii) By (i), $N$ is reduced hence a variety. According to Remark~\ref{rbo5}, 
for $(g,\poi x1{,\ldots,}{k}{}{}{})$ in $G\times {\goth b}^{k}$, 
$\gamma _{\n}(\overline{(g,\poi x1{,\ldots,}{k}{}{}{})})$ is a zero of $I$ if and only if
$\poi x1{,\ldots,}{k}{}{}{}$ are nilpotent, whence the assertion.
\end{proof}

\begin{theo}\label{tnc}
{\rm (i)} The variety ${\cal N}^{(k)}$ is normal if and only if so is 
${\cal B}_{\x}^{(k)}$. If so, $\gamma _{\n}=\gamma _{\x}$ and the restriction
of $\varpi $ to ${\cal B}_{\x}^{(k)}$ is the normalization morphism of 
${\cal B}^{(k)}$.

{\rm (ii)} The variety ${\cal N}^{(k)}$ is Cohen-Macaulay if and only if so is
${\cal B}_{\x}^{(k)}$. 

{\rm (iii)} The variety ${\cal N}^{(k)}$ has rational singularities if and only if it
is Cohen-Macaulay.

{\rm (iv)} The variety ${\cal B}_{\x}^{(k)}$ has rational singularities if and only if 
it is Cohen-Macaulay.

{\rm (v)} Let $I_{0}$ be the ideal of $\k[{\cal B}^{(k)}]$ generated by 
$\k[{\cal B}^{(k)}]^{G}_{+}$. Then $I_{0}$ is strictly contained in the ideal of 
definition of ${\cal N}^{(k)}$ in $\k[{\cal B}^{(k)}]$.
\end{theo}

\begin{proof}
(i) According to Proposition~\ref{pbo5},(ii), $\k[{\cal B}_{\x}^{(k)}]$ is a free 
extension of $\es S{{\goth h}^{k}}$ and $\k[{\cal N}^{(k)}]$ is the quotient of 
$\k[{\cal B}_{\x}^{(k)}]$ by the ideal generated by $\ec S{}{{\goth h}^{k}}{}+$. So
by~\cite[Ch. 8, Theorem 23.9 and Corollary]{Mat}, $0$ is a normal point of 
${\cal B}_{\x}^{(k)}$ if ${\cal N}^{(k)}$ is normal. As a result ${\cal B}_{\x}^{(k)}$ is
normal if so is ${\cal N}^{(k)}$ since ${\cal B}_{\x}^{(k)}$ is a cone and its set of 
normal points is open. Conversely, suppose that ${\cal B}_{\x}^{(k)}$ is normal so that
${\cal B}_{\n}^{(k)}={\cal B}_{\x}^{(k)}$ and $\gamma _{\n}=\gamma _{\x}$. Moreover,
by Corollary~\ref{cbo3},(i), the restriction of $\varpi $ to ${\cal B}_{\x}^{(k)}$
is the normalization morphism of ${\cal B}^{(k)}$. According to 
Proposition~\ref{pbo5},(i), $\k[{\cal N}_{\n}^{(k)}]$ is the image of 
$\k[{\cal B}_{\n}^{(k)}]$ by a morphism and by Proposition~\ref{pbo5},(ii), 
$\k[{\cal N}^{(k)}]$ is the image of $\k[{\cal B}_{\x}^{(k)}]$ by this morphism, 
whence $\k[{\cal N}^{(k)}]=\k[{\cal N}_{\n}^{(k)}]$.

(ii) Suppose that ${\cal N}^{(k)}$ is Cohen-Macaulay. Then the localization of 
$\k[{\cal B}_{\x}^{(k)}]$ at $0$ is Cohen-Macaulay by~\cite[Ch. 8, Theorem 23.9]{Mat}
and Proposition~\ref{pbo5},(ii). By~\cite[Ch. 8, Theorem 24.5]{Mat}, the set of points of
${\cal B}_{\x}^{(k)}$ at which the localization is Cohen-Macaulay is open. Hence 
${\cal B}_{\x}^{(k)}$ is Cohen-Macaulay since its is a cone. 

Conversely suppose that ${\cal B}_{\x}^{(k)}$ is Cohen-Macaulay. According to 
Proposition~\ref{pbo5},(ii), any basis in $\es S{{\goth h}^{k}}$ is a regular sequence
in $\k[{\cal B}_{\x}^{(k)}]$ and $\k[{\cal N}^{(k)}]$ is the quotient of 
$\k[{\cal B}_{\x}^{(k)}]$ by the ideal generated by this sequence. Then the localization
at $0$ of $\k[{\cal N}^{(k)}]$ is Cohen-Macaulay by~\cite[Ch .6, Theorem 17.4 and 
Ch. 5, Theorem 14.1]{Mat}. So, again by~\cite[Ch. 8, Theorem 24.5]{Mat},
${\cal N}^{(k)}$ is Cohen-Macaulay since its is a cone. 

(iii) According to~\cite[p. 50]{KK}, ${\cal N}^{(k)}$ is Cohen-Macaulay if it has
rational singularities. Suppose that ${\cal N}^{(k)}$ is Cohen-Macaulay. By 
Lemma~\ref{lbo1},(iv) and Corollary~\ref{cbo1},(i), ${\cal N}^{(k)}$ is smooth in 
codimension $1$. Then, by Serre's normality criterion 
\cite[\S 1,no 10, Th\'eor\`eme 4]{Bou1}, ${\cal N}^{(k)}$ is normal. So, by 
\cite[p.50]{KK}, it remains to prove that for $U$ open subset of ${\cal N}^{(k)}$ and 
$\omega $ a regular differential form of top degree on the smooth locus of $U$, 
$\upsilon ^{*}(\omega )$ has a regular extension to $\upsilon ^{-1}(U)$.

Let $U'$ be the smooth locus of $U$. According to Lemma~\ref{lbo1},(iv), $U\cap V_{k}$ 
is contained in $U'$. So by Corollary~\ref{cbo1},(iii), $\upsilon ^{-1}(U')$ is a big open
subset of $\upsilon ^{-1}(U)$. Let $\Omega _{\upsilon ^{-1}(U)}$ be the sheaf of 
regular differential forms of top degree on $\upsilon ^{-1}(U)$. For some open cover 
$\poi O1{,\ldots,}{m}{}{}{}$ of $\upsilon ^{-1}(U)$, for $i=1,\ldots,m$, the 
restriction of $\Omega _{\upsilon ^{-1}(U)}$ to $O_{i}$ is a free 
$\an {O_{i}}{}$-module of rank $1$. Denoting by $\omega _{i}$ a generator, for some 
regular function $a_{i}$ on $O_{i}\cap \upsilon ^{-1}(U')$,
$$ \omega \left \vert \right. _{O_{i}\cap \upsilon ^{-1}(U')} = 
a_{i}( \omega _{i} \left \vert  \right._{O_{i}\cap \upsilon ^{-1}(U')}) .$$
Since $O_{i}$ is normal and $O_{i}\cap \upsilon ^{-1}(U')$ is a big open subset of 
$O_{i}$, $a_{i}$ has a regular extension to $O_{i}$. Denoting again by $a_{i}$ this 
extension, $a_{i}\omega _{i}$ is a regular differential form of top degree on $O_{i}$ 
having the same restriction as $\upsilon ^{*}(\omega )$ to 
$O_{i}\cap \upsilon ^{-1}(U')$. As a result, since $\Omega _{\upsilon ^{-1}(U)}$ is
torsion free and $\upsilon ^{-1}(U)$ is irreducible, for $1\leq i,j\leq m$, 
$a_{i}\omega _{i}$ and $a_{j}\omega _{j}$ have the same restriction to $O_{i}\cap O_{j}$.
Denoting by $\omega '$ the global section of $\Omega _{\upsilon ^{-1}(U)}$ whose 
restriction to $O_{i}$ is $a_{i}\omega _{i}$ for $i=1,\ldots,m$, 
$\upsilon ^{*}(\omega )$ is the restriction of $\omega '$ to $\upsilon ^{-1}(U')$,
whence the assertion.

(iv) According to~\cite[p. 50]{KK}, ${\cal B}_{\x}^{(k)}$ is Cohen-Macaulay if it has
rational singularities. Suppose that ${\cal B}_{\x}^{(k)}$ is Cohen-Macaulay. By 
Lemma~\ref{l2bo4},(iv), ${\cal B}_{\x}^{(k)}$ is smooth in codimension $1$. Then, by 
Serre's normality criterion \cite[\S 1,no 10, Th\'eor\`eme 4]{Bou1}, 
${\cal B}_{\x}^{(k)}$ is normal. So, by \cite[p.50]{KK}, it remains to prove that for $U$
open subset of ${\cal B}_{\x}^{(k)}$ and $\omega $ a regular differential form of top 
degree on the smooth locus of $U$, $\gamma _{\x}^{*}(\omega )$ has a regular extension to
$\gamma _{\x}^{-1}(U)$.

Let $U'$ be the smooth locus of $U$. According to Lemma~\ref{l2bo4},(iv), $U\cap W_{k}$ 
is contained in $U'$. So by Corollary~\ref{cbo4}, $\gamma _{\x}^{-1}(U')$ is a big open
subset of $\gamma _{\x}^{-1}(U)$. Let $\Omega _{\gamma _{\x}^{-1}(U)}$ be the sheaf of 
regular differential forms of top degree on $\gamma _{\x}^{-1}(U)$. For some open cover 
$\poi O1{,\ldots,}{m}{}{}{}$ of $\gamma _{\x}^{-1}(U)$, for $i=1,\ldots,m$, the 
restriction of $\Omega _{\gamma _{\x}^{-1}(U)}$ to $O_{i}$ is a free 
$\an {O_{i}}{}$-module of rank $1$. Denoting by $\omega _{i}$ a generator, for some 
regular function $a_{i}$ on $O_{i}\cap \gamma _{\x}^{-1}(U')$,
$$ \omega \left \vert \right. _{O_{i}\cap \gamma _{\x}^{-1}(U')} = 
a_{i}( \omega _{i} \left \vert  \right._{O_{i}\cap \gamma _{\x}^{-1}(U')}) .$$
Since $O_{i}$ is normal and $O_{i}\cap \gamma _{\x}^{-1}(U')$ is a big open subset of 
$O_{i}$, $a_{i}$ has a regular extension to $O_{i}$. Denoting again by $a_{i}$ this 
extension, $a_{i}\omega _{i}$ is a regular differential form of top degree on $O_{i}$ 
having the same restriction as $\gamma _{\x}^{*}(\omega )$ to 
$O_{i}\cap \gamma _{\x}^{-1}(U')$. As a result, since $\Omega _{\gamma _{\x}^{-1}(U)}$ is
torsion free and $\gamma _{\x}^{-1}(U)$ is irreducible, for $1\leq i,j\leq m$, 
$a_{i}\omega _{i}$ and $a_{j}\omega _{j}$ have the same restriction to $O_{i}\cap O_{j}$.
Denoting by $\omega '$ the global section of $\Omega _{\gamma _{\x}^{-1}(U)}$ whose 
restriction to $O_{i}$ is $a_{i}\omega _{i}$ for $i=1,\ldots,m$, 
$\gamma _{\x}^{*}(\omega )$ is the restriction of $\omega '$ to $\gamma _{\x}^{-1}(U')$,
whence the assertion.

(v) Since $\k[{\cal B}^{(k)}]^{G}_{+}$ is contained in $\ec S{}{{\goth h}^{k}}{}+$, 
$I_{0}$ is contained in $I\cap \k[{\cal B}^{(k)}]$. According to Lemma~\ref{lnc},(ii) and 
(i), $I\cap \k[{\cal B}^{(k)}]$ is the ideal of definition of 
${\cal N}^{(k)}$ in $\k[{\cal B}^{(k)}]$. Let $M$ be a graded complement of 
$\k[{\cal B}^{(k)}]^{G}_{+}\k[{\cal B}^{(k)}]$ in $\k[{\cal B}^{(k)}]$. According to 
Corollary~\ref{c4bo5},(ii), $I\cap M$ is different from $\{0\}$. Hence $I_{0}$ is 
strictly contained in $I\cap \k[{\cal B}^{(k)}]$, whence the assertion. 
\end{proof}

\begin{rema}\label{rnc}
According to~\cite[6.2]{VX}, for ${\goth g}$ simple of type $B_{2}$, ${\cal N}^{(2)}$
is not normal and according to~\cite[Theorem 6.1]{VX}, for ${\goth g}={\mathrm {sl}}_{3}$,
${\cal N}^{(k)}$ has rational singularities for all $k$.
\end{rema}

Summarizing the results of the preceding subsections, Theorem~\ref{tint},(i), (ii),
(iii), (iv), (vii) are given by Theorem~\ref{tnc}, Theorem~\ref{tint},(v) is given by 
Proposition~\ref{pbo5},(ii) and Corollary~\ref{c3bo5},(i) and Theorem~\ref{tint},(vi)
is given by Corollary~\ref{c3bo5},(iii). To complete Theorem~\ref{tint}, recall that 
$\varkappa $ is the normalization morphism of ${\cal N}^{(k)}$ and denote by $\etaup $ 
the normalization morphism of ${\cal B}_{\x}^{(k)}$.

\begin{prop}\label{pnc}
{\rm (i)} The morphism $\etaup $ is a homeomorphism.

{\rm (ii)} The morphism $\varkappa $ is a homeomorphism.
\end{prop}

\begin{proof}
Recall that the morphisms
$$ \xymatrix{ \sqx G{{\goth b}} \ar[rr] && G/B\times {\goth g}} \quad  \text{and} \quad
\xymatrix{ \sqx G{{\goth b}^{k}} \ar[rr] && G/B\times {\goth g}^{k}} $$ 
are closed embeddings. For $x=(\poi x1{,\ldots,}{k}{}{}{},\poi y1{,\ldots,}{k}{}{}{})$ in
${\cal B}_{\x}^{(k)}$, denote by ${\goth B}_{x}$ the subset of Borel subalgebras 
${\goth b}'$ of ${\goth g}$ such that $\chi _{\n}({\goth b}',x_{i}) = (x_{i},y_{i})$ for 
$i=1,\ldots,k$. Then 
$\gamma _{\x}^{-1}(\{x\})={\goth B}_{x}\times \{(\poi x1{,\ldots,}{}{}{}{})\}$. From the 
two commutative  diagramms
$$ \xymatrix{ \sqx G{{\goth b}^{k}} \ar[rr]^{\gamma _{\n}} \ar[rd]_{\gamma _{\x}} && 
{\cal B}_{\n}^{(k)} \ar[ld]^{\etaup }\\ & {\cal B}_{\x}^{(k)} & }, \quad
\xymatrix{ \sqx G{{\goth u}^{k}} \ar[rr]^{\upsilon _{\n}} \ar[rd]_{\upsilon } && 
{\cal N}_{\n}^{(k)} \ar[ld]^{\varkappa }\\ & {\cal N}^{(k)} & }$$
we deduce that it suffices to prove that ${\goth B}_{x}$ is connected for all $x$ in 
${\cal B}_{\x}^{(k)}$ since $\upsilon $ is the restriction of $\gamma _{\x}$ to 
$\sqx G{{\goth u}^{k}}$. 

For $\beta $ a simple root, denote by $s_{\beta }$ the reflection of ${\goth h}$ with 
respect to $\beta $. For $w$ in $W({\cal R})$ denote by $l(w)$ its length with respect to
the set of simple roots, $n_{w}$ a representative of $w$ in $N_{G}({\goth h})$ and set
$w({\goth b}) := n_{w}({\goth b})$. Let 
$x=(\poi x1{,\ldots,}{k}{}{}{},\poi y1{,\ldots,}{k}{}{}{})$ be in 
${\cal B}_{\x}^{(k)}$ and ${\goth b}'$ and ${\goth b}''$ in ${\goth B}_{x}$. By Bruhat 
decomposition of $G$, for some $(g,b,w)$ in $G\times B\times W({\cal R})$, 
${\goth b}'=g({\goth b})$, ${\goth b}'' = gbw({\goth b})$. Set: 
$$l(w):= q, \qquad u_{i} := g^{-1}(x_{i}), \qquad v_{i} := b^{-1}(u_{i}) $$ 
for $i=1,\ldots,k$. In particular, 
$v := (\poi v1{,\ldots,}{k}{}{}{},\poi y1{,\ldots,}{k}{}{}{})$ is in 
${\cal B}_{\x}^{(k)}$ and ${\goth b}$ and $w({\goth b})$ are in 
${\goth B}_{v}$ and it suffices to prove ${\goth b}$ and $w({\goth b})$ are in the
same connected component of ${\goth B}_{v}$ since $x = gb.v$. It will be a consequence of 
the following claim.

\begin{claim}
There exist a sequence $\poi L1{,\ldots,}{q}{}{}{}$ of projective lines contained in 
${\goth B}_{v}$ and a sequence $\poi {{\goth b}}0{,\ldots,}{q}{}{}{}$ in ${\goth B}_{v}$ 
such that  
$$ {\goth b} = {\goth b}_{0}, \quad w({\goth b}) = {\goth b}_{q}, \quad 
{\goth b}_{i-1} \in L_{i}, \quad {\goth b}_{i} \in L_{i} $$
for $i=1,\ldots,q$.
\end{claim}

Prove the claim by induction on $q$. For $q=0$, ${\goth b}=w({\goth b})$. Suppose that  
$q=1$ and $w=s_{\beta }$ for some simple root $\beta $. Then $\poi v1{,\ldots,}{k}{}{}{}$
are in ${\goth b}\cap s_{\beta }({\goth b})$ and for $i=1,\ldots,k$, 
$$\chi _{\n}({\goth b},v_{i}) = \chi _{\n}(s_{\beta }({\goth b}),v_{i})= (v_{i},y_{i}) .$$
Let ${\goth p}_{\beta }$ be the parabolic subalgebra ${\goth g}^{-\beta }+{\goth b}$
and ${\goth l}_{\beta }$ the reductive factor containing ${\goth h}$. The set 
$L_{\beta }$ of Borel subalgebras of ${\goth g}$, contained in  ${\goth p}_{\beta }$, is 
a projective line. In the case ${\goth g}={\goth l}_{\beta }$, ${\cal N}^{(k)}$ is 
normal and $\eta $ is an isomorphism by Theorem~\ref{tnc},(i). Then, by Zariski's Main 
Theorem \cite[\S 9]{Mu}, the fibers of $\gamma _{\x}$ are connected. So, $L_{\beta }$ is 
contained in ${\goth B}_{v}$ since ${\goth b}$ and $s_{\beta }({\goth b})$ are two 
different points of ${\goth B}_{v}\cap L_{\beta }$.

Suppose the claim true for the integers smaller than $q$. Let 
$w=\poi s1{\cdots }{q}{}{}{}$ be a reduced decomposition of $w$ and set 
$w' := \poi s1{\cdots }{q-1}{}{}{}$. For $i=1,\ldots,q$, let $\beta _{i}$ be the simple 
root such that $s_{i}=s_{\beta _{i}}$. For $i=1,\ldots,k$,  
$$ v_{i} \in {\goth h}\oplus 
\bigoplus _{\mycom{\gamma \in {\cal R}_{+}}{w(\gamma )\in {\cal R}_{+}}} 
{\goth g}^{w(\gamma )} \quad  \text{and} \quad
n_{w'}^{-1}(v_{i}) \in {\goth h} \oplus 
\bigoplus _{\mycom{\gamma \in {\cal R}_{+}}{w(\gamma )\in {\cal R}_{+}}} 
{\goth g}^{s_{q}(\gamma )}.$$
Since $\poi s1{\cdots }{q}{}{}{}$ is the reduced decomposition of $w$, $w(\beta _{q})$ is
a negative root, whence
$$ \{\gamma \in {\cal R}_{+} \, \vert \, w(\gamma )\in {\cal R}_{+} \} \subset 
{\cal R}_{+}\setminus \{\beta _{q}\} .$$
As a result $n_{w'}^{-1}(v_{1}),\ldots,n_{w'}^{-1}(v_{k})$ are in ${\goth b}$.
So, by induction hypothesis, there exist a sequence $\poi L0{,\ldots,}{q-1}{}{}{}$ of 
projective lines contained in ${\goth B}_{v}$ and a sequence 
$\poi {{\goth b}}0{,\ldots,}{q-1}{}{}{}$ in ${\goth B}_{v}$ such that 
$$ {\goth b} = {\goth b}_{0}, \qquad w'({\goth b}) = {\goth b}_{q-1}, \qquad
{\goth b}_{i-1}\in L_{i}, \qquad {\goth b}_{i} \in L_{i} $$
for $i=1,\ldots,q-1$. By the case $q=1$, for some projective line $L'_{q}$, contained in 
${\goth B}_{n_{w'}^{-1}.v}$, ${\goth b}$ and $s_{q}({\goth b})$ are in $L'_{q}$. Then,
setting ${\goth b}_{q}=w({\goth b})$ and $L_{q}:= n_{w'}.L'_{q}$, the sequences 
$\poi L1{,\ldots,}{q}{}{}{}$ and $\poi {{\goth b}}0{,\ldots,}{q}{}{}{}$  verify the 
conditions of the claim.
\end{proof}

\section{Main varieties} \label{mv}
Denote by $X$ the closure in $\ec {Gr}g{}{}{\rg}$ of the orbit of ${\goth h}$ under 
$B$. According to Lemma~\ref{l4int}, $G.X$ is the closure in $\ec {Gr}g{}{}{\rg}$ of the
orbit of ${\goth h}$ under $G$. Let ${\cal E}_{0}$ and ${\cal E}$ be the restrictions to
$X$ and $G.X$ respectively of the totaulogical vector bundle over $\ec {Gr}g{}{}{\rg}$.
By definition, ${\cal E}$ is the subvariety of elements $(V,x)$ of $G.X \times {\goth g}$
such that $x$ is in $V$ and ${\cal E}_{0}$ is the intersection of ${\cal E}$ and 
$X\times {\goth b}$. In this section, we give geometric properties of $X$ and $G.X$. These
varieties play an important role in the study of the generalized commuting varieties and 
isospectrale commuting varieties as it is suggested by Theorem~\ref{t3int} and it will be
shown in two future notes.

\subsection{} \label{mv1}
For $\alpha $ in ${\cal R}$, set 
$V_{\alpha } := {\goth h}_{\alpha }\oplus {\goth g}^{\alpha }$ and 
denote by $X_{\alpha }$ the closure in $\ec {Gr}g{}{}{\rg}$ of the orbit of 
$V_{\alpha }$ under $B$.

\begin{lemma}\label{lmv1}
Let $\alpha $ be in ${\cal R}_{+}$. Let ${\goth p}$ be a parabolic subalgebra containing 
${\goth b}$ and let $P$ be its normalizer in $G$.

{\rm (i)} The subset $P.X$ of $\ec {Gr}g{}{}{\rg}$ is the closure in 
$\ec {Gr}g{}{}{\rg}$ of the orbit of ${\goth h}$ under $P$.

{\rm (ii)} The closed set $X_{\alpha }$ of $\ec {Gr}g{}{}{\rg}$ is an irreducible 
component of $X\setminus B.{\goth h}$.

{\rm (iii)} The set $P.X_{\alpha }$ is an irreducible component of 
$P.X\setminus P.{\goth h}$.

{\rm (iv)} The varieties $X\setminus B.{\goth h}$ and $P.X\setminus P.{\goth h}$ are 
equidimenional of codimension $1$ in $X$ and $P.X$ respectively.
\end{lemma}

\begin{proof}
(i) Since $X$ is a $B$-invariant closed subset of $\ec {Gr}g{}{}{\rg}$, $P.X$ is a closed 
subset of $\ec {Gr}g{}{}{\rg}$ by Lemma~\ref{l4int}. Hence $\overline{P.{\goth h}}$ is 
contained in $P.X$ since ${\goth h}$ is in $X$, whence the assertion since 
$\overline{P.{\goth h}}$ is a $P$-invariant subset containing $X$.

(ii) Denoting by $H_{\alpha }$ the coroot of $\alpha $,
$$ \lim _{t\rightarrow \infty } \exp (t\ad x_{\alpha })(\frac{-1}{2t} H_{\alpha }) = 
x_{\alpha } .$$
So $V_{\alpha }$ is in the closure of the orbit of ${\goth h}$ under the one parameter
subgroup of $G$ generated by $\ad x_{\alpha }$. As a result, $X_{\alpha }$ is a closed 
subset of $X\setminus B.{\goth h}$ since $V_{\alpha }$ is not a Cartan subalgebra. 
Moreover, $X_{\alpha }$ has dimension $n-1$ since the normalizer of $V_{\alpha }$ in 
${\goth g}$ is ${\goth h}+{\goth g}^{\alpha }$. Hence $X_{\alpha }$ is an irreducible 
component of $X\setminus B.{\goth h}$ since $X$ has dimension $n$.

(iii) Since $X_{\alpha }$ is a $B$-invariant closed subset of $\ec {Gr}g{}{}{\rg}$, 
$P.X_{\alpha }$ is a closed subset of $\ec {Gr}g{}{}{\rg}$ by Lemma~\ref{l4int}. According
to (ii), $P.X_{\alpha }$ is contained in $P.X\setminus P.{\goth h}$ and it has 
dimension $\dim {\goth p}-\rg-1$, whence the assertion since $P.X$ has dimension 
$\dim {\goth p}-\rg$.

(iv) Let $P_{\u}$ be the unipotent radical of $P$ and let $L$ be the reductive factor
of $P$ whose Lie algebra contains $\ad {\goth h}$. Denote by $N_{L}({\goth h})$ the 
normalizer of ${\goth h}$ in $L$. Since $B.{\goth h}$ and $P.{\goth h}$ are isomorphic to
$U$ and $L/N_{L}({\goth h})\times P_{\u}$ respectively, they are affine open subsets of 
$X$ and $P.X$ respectively, whence the assertion by \cite[Corollaire 21.12.7]{Gro1}.
\end{proof}

For $x$ in ${\goth g}$, set:
$$ V_{x} := {\mathrm {span}}(\{\poi x{}{,\ldots,}{}{\varepsilon }{1}{\rg}\}) .$$

\begin{lemma}\label{l2mv1}
{\rm (i)} For $(V,x)$ in $X \times {\goth b}$, $(V,x)$ is in the closure of 
$B.(\{{\goth h}\}\times {\goth h}_{\r})$ in $\ec {Gr}b{}{}{\rg}\times {\goth b}$ if and 
only if $x$ is in $V$.

{\rm (ii)} The set ${\cal E}$ is the closure in $\ec {Gr}g{}{}{\rg}\times {\goth g}$ of 
$G.(\{{\goth h}\}\times {\goth h}_{\r})$. 

{\rm (iii)} For $(V,x)$ in ${\cal E}$, $V_{x}$ is contained in $V$.
\end{lemma}

\begin{proof}
(i) Let ${\cal E}'_{0}$ be the closure of $B.(\{{\goth h}\}\times {\goth h}_{\r})$ in 
$\ec {Gr}b{}{}{\rg}\times {\goth b}$. Then ${\cal E}'_{0}$ is a closed subset of 
${\cal E}_{0}$. Let $(V,x)$ be in ${\cal E}_{0}$. Let $E$ be a complement to $V$ in 
${\goth b}$ and let $\Omega _{E}$ be the set of complements to $E$ in ${\goth g}$. Then 
$\Omega _{E}$ is an open neighborhood of $V$ in $\ec {Gr}b{}{}{\rg}$. Moreover, the map 
$$\xymatrix{{\mathrm {Hom}}_{\k}(V,E) \ar[rr]^{\kappaup } && \Omega _{E}}, \qquad
\varphi \longmapsto \kappaup (\varphi ) := 
{\mathrm {span}}(\{v+\varphi (v) \ \vert \ v \in V \}) .$$ 
is an isomorphism of varieties. Let $\Omega _{E}^{c}$ be the inverse image of the set of 
Cartan subalgebras. Then $0$ is in the closure of $\Omega _{E}^{c}$ in 
${\mathrm {Hom}}_{\k}(V,E)$ since $V$ is in $X$. For all $\varphi $ in 
$\Omega _{E}^{c}$, $(\kappaup (\varphi ),x+\varphi (x))$ is in ${\cal E}'_{0}$. Hence 
$(V,x)$ is in ${\cal E}'_{0}$.

(ii) Let $(V,x)$ be in ${\cal E}$. For some $g$ in $G$, $g(V)$ is in $X$. So by (i),
$(g(V),g(x))$ is in ${\cal E}_{0}$ and $(V,x)$ is in the closure of 
$G.(\{{\goth h}\}\times {\goth h}_{\r})$ in $\ec {Gr}g{}{}{\rg}\times {\goth g}$, whence
the assertion.

(iii) For $i=1,\ldots,\rg$, let ${\cal E}_{i}$ be the set of elements 
$(V,x)$ of ${\cal E}$ such that $\varepsilon _{i}(x)$ is in $V$. Then ${\cal E}_{i}$ is a 
closed subset of $G.X \times {\goth g}$, invariant under the action of 
$G$ in $\ec {Gr}g{}{}{\rg}\times {\goth g}$ since $\varepsilon _{i}$ is a $G$-equivariant
map. For all $(g,x)$ in $G\times {\goth h}_{\r}$, $(g({\goth h}),g(x))$ is in 
${\cal E}_{i}$ since $\varepsilon _{i}(g(x))$ centralizes $g(x)$. Hence 
${\cal E}_{i}={\cal E}$ since $G.({\goth h}_{\r}\times \{{\goth h}\})$ is dense in 
${\cal E}$ by (ii). As a result, for all $V$ in $G.X$ and for all $x$ in 
$V$, $\poi x{}{,\ldots,}{}{\varepsilon }{1}{\rg}$ are in $V$.
\end{proof}

\begin{coro}\label{cmv1}
Let $(V,x)$ be in ${\cal E}$. 

{\rm (i)} The space ${\goth z}_{x_{\s}}$ is contained in $V_{x}$ and $V$. 

{\rm (ii)} The space $V$ is an algebraic, commutative subalgebra of ${\goth g}$.
\end{coro}

\begin{proof}
(i) If $x$ is regular semisimple, $V$ is a Cartan subalgebra of ${\goth g}$ whence the
assertion in this case by Lemma~\ref{l2mv1},(iii) and \cite[Theorem 9]{Ko}. Suppose 
that $x$ is not regular semisimple. Let ${\goth N}_{{\goth g}^{x_{\s}}}$ be the nilpotent
cone of ${\goth g}^{x_{s}}$ and let $\Omega _{\r}$ be the regular nilpotent orbit of 
${\goth g}^{x_{\s}}$. For all $y$ in $\Omega _{\r}$, $x_{\s}+y$ is in ${\goth g}_{\r}$ 
and $\poi {x_{\s}+y}{}{,\ldots,}{}{\varepsilon }{1}{\rg}$ is a basis of 
${\goth g}^{x_{\s}+y}$ by \cite[Theorem 9]{Ko}. Then for all $z$ in ${\goth z}_{x_{\s}}$,
there exist regular functions on $\Omega _{\r}$, $\poi a{1,z}{,\ldots,}{\rg ,z}{}{}{}$, 
such that
$$ z = a_{1,z}(y)\varepsilon _{1}(x_{\s}+y) + \cdots + a_{\rg ,z}(y) 
\varepsilon _{\rg}(x_{\s}+y)$$
for all $y$ in $\Omega _{\r}$. Furthermore, these functions are uniquely defined by
this equality. Since ${\goth N}_{{\goth g}^{x_{\s}}}$ is a normal variety and  
${\goth N}_{{\goth g}^{x_{\s}}}\setminus \Omega _{\r}$ has codimension $2$ in 
${\goth N}_{{\goth g}^{x_{\s}}}$, the functions $\poi a{1,z}{,\ldots,}{\rg ,z}{}{}{}$ 
have regular extensions to ${\goth N}_{{\goth g}^{x_{\s}}}$. Denoting again by $a_{i,z}$ 
the regular extension of $a_{i,z}$ for $i=1,\ldots,\rg$,
$$ z = a_{1,z}(y)\varepsilon _{1}(x_{\s}+y) + \cdots + a_{\rg ,z}(y) 
\varepsilon _{\rg}(x_{\s}+y)$$
for all $y$ in ${\goth N}_{{\goth g}^{x_{\s}}}$. As a result, ${\goth z}_{x_{\s}}$ is 
contained in $V_{x}$. Hence ${\goth z}_{x_{\s}}$ is contained in $V$ by 
Lemma~\ref{l2mv1},(iii).

(ii) Since the set of commutative subalgebras of dimension $\rg$ is closed in 
$\ec {Gr}g{}{}{\rg}$, $V$ is a commutative subalgebra of ${\goth g}$. According to (i), 
the semisimple and nilpotent components of the elements of $V$ are 
contained in $V$. For $x$ in $V\setminus {\goth N}_{{\goth g}}$, all the replica of 
$x_{\s}$ are contained in the center of ${\goth g}^{x_{\s}}$. Hence $V$ is an algebraic 
subalgebra of ${\goth g}$ by (i). 
\end{proof}

\subsection{} \label{mv2}
For $s$ in ${\goth h}$, denote by $X^{s}$ the subset of elements of $X$, contained in
${\goth g}^{s}$.

\begin{lemma}\label{lmv2}
Let $s$ be in ${\goth h}$.

{\rm (i)} The set $X^{s}$ is the closure in $\ec {Gr}{}{{\goth g}^{s}}{}{\rg}$ of the 
orbit of ${\goth h}$ under $B^{s}$.

{\rm (ii)} The set of elements of $G.X$ containing ${\goth z}_{s}$ is the closure in 
$\ec {Gr}g{}{}{\rg}$ of the orbit of ${\goth h}$ under $G^{s}$.
\end{lemma}

\begin{proof}
(i) Set ${\goth p} := {\goth g}^{s}+{\goth b}$, let $P$ be the normalizer of
${\goth p}$ in $G$ and let ${\goth p}_{\u}$ be the nilpotent radical of ${\goth p}$. For 
$g$ in $P$, denote by $\overline{g}$ its image by the canonical projection from $P$ 
to $G^{s}$. Let $Z$ be the closure in 
$\ec {Gr}g{}{}{\rg}\times \ec {Gr}g{}{}{\rg}$ of the image of the map
$$ \xymatrix{B \ar[rr] &&  \ec {Gr}b{}{}{\rg}\times \ec {Gr}b{}{}{\rg}} , \qquad 
g \longmapsto (g({\goth h}),\overline{g}({\goth h})) $$
and let $Z'$ be the subset of elements $(V,V')$ of 
$\ec {Gr}b{}{}{\rg}\times \ec {Gr}b{}{}{\rg}$ such that
$$ V' \subset {\goth g}^{\s}\cap {\goth b} \quad  \text{and} \quad
V \subset V'\oplus {\goth p}_{\u} .$$
Then $Z'$ is a closed subset of $\ec {Gr}b{}{}{\rg}\times \ec {Gr}b{}{}{\rg}$ and 
$Z$ is contained in $Z'$ since $(g({\goth h}),\overline{g}({\goth h}))$ is in $Z'$ for 
all $g$ in $B$. Since $\ec {Gr}b{}{}{\rg}$ is a projective variety, the images 
of $Z$ by the projections $(V,V')\mapsto V$ and $(V,V')\mapsto V'$ are closed in 
$\ec {Gr}b{}{}{\rg}$ and they are equal to $X$ and $\overline{B^{s}.{\goth h}}$ 
respectively. Furthermore, $\overline{B^{s}.{\goth h}}$ is contained in $X^{s}$. 

Let $V$ be in $X^{s}$. For some $V'$ in $\ec {Gr}b{}{}{\rg}$, $(V,V')$ is in $Z$. Since
$$ V \subset {\goth g}^{s}, \ V'\subset {\goth g}^{s}, \ 
V \subset V'\oplus {\goth p}_{\u} ,$$
$V=V'$ so that $V$ is in $\overline{B^{s}.{\goth h}}$, whence the assertion.

(ii) Since ${\goth z}_{\s}$ is contained in ${\goth h}$, all element of 
$\overline{G^{s}.{\goth h}}$ is an element of $G.X$ containing ${\goth z}_{s}$. Let $V$ 
be in $G.X$, containing ${\goth z}_{s}$. Since $V$ is a commutative subalgebra of 
${\goth g}^{s}$ and since ${\goth g}^{s}\cap {\goth b}$ is a Borel 
subalgebra of ${\goth g}^{s}$, for some $g$ in $G^{s}$, $g(V)$ is contained in 
${\goth b}\cap {\goth g}^{s}$. So, one can suppose that $V$ is contained in ${\goth b}$.
According to the Bruhat decomposition of $G$, since $X$ is $B$-invariant, for some 
$b$ in $U$ and for some $w$ in $W({\cal R})$, $V$ is in $bw.X$. Set:
$${\cal R}_{+,w} := \{\alpha \in {\cal R}_{+} \ \vert \ w(\alpha ) \in {\cal R}_{+}\} ,
\qquad
{\cal R}'_{+,w} := \{\alpha \in {\cal R}_{+} \ \vert \ w(\alpha ) \not \in {\cal R}_{+}\} 
,$$
$${\goth u}_{1} := \bigoplus _{\alpha \in {\cal R}_{+,w}} {\goth g}^{w(\alpha )}, \qquad
{\goth u}_{2} := \bigoplus _{\alpha \in -{\cal R}'_{+,w}} {\goth g}^{w(\alpha )}, \qquad
{\goth u}_{3} := \bigoplus _{\alpha \in {\cal R}'_{+,w}} {\goth g}^{w(\alpha )},$$
$$B^{w} := wBw^{-1}, \qquad {\goth b}^{w} := {\goth h} \oplus {\goth u}_{1} \oplus 
{\goth u}_{3} ,$$
so that $\ad {\goth b}^{w}$ is the Lie algebra of $B^{w}$ and $w.X$ is the closure in 
$\ec {Gr}g{}{}{\rg}$ of the orbit of ${\goth h}$ under $B^{w}$. Moreover, 
${\goth u}$ is the direct sum of ${\goth u}_{1}$ and ${\goth u}_{2}$. For $i=1,2$,  
denote by $U_{i}$ the closed subgroup of $U$ whose Lie algebra is $\ad {\goth u}_{i}$. 
Then $U=U_{2}U_{1}$ and $b=b_{2}b_{1}$ with $b_{i}$ in $U_{i}$ for $i=1,2$. Since 
$w^{-1}({\goth u}_{1})$ is contained in ${\goth u}$ and $X$ is invariant under 
$B$, $b_{2}b_{1}w.X=b_{2}w.X$. Then $b_{2}^{-1}(V)$ is in $w.X$ and 
$$ b_{2}^{-1}(V) \subset {\goth b} \cap {\goth b}^{w} = {\goth h} \oplus {\goth u}_{1} $$
since $V$ is contained in ${\goth b}$. Set:
$${\goth u}_{2,1} := {\goth u}_{2}\cap {\goth g}^{s} , \qquad 
{\goth u}_{2,2} := {\goth u}_{2} \cap {\goth p}_{\u} $$
and for $i=1,2$, denote by $U_{2,i}$ the closed subgroup of $U_{2}$ whose Lie algebra 
is $\ad {\goth u}_{2,i}$. Then ${\goth u}_{2}$ is the direct sum of ${\goth u}_{2,1}$
and ${\goth u}_{2,2}$ and $U_{2}=U_{2,1}U_{2,2}$ so that $b_{2}=b_{2,1}b_{2,2}$ with
$b_{2,i}$ in $U_{2,i}$ for $i=1,2$. As a result, ${\goth z}_{s}$ is contained in 
$b_{2,1}^{-1}(V)$ and $b_{2,2}^{-1}({\goth z}_{s})$ is contained 
${\goth h}\oplus {\goth u}_{1}$. Hence $b_{2,2}^{-1}({\goth z}_{s})={\goth z}_{s}$ since 
${\goth u}_{1}\cap {\goth u}_{2,2}=\{0\}$. 

Suppose $b_{2,2}\neq 1_{{\goth g}}$. We expect a contradiction. For some $x$ in 
${\goth u}_{2,2}$, $b_{2,2}=\exp (\ad x)$. The space ${\goth u}_{2,2}$ is a direct sum
of root spaces since so are ${\goth u}_{2}$ and ${\goth p}_{\u}$. Let 
$\poi {\alpha }1{,\ldots,}{m}{}{}{}$ be the positive roots such that the corresponding
root spaces are contained in ${\goth u}_{2,2}$. They are ordered so that for $i\leq j$, 
$\alpha _{j}-\alpha _{i}$ is a positive root if it is a root. For $i=1,\ldots,m$, let
$c_{i}$ be the coordinate of $x$ at $x_{\alpha _{i}}$ and let $i_{0}$ be the smallest 
integer such that $c_{i_{0}}\neq 0$. For all $z$ in ${\goth z}_{s}$, 
$$b_{2,2}^{-1}(z)-z-c_{i_{0}}\alpha _{i_{0}}(z)x_{\alpha _{i_{0}}} \in 
\bigoplus _{j>i_{0}} {\goth g}^{\alpha _{j}} ,$$ 
whence the contradiction since for some $z$ in ${\goth z}_{s}$, 
$\alpha _{i_{0}}(z)\neq 0$. As a result, $b_{2,1}^{-1}(V)$ is an element of 
$w.X=\overline{B^{w}.{\goth h}}$, contained in ${\goth g}^{s}$. So, by (i), 
$b_{2,1}^{-1}(V)$ and $V$ are in $\overline{G^{s}.{\goth h}}$, whence the assertion.
\end{proof}

Define a torus of ${\goth g}$ as a commutative algebraic subalgebra of ${\goth g}$
whose all elements are semisimple. For $\Lambda $ subset of ${\cal R}$, denote by 
${\goth h}_{\Lambda }$ the intersection of the kernels of the elements of $\Lambda $.

\begin{coro}\label{cmv2}
Let $V$ be in $X$. Then for some subset $\Lambda $ of ${\cal R}$ and for some $g$ in 
$B$, $g(V)$ is the direct sum of ${\goth h}_{\Lambda }$ and $g(V)\cap {\goth u}$.
\end{coro}

\begin{proof}
By Corollary~\ref{cmv1},(ii), $V$ is the direct sum of a subtorus of ${\goth b}$ and its 
intersection with ${\goth u}$. So for some $g$ in $B$, 
$$g(V) = g(V)\cap {\goth h} \oplus g(V) \cap {\goth u} .$$
Let $\Lambda $ be the set of roots such that $g(V)\cap {\goth h}$ is contained in 
${\goth h}_{\Lambda }$. If $\Lambda ={\cal R}$, $g(V)$ is contained in ${\goth u}$. 
Suppose $\Lambda $ strictly containd in ${\cal R}$. For some $s$ in $g(V)\cap {\goth h}$, 
$\alpha (s)\neq 0$ for all $\alpha $ in ${\cal R}\setminus \Lambda $. Since $g(V)$ is a 
commutative algebra, $g(V)$ is contained in ${\goth g}^{s}$. So, by Lemma~\ref{lmv2},(i),
$g(V)$ is in $\overline{B^{s}.{\goth h}}$. In particular, by Corollary~\ref{cmv1},(i),
${\goth h}_{\Lambda }$ is contained in $g(V)$ since ${\goth h}_{\Lambda }$ is the center 
of ${\goth g}^{s}$, whence ${\goth h}_{\Lambda }=g(V)\cap {\goth h}$ and $g(V)$ is the 
direct sum of ${\goth h}_{\Lambda }$ and $g(V)\cap {\goth u}$.
\end{proof}

\subsection{} \label{mv3}
For $x$ in ${\goth g}$, denote by $Z_{x}$ the subset of elements of $G.X$ containing $x$ 
and by $(G^{x})_{0}$ the identity component of $G^{x}$. 

\begin{lemma}\label{lmv3}
Let $x$ be in ${\goth N}_{{\goth g}}$ and let $Z$ be an irreducible component of $Z_{x}$.
Suppose that some element of $Z$ is not contained in ${\goth N}_{{\goth g}}$.

{\rm (i)} For some torus ${\goth s}$ of ${\goth g}^{x}$, all element of 
a dense open subset of $Z$ contains a conjugate of ${\goth s}$ under $(G^{x})_{0}$.

{\rm (ii)} For some $s$ in ${\goth s}$ and for some irreducible component $Z_{1}$ of 
$Z_{s+x}$, $Z$ is the closure in $\ec {Gr}g{}{}{\rg}$ of $(G^{x})_{0}.Z_{1}$.

{\rm (iii)} If $Z_{1}$ has dimension smaller than $\dim {\goth g}^{s+x}-\rg$, then $Z$ 
has dimension smaller than $\dim {\goth g}^{x}-\rg$.
\end{lemma}

\begin{proof}
(i) After some conjugation by an element of $G$, we can suppose that 
${\goth g}^{x}\cap {\goth b}$ and ${\goth g}^{x}\cap {\goth h}$ are a Borel subalgebra 
and a maximal torus of ${\goth g}^{x}$ respectively. Let $Z_{0}$ be the subset of 
elements of $Z$ contained in ${\goth b}$ and let $(B^{x})_{0}$ be the identity component 
of $B^{x}$. Since $Z$ is an irreducible component of $Z_{x}$, $Z$ is invariant under 
$(G^{x})_{0}$ and $Z=(G^{x})_{0}.Z_{0}$. Since $(G^{x})_{0}/(B^{x})_{0}$ is a projective 
variety, according to the proof of Lemma~\ref{l4int}, $(G^{x})_{0}.Z_{*}$ is a closed 
subset of $Z$ for all closed subset $Z_{*}$ of $Z$. Hence for some irreducible component 
$Z_{*}$ of $Z_{0}$, $Z=(G^{x})_{0}.Z_{*}$. 

For $\Lambda $ subset of ${\cal R}$, denote by $Z_{*,\Lambda }$ the subset of elements
$V$ of $Z_{*}$ such that 
$$ g(V) = {\goth h}_{\Lambda } \oplus g(V)\cap {\goth u} $$
for some $g$ in $(B^{x})_{0}$. According to Corollary~\ref{cmv2}, $Z_{*}$ is the union 
of $Z_{*,\Lambda },\Lambda \subset {\cal R}$. Since all element of $Z_{*,\Lambda }$
is contained in ${\goth h}_{\Lambda }+{\goth u}$, 
$$ \overline{Z_{*,\Lambda }} \subset 
\bigcup _{{\cal R} \supset \Lambda ' \supset \Lambda } Z_{*,\Lambda '} .$$
So, by induction on $\vert {\cal R}\setminus \Lambda  \vert$, $Z_{*,\Lambda }$ is 
a constructible subset of $Z_{*}$. Then, since ${\cal R}$ is finite, for some subset 
$\Lambda $ of ${\cal R}$, $Z_{*,\Lambda }$ is dense in $Z_{*}$. As a result, 
$(G^{x})_{0}.Z_{*,\Lambda }$ contains a dense open subset of $Z$ and for all $V$
in $(G^{x})_{0}.Z_{*,\Lambda }$, the biggest torus contained in $V$ is conjugate 
to ${\goth h}_{\Lambda }$ under $(G^{x})_{0}$.

(ii) For some $s$ in ${\goth s}$, ${\goth g}^{s}$ is the centralizer of ${\goth s}$ in 
${\goth g}$. Let $Z^{s}$ be the subset of elements of $Z$ containing $s$. Then $Z^{s}$
is contained in $Z_{s+x}$ and according to Corollary~\ref{cmv1},(i), $Z^{s}$ is the 
subset of elements of $Z$, containing ${\goth s}$. By (i), for some irreducible component
$Z'_{1}$ of $Z^{s}$, $(G^{x})_{0}.Z'_{1}$ is dense in $Z$. Let $Z_{1}$ be an irreducible 
component of $Z_{s+x}$, containing $Z'_{1}$. According to Corollary~\ref{cmv1},(ii), 
$Z_{1}$ is contained in $Z_{x}$ since $x$ is the nilpotent component of $s+x$. So 
$Z_{1}=Z'_{1}$ and $(G^{x})_{0}.Z_{1}$ is dense in $Z$.

(iii) Since $Z_{1}$ is an irreducible component of $Z_{s+x}$, $Z_{1}$ is invariant under 
the identity component of $G^{s+x}$. Moreover, $G^{s+x}$ is contained in $G^{x}$ since 
$x$ is the nilpotent component of $s+x$. As a result, by (ii), 
$$ \dim Z \leq \dim {\goth g}^{x}-\dim {\goth g}^{s+x} + \dim Z_{1}, $$
whence the assertion.  
\end{proof}

Denote by $C_{h}$ the $G$-invariant closed cone generated by $h$ with $h$ in ${\goth h}$
such that $\beta (h)=2$ for all $\beta $ in $\Pi $.

\begin{lemma}\label{l2mv3}
Suppose ${\goth g}$ semisimple. Let $\Gamma $ be the closure in 
$\ec {Gr}g{}{}{\rg}\times {\goth g}$ of the image of the map
$$ \xymatrix{\k^{*} \times G \ar[rr] &&  \ec {Gr}g{}{}{\rg}\times {\goth g}}, \qquad 
(t,g) \longmapsto (g({\goth h}),tg(h)) $$
and $\Gamma _{0}$ the intersection of $\Gamma $ and 
$\ec {Gr}g{}{}{\rg}\times {\goth N}_{{\goth g}}$.

{\rm (i)} The subvariety $\Gamma $ of $\ec {Gr}g{}{}{\rg}\times {\goth g}$ has 
dimension $2n+1$. Moreover, $\Gamma $ is contained in ${\cal E}$.

{\rm (ii)} The varieties $G.X$ and $C_{h}$ are the images of $\Gamma $ by the first and 
second projections respectively. 

{\rm (iii)} The subvariety $\Gamma _{0}$ of $\Gamma $ is equidimensional of codimension 
$1$. 

{\rm (iv)} For $x$ nilpotent in ${\goth g}$, the subvariety of elements $V$ of $G.X$, 
containing $x$ and contained in $\overline{G(x)}$, has dimension at most 
$\dim {\goth g}^{x}-\rg$.
\end{lemma}

\begin{proof}
(i) Since the stabilizer of $({\goth h},h)$ in $\k^{*}\times G$ equals $\{1\}\times H$,
$\Gamma $ has dimension $2n+1$. Since $tg(h)$ is in $g({\goth h})$ for 
all $(t,g)$ in $\k^{*}\times G$ and ${\cal E}$ is a closed subset of
$\ec {Gr}g{}{}{\rg}\times {\goth g}$, $\Gamma $ is contained in ${\cal E}$.

(ii) Since $\ec {Gr}g{}{}{\rg}$ is a projective variety, the image of $\Gamma $ by the 
second projection is closed in ${\goth g}$. So, it equals $C_{h}$ since it is contained 
in $C_{h}$ and it contains the cone generated by $G.h$. Let $Y$ be the image of $\Gamma $
by the first projection. Since $\Gamma $ is a closed subset of 
$\ec {Gr}g{}{}{\rg}\times {\goth g}$, invariant by the automorphisms 
$(V,x)\mapsto (V,tx)$ with $t$ in $\k^{*}$, $Y\times \{0\}$ is the intersection of 
$\Gamma $ and $\ec {Gr}g{}{}{\rg}\times \{0\}$. Then $Y$ is a closed subset of
$\ec {Gr}g{}{}{\rg}$ containing $G.{\goth h}$. Moreover $\Gamma $ is contained in the 
closed subset $G.X\times {\goth g}$ of $\ec {Gr}g{}{}{\rg}\times {\goth g}$. Hence 
$Y=G.X$.

(iii) The subvariety $C_{h}$ of ${\goth g}$ has dimension $2n+1$ and the nullvariety of 
$p_{1}$ in $C_{h}$ is contained in ${\goth N}_{{\goth g}}$ since it is the nullvariety in
${\goth g}$ of the polynomials $\poi p1{,\ldots,}{\rg}{}{}{}$. Hence 
${\goth N}_{{\goth g}}$ is the nullvariety of $p_{1}$ in $C_{h}$ and $\Gamma _{0}$ is the
nullvariety in $\Gamma $ of the function $(V,x)\mapsto p_{1}(x)$. So $\Gamma _{0}$ is 
equidimensional of codimension $1$ in $\Gamma $.

(iv) Let $T$ be the subset of elements $V$ of $G.X$, containing $x$ and contained in 
$\overline{G(x)}$. Denote by $\Gamma _{T}$ the inverse image of $\overline{G.T}$ by 
the projection $\xymatrix{\Gamma \ar[r] & G.X}$. Then $\Gamma _{T}$ is contained in 
$\Gamma _{0}$. Since all element of $T$ contains $x$ and is contained in 
$\overline{G(x)}$ and since $\Gamma _{T}$ is invariant under $G$, the image of 
$\Gamma _{T}$ by the second projection is equal to $\overline{G(x)}$. Moreover,
$T \times \{x\}\subset G.X \times \{x\}\cap \Gamma _{T}$. Hence 
$$ \dim \Gamma _{T} \geq  \dim T + \dim {\goth g}-\dim {\goth g}^{x} .$$
By (i) and (iii), 
$$\dim \Gamma _{T}\leq \dim {\goth g}-\rg $$
since $\Gamma _{T}$ is contained in $\Gamma _{0}$. Hence $T$ has dimension 
at most $\dim {\goth g}^{x}-\rg$.
\end{proof}

When ${\goth g}$ is semisimple, denote by $(G.X)_{\u}$ the subset of elements of 
$G.X$ contained in ${\goth N}_{{\goth g}}$.

\begin{coro}\label{cmv3}
Suppose ${\goth g}$ semisimple. Let $x$ be in ${\goth N}_{{\goth g}}$.

{\rm (i)} The variety $(G.X)_{\u}$ has dimension at most $2n-\rg$.

{\rm (ii)} The variety $Z_{x}\cap (G.X)_{\u}$ has dimension at most 
$\dim {\goth g}^{x}-\rg$.
\end{coro}

\begin{proof}
(i) Let $T$ be an irreducible component of $(G.X)_{\u}$ and let ${\cal E}_{T}$ be 
the restriction to $T$ of the vector bundle ${\cal E}$ over $G.X$. Then ${\cal E}_{T}$ is
irreducible and has dimension $\dim T+\rg$. Denoting by $Y$ the image of the projection 
$\xymatrix{{\cal E}_{T} \ar[r] & {\goth g}}$, $Y$ is an irreducible closed subvariety of 
${\goth g}$ contained in ${\goth N}_{{\goth g}}$. The subvariety $(G.X)_{\u}$ of $G.X$ is 
invariant under $G$ since so is ${\goth N}_{{\goth g}}$. Hence ${\cal E}_{T}$ and 
$Y$ are $G$-invariant and for some $y$ in ${\goth N}_{{\goth g}}$, $Y=\overline{G(y)}$
since ${\goth N}_{{\goth g}}$ is a finite union of orbits. Denoting by $F_{y}$ the fiber 
at $y$ of the projection $\xymatrix{{\cal E}_{T}\ar[r]& Y}$, $V$ is contained in 
$\overline{G(y)}$ and contains $y$ for all $V$ in $F_{y}$. So, by Lemma~\ref{l2mv3},(iv), 
$$ \dim F_{y} \leq \dim {\goth g}^{y} - \rg .$$ 
Since the projection is $G$-equivariant, this inequality holds for the fibers at the
elements of $G(y)$. Hence, 
$$ \dim {\cal E}_{T} \leq \dim {\goth g} - \rg \mbox{ and }
\dim T \leq 2n - \rg .$$ 

(ii) Let $Z$ be an irreducible component of $Z_{x}\cap (G.X)_{\u}$ and let $T$ be 
an irreducible component of $(G.X)_{\u}$, containing $Z$. Let ${\cal E}_{T}$ and $Y$ be
as in (i). Then $G(x)$ is contained in $Y$ and the inverse image of $\overline{G(x)}$ 
in ${\cal E}_{T}$ has dimension at least $\dim G(x)+\dim Z$. So, by (i),
$$ \dim G(x) + \dim Z \leq \dim {\goth g} - \rg ,$$
whence the assertion.
\end{proof}

\begin{theo}\label{tmv3}
For $x$ in ${\goth g}$, the variety of elements of $G.X$, containing $x$, has 
dimension at most $\dim {\goth g}^{x}-\rg$.
\end{theo}

\begin{proof}
Prove the theorem by induction on $\dim {\goth g}$. If ${\goth g}$ is commutative,
$G.X=\{{\goth g}\}$. If the derived Lie algebra of ${\goth g}$ is simple of dimension 
$3$, $G.X$ has dimension $2$ and for $x$ not in the center of ${\goth g}$, 
$Z_{x}=\{{\goth g}^{x}\}$. Suppose the theorem true for all reductive 
Lie algebra of dimension strictly smaller than $\dim {\goth g}$. Let $x$ be in 
${\goth g}$. Since $G.X$ has dimension $\dim {\goth g}-\rg$, we can suppose that 
$x$ is not in the center of ${\goth g}$. Suppose that $x$ is not nilpotent. Then 
${\goth g}^{x_{\s}}$ has dimension strictly smaller than $\dim {\goth g}$ and all element
of $G.X$ containing $x$ is contained in ${\goth g}^{x_{\s}}$ and contains the center
of ${\goth g}^{x_{\s}}$ by Corollary~\ref{cmv1},(i). So, by Lemma~\ref{lmv2},(ii),
$Z_{x}$ is contained in $\overline{G^{x_{\s}}.{\goth h}}$, whence the theorem in this 
case by induction hypothesis. As a result, by Lemma~\ref{lmv3}, for all $x$ in 
${\goth g}$, all irreducible component of $Z_{x}$, containing an element not contained in
${\goth N}_{{\goth g}}$, has dimension at most $\dim {\goth g}^{x}-\rg$. 

Let $x$ be a nilpotent element of ${\goth g}$. Denoting by $Z'_{x}$ the subset of 
elements of $\overline{G.({\goth h}\cap [{\goth g},{\goth g}]})$ containing $x$, $Z_{x}$ 
is the image of $Z'_{x}$ by the map $V\mapsto V+{\goth z}_{{\goth g}}$, whence the 
theorem by Corollary~\ref{cmv3}.
\end{proof}

\subsection{} \label{mv4}
Let $s$ be in ${\goth h}\setminus \{0\}$. Set ${\goth p} := {\goth g}^{s}+{\goth b}$
and denote by ${\goth p}_{\u}$ the nilpotent radical of ${\goth p}$. Let $P$ be
the normalizer of ${\goth p}$ in $G$ and let $P_{\u}$ be its unipotent radical. For a 
nilpotent orbit $\Omega $ of $G^{s}$ in ${\goth g}^{s}$, denote by $\Omega ^{\#}$ the
induced orbit by $\Omega $ from ${\goth g}^{s}$ to ${\goth g}$. 

\begin{lemma}\label{lmv4}
Let $Y$ be a $G$-invariant irreducible closed subset of ${\goth g}$ and let $Y'$ be the 
union of $G$-orbits of maximal dimension in $Y$. Suppose that $s$ is the semisimple 
component of an element $x$ of $Y'$. Denote by $\Omega $ the orbit of $x_{\n} $ under 
$G^{s}$ and set $Y_{1} := {\goth z}_{s} + \overline{\Omega } + {\goth p}_{\u}$.

{\rm (i)} The subset $Y_{1}$ of ${\goth p}$ is closed and invariant under $P$. 
                                   
{\rm (ii)} The subset $G(Y_{1})$ of ${\goth g}$ is a closed subset of dimension 
$\dim {\goth z}_{s}+\dim G(x)$.

{\rm (iii)} For some nonempty open subset $Y''$ of $Y'$, the conjugacy class of 
${\goth g}^{y_{\s}}$ under $G$ does not depend on the element $y$ of $Y''$.

{\rm (iv)} For a good choice of $x$ in $Y''$, $Y$ is contained in $G(Y_{1})$.
\end{lemma}

\begin{proof}
(i) By~\cite[\S 3.2, Lemma 5]{Ko}, $G^{s}$ is connected and $P=P_{\u}G^{s}$. 
For all $y$ in ${\goth p}$ and for all $g$ in $P_{\u}$, $g(y)$ is in $y+{\goth p}_{\u}$. 
Hence $Y_{1}$ is invariant under $P$ since it is invariant under $G^{s}$. Moreover,
it is a closed subset of ${\goth p}$ since ${\goth z}_{s}+\overline{\Omega }$ is a closed 
subset of ${\goth g}^{s}$.    

(ii) According to (i) and Lemma~\ref{l4int}, $G(Y_{1})$ is a closed subset of 
${\goth g}$. According to \cite[Theorem 7.1.1]{CMa},  
$\Omega ^{\#}\cap (\Omega +{\goth p}_{\u})$ is a $P$-orbit and the centralizers in 
${\goth g}$ of its elements are contained in ${\goth p}$. For $y$ in 
$\Omega ^{\#}\cap (\Omega +{\goth p}_{\u})$ and for $g$ in $G$, if $g(y)$ is in
$Y_{1}$ then it is in $\Omega +{\goth p}_{\u}$ since it is nilpotent. So, for $y$ in 
$\Omega ^{\#}\cap (\Omega +{\goth p}_{\u})$, the subset of elements $g$ of $G$ such that
$g(y)$ is in $Y_{1}$ has dimension $\dim {\goth p}$. As a result, 
$$ \dim G(Y_{1}) = \dim G\times _{P}Y_{1} = 
\dim {\goth p}_{\u} + \dim Y_{1} .$$
Since $\dim {\goth g}^{x} = \dim {\goth g}^{s} - \dim \Omega $,
\begin{eqnarray*}
\dim Y_{1} = & \dim {\goth z}_{s} + \dim {\goth p}_{\u} + \dim {\goth g}^{s} 
- \dim {\goth g}^{x} \\ 
\dim G(Y_{1}) = & \dim {\goth z}_{s} + 2\dim {\goth p}_{\u} + 
\dim {\goth g}^{s} - \dim {\goth g}^{x} \\  = & \dim {\goth z}_{s} + \dim G(x) .
\end{eqnarray*}

(iii) Let $\tauup $ be the canonical morphism from ${\goth g}$ to its categorical quotient
${\goth g}//G$ under $G$ and let $Z$ be the closure in ${\goth g}//G$ of 
$\tauup (Y)$. Since $Y$ is irreducible, $Z$ is irreducible and there exists an 
irreducible component $\widetilde{Z}$ of the preimage of $Z$ in ${\goth h}$ whose image 
in ${\goth g}//G$ equals $Z$. Since the set of conjugacy classes under $G$ of the 
centralizers of the elements of ${\goth h}$ in ${\goth g}$ is finite, for some nonempty 
open subset $Z^{\#}$ of $\widetilde{Z}$, the centralizers of its elements are conjugate 
under $G$. The image of $Z^{\#}$ in ${\goth g}//G$ contains a dense open subset $Z'$ 
of $Z$. Let $Y''$ be the inverse image of $Z'$ by the restriction of $\tauup $ to $Y'$. 
Then $Y''$ is a dense open subset of $Y$ and the centralizers in ${\goth g}$ of the 
semisimple components of its elements are conjugate under $G$. 

(iv) Suppose that $x$ is in $Y''$. Let $Z_{Y}$ be the set of elements $y$ of $Y''$ 
such that ${\goth g}^{y_{\s}}={\goth g}^{s}$. Then $G.Z_{Y}=Y''$. For all nilpotent orbit 
$\Omega $ of $G^{s}$ in ${\goth g}^{s}$, set:
$$ Y_{\Omega } = {\goth z}_{s} + \overline{\Omega } + {\goth p}_{\u}$$    
Then $Z_{Y}$ is contained in the union of the $Y_{\Omega }$'s. Hence $Y''$ is contained
in the union of the $G(Y_{\Omega })$'s. According to (ii), $G(Y_{\Omega })$ is 
a closed subset of ${\goth g}$. Hence $Y$ is contained in the union of the 
$G(Y_{\Omega })$'s since $Y''$ is dense in $Y$. Then $Y$ is contained in 
$G(Y_{\Omega })$ for some $\Omega $ since $Y$ is irreducible and there are finitely 
many nilpotent orbits in ${\goth g}^{s}$, whence the assertion.
\end{proof}

\begin{theo}\label{tmv4}
(i) The variety $G.X$ is the union of $G.{\goth h}$ and the $G.X_{\beta }$'s, 
$\beta \in \Pi $.

{\rm (ii)} The variety $X$ is the union of $U.{\goth h}$ and the $X_{\alpha }$'s, 
$\alpha \in {\cal R}_{+}$.
\end{theo}

\begin{proof}
Let $\mu $ be the map
$$ \xymatrix{\ec {Gr}{}{[{\goth g},{\goth g}]}{}{\rg'} \ar[rr] &&
\ec {Gr}g{}{}{\rg}}, \qquad V\longmapsto {\goth z}_{{\goth g}}+V $$ 
with $\rg'$ the rank of $[{\goth g},{\goth g}]$ and set:
$$X_{d} := \overline{B.({\goth h}\cap [{\goth g},{\goth g}])}, \qquad
X_{\alpha ,d} := \overline{B.(V_{\alpha }\cap [{\goth g},{\goth g}])} $$
for $\alpha $ in ${\cal R}_{+}$. Then $X$, $G.X$, $X_{\alpha }$, $G.X_{\alpha }$ are the 
images of $X_{d}$, $G.X_{d}$, $X_{\alpha ,d}$, $G.X_{\alpha ,d}$ by $\mu $ respectively.
So we can suppose ${\goth g}$ semisimple.

(i) For $\rg =1$, ${\goth g}$ is simple of dimension $3$. In this case, $G.X$ is the 
union of $G.{\goth h}$ and $G.{\goth g}^{e}$. So, we can suppose $\rg \geq 2$. 
According to Lemma~\ref{lmv1},(iii), for $\alpha $ in ${\cal R}_{+}$, $G.X_{\alpha }$ is 
an irreducible component of $G.X\setminus G.{\goth h}$. Moreover, for all $\beta $ in 
$\Pi \cap W({\cal R})(\alpha )$, $G.X_{\alpha }=G.X_{\beta }$ since
$V_{\alpha }$ and $V_{\beta }$ are conjugate under $N_{G}({\goth h})$.

Let $T$ be an irreducible component of $G.X\setminus G.{\goth h}$. Set:
$${\cal E}_{T} := {\cal E}\cap T\times {\goth g}$$
and denote by $Y$ the image of ${\cal E}_{T}$ by the second projection. Then $Y$ is 
closed in ${\goth g}$ since $\ec {Gr}g{}{}{\rg}$ is a projective variety. Since 
${\cal E}_{T}$ is a vector bundle over $T$ and since $T$ is irreducible, ${\cal E}_{T}$ 
is irreducible and so is $Y$. Since $T$ is an irreducible component of 
$G.X\setminus G.{\goth h}$, $T$, ${\cal E}_{T}$ and $Y$ are $G$-invariant. According to
Lemma~\ref{lmv1},(iii), $T$ has codimension $1$ in $G.X$. Hence, by 
Corollary~\ref{cmv3},(i) $Y$ is not contained in the nilpotent cone since $\rg \geq 2$. 
Let $Y'$ be the set of elements $x$ of $Y$ such that ${\goth g}^{x}$ has minimal 
dimension. According to Lemma~\ref{lmv4},(ii) and (iv), for some $x$ in $Y'$,
$$ \dim Y \leq \dim G(x) + \dim {\goth z}_{x_{\s}}$$
and according to Theorem~\ref{tmv3}, 
$$ \dim {\cal E}_{T} \leq \dim G(x) + \dim {\goth z}_{x_{\s}} + \dim {\goth g}^{x}-\rg
= \dim {\goth g} + \dim {\goth z}_{x_{\s}}- \rg $$
Hence ${\cal E}_{T}$ has dimension at most $2n + \dim {\goth z}_{x_{\s}}$ and 
$\dim {\goth z}_{x_{\s}}=\rg -1$ since $T$ has codimension $1$ in $G.X$. As a result, 
$x_{\s}$ is subregular and for some $g$ in $G$, $g({\goth z}_{x_{\s}})$ is the kernel of 
a positive root $\alpha $. Denoting by ${\goth s}_{\alpha }$ the subalgebra of 
${\goth g}$ generated by ${\goth g}^{\alpha }$ and ${\goth g}^{-\alpha }$, 
${\goth g}^{g(x_{\s})}$ is the direct sum of ${\goth h}_{\alpha }$ and 
${\goth s}_{\alpha }$. Since the maximal commutative subalgebras of 
${\goth s}_{\alpha }$ have dimension $1$, a commutative subalgebra of 
dimension $\rg$ of ${\goth g}^{g(x_{\s})}$ is either a Cartan subalgebra of ${\goth g}$ 
or conjugate to $V_{\alpha }$ under the adjoint group of ${\goth g}^{g(x_{\s})}$. As a 
result, $V_{\alpha }$ is in $T$ and $T=\overline{G.V_{\alpha }}=G.X_{\alpha }$ since $T$ 
is $G$-invariant, whence the assertion.

(ii) According to Lemma~\ref{lmv1},(ii), for $\alpha $ in ${\cal R}_{+}$, $X_{\alpha }$ is
an irreducible component of $X\setminus B.{\goth h}$. Let 
$\poi {{\goth g}}1{,\ldots,}{m}{}{}{}$ be the simple factors of ${\goth g}$. For 
$j=1,\ldots,m$, denote by $X_{j}$ the closure in $\ec {Gr}g{}j{\j gj}$ of the orbit of 
${\goth h}\cap {\goth g}_{j}$. Then $X=\poi X1{\times \cdots \times }{m}{}{}{}$ and 
the complement to $B.{\goth h}$ in $X$ is the union of the 
$$ \poi X1{\times \cdots \times }{j-1}{}{}{}
\times (X_{j}\setminus B.({\goth h}\cap {\goth g}_{j})) \times 
\poi X{j+1}{\times \cdots \times }{m}{}{}{}$$ 
So, we can suppose ${\goth g}$ simple. Consider
$${\goth b}=\poi {{\goth p}}0{ \subset \cdots \subset }{\rg}{}{}{} = {\goth g}$$
an increasing sequence of parabolic subalgebras verifying the following condition:
for $i=0,\ldots,\rg -1$, there is no parabolic subalgebra ${\goth q}$ of ${\goth g}$ such
that
$$ {\goth p}_{i} \subsetneqq {\goth q} \subsetneqq {\goth p}_{i+1} .$$
For $i=0,\ldots,\rg$, let $P_{i}$ be the normalizer of ${\goth p}_{i}$ in $G$, let 
${\goth p}_{i,\u}$ be the nilpotent radical of ${\goth p}_{i}$ and let $P_{i,\u}$ be the 
unipotent radical of $P_{i}$. For $i=0,\ldots,\rg$ and for $\alpha $ in ${\cal R}_{+}$, 
set $X_{i}:= P_{i}.X$ and $X_{i,\alpha }:=P_{i}.X_{\alpha }$. Prove by 
induction on $\rg -i$ that for all sequence of parabolic subalgebras verifying the 
above condition, the $X_{i,\alpha }$'s, $\alpha \in {\cal R}_{+}$, are the irreducible
components of $X_{i}\setminus P_{i}.{\goth h}$.

For $i=\rg$, it results from (i). Suppose that it is true for $i+1$. According to
Lemma~\ref{lmv1},(iii), the $X_{i,\alpha }$'s are irreducible components of 
$X_{i}\setminus P_{i}.{\goth h}$.

\begin{claim}\label{clmv4}
Let $T$ be an irreducible component of $X_{i}\setminus P_{i}.{\goth h}$ such that $P_{i}$
is its stabilizer in $P_{i+1}$. Then $T=X_{i,\alpha }$ for some $\alpha $ in 
${\cal R}_{+}$.
\end{claim}

\begin{proof}
According to the induction hypothesis, $T$ is contained in $X_{i+1,\alpha }$ for 
some $\alpha $ in ${\cal R}_{+}$. According to Lemma~\ref{lmv1},(iv), $T$ has codimension
$1$ in $X_{i}$ so that $P_{i+1}.T$ and $X_{i+1,\alpha }$ have the same dimension. Then 
they are equal and $T$ contains ${\goth g}^{x}$ for some $x$ in ${\goth b}_{\r}$ such 
that $x_{\s}$ is a subregular element belonging to ${\goth h}$. Denoting by $\alpha '$ the
positive root such that $\alpha '(x_{\s})=0$, ${\goth g}^{x}=V_{\alpha '}$ since 
$V_{\alpha '}$ is the commutative subalgebra contained in ${\goth b}$ and containing
${\goth h}_{\alpha '}$, which is not Cartan, so that $T=X_{i,\alpha '}$.
\end{proof}

Suppose that $X_{i}\setminus P_{i}.{\goth h}$ is not the union of the 
$X_{i,\alpha }$'s, $\alpha \in {\cal R}_{+}$. We expect a contradiction. Let $T$ be 
an irreducible component of $X_{i}\setminus P_{i}.{\goth h}$, different from 
$X_{i,\alpha }$ for all $\alpha $. According to Claim~\ref{clmv4} and according to the 
condition verified by the sequence, $T$ is invariant under $P_{i+1}$. Moreover, according
to Claim~\ref{clmv4}, it is so for all sequence 
$\poie {{\goth p}}0{,\ldots,}{\rg}{}{}{}{\prime}{\prime}$ of parabolic 
subalgebras verifying the above condition and such that ${\goth p}'_{j}={\goth p}_{j}$
for $j=0,\ldots,i$. As a result, for all simple root $\beta $ such that 
${\goth g}^{-\beta }$ is not in ${\goth p}_{i}$, $T$ is invariant under the one parameter
subgroup of $G$ generated by $\ad {\goth g}^{-\beta }$. Hence $T$ is invariant under 
$G$. It is impossible since for $x$ in ${\goth g}\setminus \{0\}$, the orbit $G(x)$
is not contained in ${\goth p}_{i}$ since ${\goth g}$ is simple, whence the assertion.
\end{proof}

\subsection{} \label{mv5}
Let $X'$ be the subset of ${\goth g}^{x}$ with $x$ in ${\goth b}_{\r}$ such that 
$x_{\s}$ is regular or subregular. For $\alpha $ in ${\cal R}_{+}$, denote by 
$\thetaup _{\alpha }$ the map
$$ \xymatrix{\k \ar[rr] && X} , \qquad t \longmapsto \exp(t\ad x_{\alpha }).{\goth h} .$$
According to ~\cite[Ch.~VI, Theorem 1]{Sh}, $\thetaup _{\alpha }$ has a regular extension
to ${\Bbb P}^{1}(\k)$, also denoted by $\thetaup _{\alpha }$. Set 
$Z_{\alpha } := \thetaup _{\alpha }({\Bbb P}^{1}(\k))$ and $X'_{\alpha }:=B.Z_{\alpha }$
so that $X'_{\alpha }=U.{\goth h}\cup B.V_{\alpha }$.

\begin{lemma}\label{lmv5}
Let $\alpha $ be in ${\cal R}_{+}$ and let $V$ be in $X$. Then $V$ is in $B.Z_{\alpha }$
if and only if $g(V)$ contains ${\goth h}_{\alpha }$ for some $g$ in $B$. 
\end{lemma}

\begin{proof}
The condition is necessary by definition. Suppose that $V$ contains ${\goth h}_{\alpha }$.
Since $V$ is commutative by Corollary~\ref{cmv1},(ii), $V$ is contained in 
${\goth h}+{\goth g}^{\alpha }$. If $V$ is a Cartan subalgebra, then 
$V=\thetaup _{\alpha }(t)$ for some $t$ in $\k$. Otherwise, 
$V=\thetaup _{\alpha }(\infty )$, whence the lemma.
\end{proof}

\begin{coro}\label{cmv5}
Let $\alpha $ be a positive root.

{\rm (i)} The sets $X'_{\alpha }$ and $G.X'_{\alpha }$ are open subsets of $X$ and  
$G.X$ respectively.

{\rm (ii)} The sets $X'$ and $G.X'$ are big open subsets of $X$ and  
$G.X$ respectively.
\end{coro}

\begin{proof}
(i) Since $X'_{\alpha }$ is a $B$-invariant subset containing the open subset 
$U.{\goth h}$, it suffices to prove that $X'_{\alpha }$ is a neighborhood of 
$V_{\alpha }$ in $X$. Denote by $H_{\alpha }$ the coroot of $\alpha $ and set:
$$E' := \bigoplus _{\gamma \in {\cal R}_{+}\setminus \{\alpha \}} {\goth g}^{\gamma } ,
\qquad E := \k H_{\alpha } \oplus E' .$$
Let $\Omega _{E}$ be the set of subspaces $V$ of ${\goth b}$ such that $E$ is a complement
to $V$ in ${\goth b}$ and let $\Omega '_{E}$ be the complement in $X\cap \Omega _{E}$ to 
the union of $X_{\gamma }, \gamma \in {\cal R}_{+}\setminus \{\alpha \}$. Then  
$\Omega _{E}'$ is an open neighborhood of $V_{\alpha }$ in $X$. Since $X'_{\alpha }$
contains $U.{\goth h}$, $X'_{\alpha }$ contains all the Cartan subalgebras contained 
in $\Omega '_{E}$. Let $V$ be in $\Omega '_{E}$ such that $V$ is not a Cartan subalgebra.
According to Corollary~\ref{cmv2}, for some nonempty subset $\Lambda $ of ${\cal R}$, 
$V$ is contained ${\goth h}_{\Lambda }+{\goth u}$ and contains a conjugate of 
${\goth h}_{\Lambda }$ under $B$. Then ${\goth h}=\k H_{\alpha }+{\goth h}_{\Lambda }$ 
since $V$ is in $\Omega _{E}$. As a result, ${\goth h}_{\Lambda }={\goth h}_{\gamma }$ 
for some positive root $\gamma $ and $V$ is conjugate to $V_{\gamma }$ under $B$ by 
Lemma~\ref{lmv5}. Since $V$ is not in $X_{\delta }$ for all $\delta $ in 
${\cal R}_{+}\setminus \{\alpha \}$, $\gamma =\alpha $ and $V$ is in $X'_{\alpha }$.
Then $X'_{\alpha }$ contains $\Omega '_{E}$. As a result, $X'_{\alpha }$ is an open 
subset of $X$ and $G.(X\setminus X'_{\alpha })$ is a closed subset of $G.X$ by 
Lemma~\ref{l4int}, whence the assertion.

(ii) By definition, $X'$ is the union of $X'_{\alpha }, \alpha \in {\cal R}_{+}$. 
Hence $X'$ is an open subset of $X$ by (i). Moreover, by Theorem~\ref{tmv4},(ii), 
$X\setminus X'$ is the union of $X_{\alpha }\setminus X',\alpha \in {\cal R}_{+}$. Then 
$X'$ is a big open subset of $X$ since, for all $\alpha $, $X_{\alpha }\setminus X'$ is 
strictly contained in the irreducible subvariety $X_{\alpha }$ of $X$.

Since $G.X'$ is the union of $G.X'_{\alpha },\alpha \in {\cal R}_{+}$, 
$G.X'$ is an open subset of $G.X$ by (i). Moreover, by Theorem~\ref{tmv4},(i), 
$G.X\setminus G.X'$ is the union of $G.X_{\beta }\setminus G.X',\beta \in \Pi $.
Hence $G.X'$ is a big open subset of $G.X$ since, for all $\beta $, 
$G.X_{\beta }\setminus G.X'$ is strictly contained in the irreducible subvariety
$G.X_{\beta }$ of $G.X$.
\end{proof}

\begin{prop}\label{pmv5}
The sets $X'$ and $G.X'$ are smooth big open subsets of $X$ and $G.X$ respectively.
\end{prop}

\begin{proof}
According to Corollary~\ref{cmv5},(ii), it remains to prove that $X'$ and $G.X'$
are smooth open subsets of $X$ and $G.X$ respectively. Denote by $\pi $ the 
bundle projection of the vector bundle ${\cal E}$ over $G.X$. Recall 
${\cal E}_{0}:=\pi ^{-1}(X)$. Let $\mu $ be the map
$$ \xymatrix{{\goth g}_{\r} \ar[rr] && \ec {Gr}g{}{}{\rg}}, \qquad 
x \longmapsto {\goth g}^{x} $$
and let $\mu _{0}$ be its restriction to ${\goth b}_{\r}$. Then $\mu $ is a regular map. 
Let $\Gamma _{\mu }$ and $\Gamma _{\mu _{0}}$ be the images of the graphs of $\mu $ and 
$\mu _{0}$ respectively by the isomorphism 
$$ \xymatrix{ {\goth g} \times \ec {Gr}g{}{}{\rg} \ar[rr] && 
\ec {Gr}g{}{}{\rg}\times {\goth g}}, \qquad (x,V) \longmapsto (V,x) .$$
Then $\Gamma _{\mu }$ and $\Gamma _{\mu _{0}}$ are smooth varieties contained in 
${\cal E}$ and ${\cal E}_{0}$ respectively since for $x$ in ${\goth g}_{\rs}$, 
${\goth g}^{x}$ is a Cartan subalgebra, contained in ${\goth b}$ when $x$ is in 
${\goth b}$. Set:
$$ \Gamma _{\mu }' := \Gamma _{\mu }\cap \pi ^{-1}(G.X') = 
{\cal E}\cap G.X'\times {\goth g}_{\r} \quad  \text{and} \quad
\Gamma _{\mu _{0}}' := \Gamma _{\mu _{0}}\cap \pi ^{-1}(X') = 
{\cal E}\cap X'\times {\goth b}_{\r} .$$ 
Then $\Gamma _{\mu }'$ is a smooth variety as an open susbet of $\Gamma _{\mu }$
and $\Gamma _{\mu }'$ is an open subset of $\pi ^{-1}(G.X')$ such that 
$\pi (\Gamma _{\mu }')=G.X'$ since all element of $G.X'$ contains regular elements. 
In the same way, $\Gamma _{\mu _{0}}'$ is a smooth open subset of $\pi ^{-1}(X')$
such that $\pi (\Gamma _{\mu _{0}}')=X'$. As a result, $\Gamma _{\mu }'$ and
$\Gamma _{\mu _{0}}'$ are smooth open subsets of vector bundles over $G.X'$ and 
$X'$ respectively since ${\cal E}$ and ${\cal E}_{0}$ are vector bundles over $G.X$ and 
$X$ respectively. Hence $G.X'$ and $X'$ are smooth varieties by 
\cite[Ch. 8, Theorem 23.7]{Mat}. 
\end{proof}

Summarizing the results of the section, Theorem~\ref{t2int},(i) is given by
Corollay~\ref{cmv1},(ii), Theorem~\ref{t2int},(ii) is given by Theorem~\ref{tmv3},
Theorem~\ref{t2int},(iii) is given by Lemma~\ref{lmv1},(iv) since $X$ and $G.X$ have
dimension $n$ and $2n$ respectively and Theorem~\ref{t2int},(iv) is given by 
Proposition~\ref{pmv5}.

\section{On the generalized isospectral commuting variety} \label{isc}
Let $k\geq 2$ be an integer. According to Section~\ref{bo}, we have the 
commutative diagram
$$\xymatrix{\sqx G{{\goth b}^{k}} \ar[rr]^{\gamma _{\x}} \ar[rd]_{\gamma }
&& {\cal B}_{\x}^{(k)} \ar[ld]^{\eta } \\ & {\cal B}^{(k)} & }.$$ 
By Lemma~\ref{lbo3},(i), $\iota _{k}$ is a closed embedding of 
${\goth b}^{k}$ into ${\cal B}_{\x}^{(k)}$, by Corollary~\ref{cbo3},(i)
${\cal B}_{\x}^{(k)}=G.\iota _{k}({\goth b}^{k})$ is closed in ${\cal X}^{k}$ and 
$\eta $ is the restriction to ${\cal B}_{\x}^{(k)}$ of the canonical projection from 
${\cal X}^{k}$ to ${\goth g}^{k}$. Denote by ${\cal C}^{(k)}$ the closure of 
$G.{\goth h}^{k}$ in ${\goth g}^{k}$ with respect to the diagonal action of $G$ in 
${\goth g}^{k}$ and set ${\cal C}_{\x}^{(k)} := \eta ^{-1}({\cal C}^{(k)})$. 
The varieties ${\cal C}^{(k)}$ and ${\cal C}_{\x}^{(k)}$ are called 
{\it generalized commuting variety} and {\it generalized isospectral commuting variety} 
respectively. For $k=2$, ${\cal C}_{\x}^{(k)}$ is the isospectral commuting variety 
considered by M. Haiman in~\cite[\S 8]{Ha1} and~\cite[\S 7.2]{Ha2}. 

\subsection{} \label{isc1}
Set:
$$ {\cal E}_{0}^{(k)} :=  \{(u,\poi x1{,\ldots,}{k}{}{}{}) \in X\times {\goth b}^{k} 
\ \vert \ \poi {u \ni x}1{,\ldots,}{k}{}{}{}\} .$$

\begin{lemma}\label{lisc1}
Denote by ${\cal E}_{0}^{(k,*)}$ the intersection of ${\cal E}_{0}^{(k)}$ and 
$U.{\goth h}\times ({\goth g}_{\rs}\cap {\goth b})^{k}$ and for $w$ in $W({\cal R})$, 
denote by $\theta _{w}$ the map
$$ {\cal E}_{0}^{(k)} \longrightarrow {\goth b}^{k}\times {\goth h}^{k} , \qquad
(u,\poi x1{,\ldots,}{k}{}{}{}) \longmapsto 
(\poi x1{,\ldots,}{k}{}{}{},w(\overline{x_{1}}),\ldots,w(\overline{x_{k}})) .$$ 

{\rm (i)} Denoting by ${\goth X}_{0,k}$ the image of ${\cal E}_{0}^{(k)}$ by the 
projection $(u,\poi x1{,\ldots,}{k}{}{}{})\mapsto (\poi x1{,\ldots,}{k}{}{}{})$, 
${\goth X}_{0,k}$ is the closure of $B.{\goth h}^{k}$ in ${\goth b}^{k}$ and 
${\cal C}^{(k)}$ is the image of $G\times {\goth X}_{0,k}$ by the map 
$(g,\poi x1{,\ldots,}{k}{}{}{})\mapsto (\poi x1{,\ldots,}{k}{g}{}{})$.

{\rm (ii)} For all $w$ in $W({\cal R})$, $\theta _{w}({\cal E}_{0}^{(k,*)})$ is dense in 
$\theta _{w}({\cal E}_{0}^{(k)})$.
\end{lemma}

\begin{proof}
(i) Since $X$ is a projective variety, ${\goth X}_{0,k}$ is a closed subset of 
${\goth b}^{k}$. The variety ${\cal E}_{0}^{(k)}$ is irreducible of dimension $n+k \rg$ 
as a vector bundle of rank $k \rg$ over the irreducible variety $X$. So, 
$B.(\{{\goth h}\}\times {\goth h}^{k})$ is dense in ${\cal E}_{0}^{(k)}$ and 
${\goth X}_{0,k}$ is the closure of $B.{\goth h}^{k}$ in ${\goth b}^{k}$, whence the 
assertion by Lemma~\ref{l4int}.

(ii) Since $U.{\goth h}\times ({\goth g}_{\rs}\cap {\goth b})^{k}$ is an open susbet of 
$X\times {\goth b}^{k}$, ${\cal E}_{0}^{(k,*)}$ is an open subset of 
${\cal E}_{0}^{(k)}$. Moreover, it is a dense open subset since ${\cal E}_{0}^{(k)}$ is 
irreducible, whence the assertion since $\theta _{w}$ is a morphism of algebraic 
varieties.
\end{proof}

\subsection{} \label{isc2}
Let $s$ be in ${\goth h}$. According to~\cite[\S 3.2, Lemma 5]{Ko}, $G^{s}$ is connected.
Denote by ${\cal R}_{s}$ the set of roots whose kernel contains $s$ and denote
by $W({\cal R}_{s})$ the Weyl group of ${\cal R}_{s}$. 

\begin{lemma}\label{lisc2}
Let $x=(\poi x1{,\ldots,}{k}{}{}{})$ be in ${\cal C}^{(k)}$ verifying the following 
conditions:
\begin{itemize}
\item [{\rm (1)}] $s$ is the semisimple component of $x_{1}$,
\item [{\rm (2)}] for $z$ in $E_{x}$, the centralizer in ${\goth g}$ of the semisimple
component of $z$ has dimension at least $\dim {\goth g}^{s}$.
\end{itemize}
Then for $i=1,\ldots,k$, the semisimple component of $x_{i}$ is in ${\goth z}_{s}$.
\end{lemma}

\begin{proof}
Since $x$ is in ${\cal C}^{(k)}$, $[x_{i},x_{j}]=0$ for all $(i,j)$. Suppose that 
for some $i$, the semisimple component $x_{i,\s}$ of $x_{i}$ is not in ${\goth z}_{s}$.
A contradiction is expected. Since $[x_{1},x_{i}]=0$, for all $t$ in $\k$, 
$s+tx_{i,\s}$ is the semisimple component of $x_{1}+tx_{i}$. Moreover, after 
conjugation by an element of $G^{s}$, we can suppose that $x_{i,\s}$ is in ${\goth h}$. 
Since ${\cal R}$ is finite, there exists $t$ in $\k^{*}$ such that the subset of roots 
whose kernel contains $s+tx_{i,\s}$ is contained in ${\cal R}_{s}$. Since $x_{i,\s}$ is 
not in ${\goth z}_{s}$, for some $\alpha $ in ${\cal R}_{s}$, 
$\alpha (s+tx_{i,\s})\neq 0$ that is ${\goth g}^{s+tx_{i,\s}}$ is 
strictly contained in ${\goth g}^{s}$, whence the contradiction.
\end{proof}

For $w$ in $W({\cal R})$, set: 
$$C_{w} := G^{s}wB/B , \qquad B^{w} := wBw^{-1} .$$

\begin{lemma}\label{l2isc2}~\cite[\S 6.17, Lemma]{Hu}
Let ${\goth B}$ be the set of Borel subalgebras of ${\goth g}$ and let ${\goth B}_{s}$ be 
the set of Borel subalgebras of ${\goth g}$ containing $s$.

{\rm (i)} For all $w$ in $W({\cal R})$, $C_{w}$ is a connected component of 
${\goth B}_{s}$.

{\rm (ii)} For $(w,w')$ in $W({\cal R})\times W({\cal R})$, $C_{w}=C_{w'}$ if and only
if $w'w^{-1}$ is in $W({\cal R}_{s})$.

{\rm (iii)} The variety $C_{w}$ is isomorphic to $G^{s}/(G^{s}\cap B^{w})$.
\end{lemma}

For $x$ in ${\cal B}^{(k)}$, denote by ${\goth B}_{x}$ the subset of Borel 
subalgebras containing $E_{x}$.

\begin{coro}\label{cisc2}
Let $x=(\poi x1{,\ldots,}{k}{}{}{})$ be in ${\cal C}^{(k)}$. Suppose that $x$ 
verifies Conditions {\rm (1)} and {\rm (2)} of Lemma~{\rm \ref{lisc2}}. Then  
$\{C_{w}\cap {\goth B}_{x}\; \vert \; w \in W({\cal R})\}$ is the set of connected 
components of ${\goth B}_{x}$.
\end{coro}

\begin{proof}
Since a Borel subalgebra contains the semisimple component of its elements and since $s$
is the semisimple component of $x_{1}$, ${\goth B}_{x}$ is contained in ${\goth B}_{s}$.
As a result, according to Lemma~\ref{l2isc2},(i), every connected component of 
${\goth B}_{x}$ is contained in $C_{w}$ for some $w$ in $W({\cal R})$. Set
$x_{\n} := (\poi x{1,\n}{,\ldots,}{k,\n}{}{}{})$. Since $[x_{i},x_{j}]=0$ for all 
$(i,j)$, $E_{x}$ is contained in ${\goth g}^{s}$. Let ${\goth B}^{s}$ be the set 
of Borel subalgebras of ${\goth g}^{s}$ and for $y$ in $({\goth g}^{s})^{k}$, let 
${\goth B}^{s}_{y}$ be the set of Borel subalgebras of ${\goth g}^{s}$ containing 
$E_{y}$. According to~\cite[Theorem 6.5]{Hu}, ${\goth B}^{s}_{x_{\n}}$ is connected.
Moreover, according to Lemma~\ref{lisc2}, the semisimple components 
of $\poi x1{,\ldots,}{k}{}{}{}$ are in ${\goth z}_{s}$ so that 
${\goth B}^{s}_{x_{\n}}={\goth B}^{s}_{x}$. Let $w$ be in $W({\cal R})$. 
According to Lemma~\ref{l2isc2},(iii), there is an isomorphism from ${\goth B}^{s}$ to 
$C_{w}$. Moreover, the image of ${\goth B}^{s}_{x}$ by this isomorphism is equal to
$C_{w}\cap {\goth B}_{x}$, whence the corollary.  
\end{proof}

\begin{coro}\label{c2isc2}
Let $x=(\poi x1{,\ldots,}{k}{}{}{})$ be in ${\cal C}^{(k)}$ verifying Conditions 
{\rm (1)} and {\rm (2)} of Lemma~{\rm \ref{lisc2}}. Then $\eta ^{-1}(x)$ is contained 
in $\{(\poi x1{,\ldots,}{k}{}{}{},\poi x{1,\s}{,\ldots,}{k,\s}{w}{}{}) \ \vert \ 
w \in W({\cal R}) \}$.
\end{coro}

\begin{proof}
Since $\gamma = \eta \rond \gamma _{\x}$, $\eta ^{-1}(x)$ is the image of 
$\gamma ^{-1}(x)$ by $\gamma _{\x}$. Furthermore, $\gamma _{\x}$ is constant 
on the connected components of $\gamma ^{-1}(x)$ since $\eta ^{-1}(x)$ is finite. Let
$C$ be a connected component of $\gamma ^{-1}(x)$. Identifying $\sqx G{{\goth b}^{k}}$ 
with the subvariety of elements $(u,x)$ of ${\goth B}\times {\goth g}^{k}$ such that 
$E_{x}$ is contained in $u$, $C$ identifies with $(C_{w}\cap {\goth B}_{x})\times \{x\}$ 
for some $w$ in $W({\cal R})$ by Corollary~\ref{cisc2}. Then for some $g$ in $G^{s}$ and 
for some representative $g_{w}$ of $w$ in $N_{G}({\goth h})$,  
$gg_{w}({\goth b})$ contains $E_{x}$ so that 
$$\gamma _{\x}(C)=
\{(\poi x1{,\ldots,}{k}{}{}{},\overline{(gg_{w})^{-1}(x_{1})},\ldots,
\overline{(gg_{w})^{-1}(x_{k})})\} .$$
By Lemma~\ref{lisc2}, $\poi x{1,\s}{,\ldots,}{k,\s}{}{}{}$ are in ${\goth z}_{s}$ so that
$w^{-1}(x_{i,\s})$ is the semisimple component of $(gg_{w})^{-1}(x_{i})$ for 
$i=1,\ldots,k$. Hence 
$$\gamma _{\x}(C)=\{(\poi x1{,\ldots,}{k}{}{}{},
\poi x{1,\s}{,\ldots,}{k,\s}{w^{-1}}{}{})\},$$ 
whence the corollary.
\end{proof}

\begin{prop}\label{pisc2}
The variety ${\cal C}_{\x}^{(k)}$ is irreducible and equal to the closure of 
$G.\iota _{k}({\goth h}^{k})$ in ${\cal B}_{\x}^{(k)}$.
\end{prop}

\begin{proof}
Denote by $\overline{G.\iota _{k}({\goth h}^{k})}$ the closure of 
$G.\iota _{k}({\goth h}^{k})$ in ${\cal B}_{\x}^{(k)}$. Then 
$\overline{G.\iota _{k}({\goth h}^{k})}$ is irreducible as the closure of an 
irreducible set. Since $\eta $ is $G$-equivariant, 
$\eta (G.\iota _{k}({\goth h}^{k}))=G.{\goth h}^{k}$. Hence 
$\eta (\overline{G.\iota _{k}({\goth h}^{k})})={\cal C}^{(k)}$ since $\eta $ is 
a finite morphism and ${\cal C}^{(k)}$ is the closure of $G.{\goth h}^{k}$ in 
${\goth g}^{k}$ by definition.  So, it remains to prove that for all 
$x$ in ${\cal C}^{(k)}$, $\eta ^{-1}(x)$ is contained in 
$\overline{G.\iota _{k}({\goth h}^{k})}$. There is a canonical action of 
${\mathrm {GL}}_{k}(\k)$ on ${\goth g}^{k}$ and ${\cal X}^{k}$. Since this action 
commutes with the action of $G$ in ${\cal X}^{k}$, ${\cal B}_{\x}^{(k)}$ is invariant 
under ${\mathrm {GL}}_{k}(\k)$ and $\eta $ is ${\mathrm {GL}}_{k}(\k)$-equivariant.
As a result, since ${\cal C}^{(k)}$ and $G.\iota _{k}({\goth h}^{k})$ are invariant 
under ${\mathrm {GL}}_{k}(\k)$, for $x$ in ${\cal C}^{(k)}$, $\eta ^{-1}(x')$ is 
contained in $\overline{G.\iota _{k}({\goth h}^{k})}$ for all $x'$ in $E_{x}^{k}$ such
that $E_{x'}=E_{x}$ if $\eta ^{-1}(x)$ is contained in 
$\overline{G.\iota _{k}({\goth h}^{k})}$. Then, according to Lemma~\ref{lisc2}, since 
$\eta $ is $G$-equivariant, it suffices to prove that $\eta ^{-1}(x)$ is contained in 
$\overline{G.\iota _{k}({\goth h}^{k})}$ for $x$ in ${\cal C}^{(k)}\cap {\goth b}^{k}$ 
verifying Conditions (1) and (2) of Lemma~\ref{lisc2} for some $s$ in ${\goth h}$. 

According to Corollary~\ref{c2isc2}, 
$$ \eta ^{-1}(x) \subset 
\{(\poi x1{,\ldots,}{k}{}{}{},\poi x{1,\s}{,\ldots,}{k,\s}{w}{}{}) \ 
\vert \ w \in W({\cal R})\} \mbox{ with } x = (\poi x1{,\ldots,}{k}{}{}{}).$$ 
For $s$  regular, $E_{x}$ is contained in ${\goth h}$ and $x_{i}=x_{i,\s}$ for 
$i=1,\ldots,k$. By definition, 
$$(\poi x1{,\ldots,}{k}{w}{}{},\poi x1{,\ldots,}{k}{w}{}{}) \in 
\iota _{k}({\goth h}^{k})$$ 
and for $g_{w}$ a representative of $w$ in $N_{G}({\goth h})$, 
$$ g_{w}^{-1}.(\poi x1{,\ldots,}{k}{w}{}{},\poi x1{,\ldots,}{k}{w}{}{}) = 
(\poi x1{,\ldots,}{k}{}{}{},\poi x1{,\ldots,}{k}{w}{}{}).$$
Hence $\eta ^{-1}(x)$ is contained in $G.\iota _{k}({\goth h}^{k})$. As a result, 
according to the notations of Lemma~\ref{lisc1}, for all $w$ in $W({\cal R})$, 
$\theta _{w}({\cal E}_{0}^{(k,*)})$ is contained in $G.\iota _{k}({\goth h}^{k})$. Hence, by 
Lemma~\ref{lisc1},(ii), $\theta _{w}({\cal E}_{0}^{(k)})$ is contained in 
$\overline{G.\iota _{k}({\goth h}^{k})}$, whence the proposition. 
\end{proof}

\subsection{} \label{isc3}
Let $\varpi $ be the canonical projection from ${\cal X}^{k}$ to ${\goth g}^{k}$. 
By Corollary~\ref{cbo2},(ii), ${\cal B}_{\x}^{(k)}$ is an irreducible
component of $\varpi ^{-1}({\cal B}^{(k)})$ and the action of $W({\cal R})^{k}$ on 
${\cal X}^{k}$ induces a simply transitive action on the set of irreducible components
of $\varpi ^{-1}({\cal B}^{(k)})$. According to Remark~\ref{rbo5}, there is an 
embedding $\Phi $ of $\e Sh^{\tens k}$ into $\k[{\cal B}_{\x}^{(k)}]$ given by
$$ p \longmapsto ((\poi x1{,\ldots,}{k}{}{}{},\poi y1{,\ldots,}{k}{}{}{})
\mapsto p(\poi y1{,\ldots,}{k}{}{}{})) .$$
By Corollary~\ref{c3bo5},(i), this embedding identifies $\e Sh^{\tens k}$ with 
$\k[{\cal B}_{\x}^{(k)}]^{G}$.

\begin{lemma}\label{lisc3}
Let $\Psi $ be the restriction to $\e Sh^{\tens k}$ of the canonical map from 
$\k[{\cal B}_{\x}^{(k)}]$ to $\k[{\cal C}_{\x}^{(k)}]$. 

{\rm (i)} The subvariety ${\cal C}_{\x}^{(k)}$ of ${\cal X}^{k}$ is invariant under the 
diagonal action of $W({\cal R})$ in ${\cal X}^{k}$.

{\rm (ii)} The map $\Psi $ is an embedding of $\e Sh^{\tens k}$ into 
$\k[{\cal C}_{\x}^{(k)}]$. Moreover, $\Psi (\e Sh^{\tens k})$ is equal to 
$\k[{\cal C}_{\x}^{(k)}]^{G}$.

{\rm (iii)} The image of $(\e Sh^{\tens k})^{W({\cal R})}$ by $\Psi $ equals 
$\k[{\cal C}^{(k)}]^{G}$.
\end{lemma}

\begin{proof}
(i) For $x$ in ${\cal B}_{\x}^{(k)}$ and $w$ in $W({\cal R})$, 
$\eta (x)=\eta (w.x)$, whence the assertion by Proposition~\ref{pisc2}. 

(ii) For $P$ in $\e Sh^{\tens k}$, $P=0$ if $P(x)=0$ for all $x$ in 
$\iota _{k}({\goth h}^{k})$. Hence $\Psi $ is injective. Since $G$ is reductive, 
$\k[{\cal C}_{\x}^{(k)}]^{G}$ is the image of $\k[{\cal B}_{\x}^{(k)}]^{G}$ by the 
quotient morphism, whence the assertion.

(iii) Since $G$ is reductive, $\k[{\cal C}^{(k)}]^{G}$ is the image of 
$\k[{\cal B}^{(k)}]^{G}$ by the quotient morphism, whence the assertion since 
$(\e Sh^{\tens k})^{W({\cal R})}$ is equal to $\k[{\cal B}^{(k)}]^{G}$ by
Corollary~\ref{c3bo5},(iii).
\end{proof}

Identify $\e Sh^{\tens k}$ with $\k[{\cal C}_{\x}^{(k)}]^{G}$ by $\Psi $. 

\begin{prop}\label{pisc3}
Let $\widetilde{{\cal C}_{\x}^{(k)}}$ and $\widetilde{{\cal C}^{(k)}}$ be the 
normalizations of ${\cal C}_{\x}^{(k)}$ and ${\cal C}^{(k)}$.

{\rm (i)} The variety ${\cal C}^{(k)}$ is the categorical quotient of 
${\cal C}_{\x}^{(k)}$ under the action of $W({\cal R})$.

{\rm (ii)} The variety $\widetilde{{\cal C}^{(k)}}$ is the categorical quotient of 
$\widetilde{{\cal C}_{\x}^{(k)}}$ under the action of $W({\cal R})$.
\end{prop}

\begin{proof}
(i) According to Corollary~\ref{c3bo5},(i), $\k[{\cal B}_{\x}^{(k)}]$ is generated 
by $\k[{\cal B}^{(k)}]$ and $\e Sh^{\tens k}$. Since 
${\cal C}_{\x}^{(k)}=\eta ^{-1}({\cal C}^{(k)})$ by Proposition~\ref{pisc2}, the image 
of $\k[{\cal B}^{(k)}]$ in $\k[{\cal C}_{\x}^{(k)}]$ by the quotient morphism is equal
to $\k[{\cal C}^{(k)}]$. Hence $\k[{\cal C}_{\x}^{(k)}]$ is generated by 
$\k[{\cal C}^{(k)}]$ and $\e Sh^{\tens k}$. Then, by Lemma~\ref{lisc3},(iii), 
\sloppy \hbox{$\k[{\cal C}_{\x}^{(k)}]^{W({\cal R})}=\k[{\cal C}^{(k)}]$}.

(ii) Let $K$ be the fraction field of $\k[{\cal C}_{\x}^{(k)}]$. Since 
${\cal C}_{\x}^{(k)}$ is a $W({\cal R})$-variety, there is an action of $W({\cal R})$ in 
$K$ and $K^{W({\cal R})}$ is the fraction field of 
$\k[{\cal C}_{\x}^{(k)}]^{W({\cal R})}$ since $W({\cal R})$ is finite. As a result, the 
integral closure $\k[\widetilde{{\cal C}_{\x}^{(k)}}]$ of $\k[{\cal C}_{\x}^{(k)}]$ in 
$K$ is invariant under $W({\cal R})$ and $\k[\widetilde{{\cal C}^{(k)}}]$ is contained in 
$\k[\widetilde{{\cal C}_{\x}^{(k)}}]^{W({\cal R})}$ by (i). Let $a$ be in 
$\k[\widetilde{{\cal C}_{\x}^{(k)}}]^{W({\cal R})}$. Then $a$ verifies a 
dependence integral equation over $\k[{\cal C}_{\x}^{(k)}]$,
$$ a^{m} + a_{m-1}a^{m-1} + \cdots + a_{0} = 0$$   
whence
$$ a^{m} + (\frac{1}{\vert W({\cal R}) \vert} \sum_{w\in W({\cal R})} w.a_{m-1})a^{m-1} +
\cdots + \frac{1}{\vert W({\cal R}) \vert} \sum_{w\in W({\cal R})} w.a_{0} = 0$$   
since $a$ is invariant under $W({\cal R})$ so that $a$ is in 
$\k[\widetilde{{\cal C}^{(k)}}]$ by (i), whence the assertion.
\end{proof}

\section{Desingularization} \label{ds}
Let $k\geq 2$ be an integer. Let $X$, $X'$ be as in Subsection~\ref{mv5}. Denote
by $X_{\n}$ the normalization of $X$ and by $\thetaup _{0}$ the normalization 
morphism. According to Proposition~\ref{pmv5}, $X'$ identifies with a smooth big open 
subset of $X_{\n}$ and according to~\cite{Hi}, there exists a desingularization 
$(\Gamma ,\pi _{\n})$ of $X_{\n}$ in the category of $B$-varieties such that the 
restriction of $\pi _{\n}$ to $\pi _{\n}^{-1}(X')$ is an isomorphism onto $X'$. Set 
$\pi = \thetaup _{0}\rond \pi _{\n}$ so that $(\Gamma ,\pi )$ is a desingularization of 
$X$ in the category of $B$-varieties. Recall that ${\cal E}_{0}$ is the restriction 
to $X$ of the tautological vector bundle over $\ec {Gr}g{}{}{\rg}$ and ${\goth X}_{0,k}$ 
is the closure in ${\goth b}^{k}$ of $B.{\goth h}^{k}$. Set 
${\goth X}_{k} := \sqx G{{\goth X}_{0,k}}$. Then ${\goth X}_{k}$ is a closed subvariety 
of $\sqx G{{\goth b}^{k}}$. 

\begin{lemma}\label{lds}
Let $\tau '$ be the canonical morphism from ${\cal E}_{0}$ to ${\goth b}$. 

{\rm (i)} The morphism $\tau '$ is projective and birational.

{\rm (ii)} Let $\nu $ be the canonical map from $\pi ^{*}({\cal E}_{0})$ to 
${\cal E}_{0}$. Then $\nu $ and $\tau := \tau '\rond \nu $ are $B$-equivariant birational
projective morphisms from $\pi ^{*}({\cal E}_{0})$ to ${\cal E}_{0}$ and ${\goth b}$ 
respectively. In particular, $\pi ^{*}({\cal E}_{0})$ is a desingularization of 
${\cal E}_{0}$ and ${\goth b}$.
\end{lemma}

\begin{proof}
(i) Since $X$ is a projective variety, $\tau '$ is a projective morphism and 
$\tau '({\cal E}_{0})$ is closed in ${\goth b}$. Moreover, $\tau '({\cal E}_{0})$ is 
$B$-invariant since $\tau '$ is a $B$-equivariant morphism and it contains ${\goth h}$ 
since ${\goth h}$ is in $X$. For $x$ in ${\goth h}_{\r}$, 
$(\tau ')^{-1}(x)=\{({\goth h},x)\}$. Hence $\tau '$ is a birational morphism and 
$\tau '({\cal E}_{0})={\goth b}$ since $B({\goth h}_{\r})$ is an open subset 
of ${\goth b}$. 

(ii) Since ${\cal E}_{0}$ is a vector bundle over $X$ and since $\pi $ is a projective 
birational morphism, $\nu $ is a projective birational morphism. Then $\tau $ is a 
projective birational morphism from $\pi ^{*}({\cal E}_{0})$ to ${\goth b}$ by (i). It is 
$B$-equivariant since so are $\nu $ and $\tau '$. Moreover, $\pi ^{*}({\cal E}_{0})$ is a 
desingularization of ${\cal E}_{0}$ and ${\goth b}$ since $\pi ^{*}({\cal E}_{0})$ is 
smooth as a vector bundle over a smooth variety.
\end{proof}

Denote by $\psi $ the canonical projection from $\pi ^{*}({\cal E}_{0})$ to $\Gamma $. 
Then, according to the above notations, we have the commutative diagram:
$$\xymatrix{ & \pi ^{*}({\cal E}_{0}) \ar[ld]_{\tau }\ar[r]^{\psi } \ar[d]_{\nu } & 
\Gamma  \ar[d]^{\pi } \\ {\goth b} & \ar[l]^{\tau '} {\cal E}_{0} \ar[r] & X }$$
Recall that ${\cal E}_{0}^{(k)}$ is the subvariety of $X\times {\goth b}^{k}$:
$$ {\cal E}_{0}^{(k)} :=  \{(u,\poi x1{,\ldots,}{k}{}{}{}) \in X\times {\goth b}^{k} 
\ \vert \ \poi {u \ni x}1{,\ldots,}{k}{}{}{}\} .$$
As ${\cal E}_{0}$ is a vector bundle over $X$, so is ${\cal E}_{0}^{(k)}$.

\begin{lemma}\label{l2ds}
Set ${\cal E}_{\s}^{(k)} := \pi ^{*}({\cal E}_{0}^{(k)})$. Let $\tau _{k}$ be the 
canonical morphism from ${\cal E}_{\s}^{(k)}$ to ${\goth b}^{k}$.

{\rm (i)} The vector bundle ${\cal E}_{\s}^{(k)}$ over $\Gamma $ is a vector subbundle of 
the trivial bundle $\Gamma \times {\goth b}^{k}$. Moreover, ${\cal E}_{\s}^{(k)}$ has 
dimension $k\rg + n$.

{\rm (ii)} The morphism $\tau _{k}$ is a projective birational morphism from 
${\cal E}_{\s}^{(k)}$ onto ${\goth X}_{0,k}$. Moreover, ${\cal E}_{\s}^{(k)}$ is a 
desingularization of ${\goth X}_{0,k}$ in the category of $B$-varieties.
\end{lemma}

\begin{proof}
(i) By definition, ${\cal E}_{\s}^{(k)}$ is the subvariety of 
$\Gamma \times {\goth b}^{k}$. Since $X$ and $\Gamma $ have dimension 
$n$, ${\cal E}_{\s}^{(k)}$ has dimension $k\rg +n$ as a vector bundle of rank $k\rg$ over 
$\Gamma $. 

(ii) Since $\Gamma $ is a projective variety, $\tau _{k}$ is a projective morphism and 
$\tau _{k}({\cal E}_{\s}^{(k)})={\goth X}_{0,k}$ by Lemma~\ref{lisc1},(i). For 
$(\poi x1{,\ldots,}{k}{}{}{})$ in ${\goth b}_{\rs}^{k}\cap {\goth X}_{0,k}$, 
$\tau _{k}^{-1}(\poi x1{,\ldots,}{k}{}{}{})=
\{(\pi ^{-1}({\goth g}^{x_{1}}),(\poi x1{,\ldots,}{k}{}{}{}))\}$ since 
${\goth g}^{x_{1}}$ is a Cartan subalgebra. Hence $\tau _{k}$ is a birational morphism,
whence the assertion since ${\cal E}_{\s}^{(k)}$ is a smooth $B$-variety as a vector 
bundle over the smooth $B$-variety $\Gamma $.
\end{proof}

Set ${\goth Y} := \sqx G{(\Gamma \times {\goth b}^{k})}$. The canonical projections 
from $G\times \Gamma \times {\goth b}^{k}$ to $G\times \Gamma $ and 
$G\times {\goth b}^{k}$ define through the quotients morphisms from 
${\goth Y}$ to $\sqx G{\Gamma }$ and $\sqx G{{\goth b}^{k}}$. Denote by $\varsigma $ 
and $\zeta $ these morphisms. Then we have the following diagram:
$$ \xymatrix{ {\goth Y} \ar[rr]^{\zeta }\ar[d]_{\varsigma } && 
\sqx G{{\goth b}^{k}}\ar[d]^{\gamma _{\x}} \\ \sqx G{\Gamma } && {\cal B}_{\x}^{(k)}}$$
The map $(g,x)\mapsto (g,\tau _{k}(x))$ from $G\times {\cal E}_{\s}^{(k)}$ to 
$G\times {\goth b}^{k}$ defines through the quotient a morphism 
$\overline{\tau _{k}}$ from $\sqx G{{\cal E}_{\s}^{(k)}}$ to ${\goth X}_{k}$.

\begin{prop}\label{pds}
Set $\xiup  := \gamma _{\x}\rond \overline{\tau _{k}}$.

{\rm (i)} The variety $\sqx G{{\cal E}_{\s}^{(k)}}$ is a closed subvariety of 
${\goth Y}$. 

{\rm (ii)} The variety $\sqx G{{\cal E}_{\s}^{(k)}}$ is a vector bundle of rank $k\rg$ 
over $\sqx G{\Gamma }{}$. Moreover, $\sqx G{\Gamma }{}$ and 
$\sqx G{{\cal E}_{\s}^{(k)}}{}$ are smooth varieties.

{\rm (iii)} The morphism $\xiup $ is a projective birational morphism from 
$\sqx G{{\cal E}_{\s}^{(k)}}$ onto ${\cal C}_{\x}^{(k)}$. 
\end{prop}

\begin{proof}
(i) According to Lemma~\ref{l2ds},(i), ${\cal E}_{\s}^{(k)}$ is a closed subvariety of 
$\Gamma \times {\goth b}^{k}$, invariant under the diagonal action of $B$. Hence 
$G\times {\cal E}_{\s}^{(k)}$ is a closed subvariety of 
$G\times \Gamma \times {\goth b}^{k}$, invariant under the action of $B$, whence the 
assertion.

(ii) Since ${\cal E}_{\s}^{(k)}$ is a $B$-equivariant vector bundle over $\Gamma $, 
$\sqx G{{\cal E}_{\s}^{(k)}}{}$ is a $G$-equivariant vector bundle over 
$\sqx G{\Gamma }{}$. Since $\sqx G{\Gamma }{}$ is a fiber bundle over the smooth variety 
$G/B$ with smooth fibers, $\sqx G{\Gamma }{}$ is a smooth variety. As a result, 
$\sqx G{{\cal E}_{\s}^{(k)}}{}$ is a smooth variety.  

(iii) According to Lemma~\ref{l2ds},(ii) and Lemma~\ref{l4int}, $\overline{\tau _{k}}$ is
a projective birational morphism from $\sqx G{{\cal E}_{\s}^{(k)}}{}$ to ${\goth X}_{k}$.
Since ${\goth X}_{0,k}$ is a $B$-invariant closed subvariety of ${\goth b}^{k}$, 
${\goth X}_{k}$ is closed in $\sqx G{{\goth b}^{k}}$. According to 
Lemma~\ref{lisc1},(i), $\gamma ({\goth X}_{k})={\cal C}^{(k)}$. Moreover,
$\gamma _{\x}({\goth X}_{k})$ is an irreducible closed subvariety of 
${\cal B}_{\x}^{(k)}$ since $\gamma _{\x}$ is a projective morphism by Lemma~\ref{l4int}.
Hence $\gamma _{\x}({\goth X}_{k})={\cal C}_{\x}^{(k)}$ by Proposition~\ref{pisc2}.
For all $z$ in $G.\iota _{k}({\goth h}_{\r}^{k})$, 
$\vert \gamma _{\x}^{-1}(z) \vert =1$. Hence the restriction of $\gamma _{\x}$ to 
${\goth X}_{k}$ is a birational morphism onto ${\cal C}_{\x}^{(k)}$ since 
$G.\iota _{k}({\goth h}_{\r}^{k})$ is dense in ${\cal C}_{\x}^{(k)}$. Moreover, this 
morphism is projective since $\gamma _{\x}$ is projective. As a result, $\xiup $ is a
projective birational morphism from $\sqx G{{\cal E}_{\s}^{(k)}}{}$ onto 
${\cal C}_{\x}^{(k)}$. 
\end{proof}
 
Theorem~\ref{t3int} results from Proposition~\ref{pisc2} and Proposition~\ref{pds},(ii)
and (iii) and the following corollary results from Lemma~\ref{l2ds},(ii), 
Proposition~\ref{pds},(ii) and (iii), and Lemma~\ref{lint}.

\begin{coro}\label{cds}
Let $\widetilde{{\goth X}_{0,k}}$ and $\widetilde{{\cal C}_{\x}^{(k)}}$ be the 
normalizations of ${\goth X}_{0,k}$ and ${\cal C}_{\x}^{(k)}$ respectively. Then 
$\k[\widetilde{{\goth X}_{0,k}}]$ and $\k[\widetilde{{\cal C}_{\x}^{(k)}}]$ are the 
spaces of global sections of $\an {{\cal E}_{\s}^{(k)}}{}$ and 
$\an {\sqx G{{\cal E}_{\s}^{(k)}}}{}$ respectively.
\end{coro}

\end{document}